\documentclass[11pt]{amsart}
\usepackage[utf8]{inputenc}
\usepackage{amssymb, amsthm}
\usepackage[leqno]{amsmath}
    \usepackage[foot]{amsaddr}
\pdfoutput=1 
\usepackage{enumerate}

\usepackage{prettyref,cite}
\usepackage[colorlinks,linkcolor=blue,citecolor=red]{hyperref}

\usepackage [a4paper, footskip=1cm, headheight = 16pt, top=3.7cm, bottom=3cm,  right=2.5cm,  left=2.5cm ]{geometry}
\newtheorem{theorem}{Theorem}
\newtheorem{lemma}[theorem]{Lemma}
\newtheorem{cor}[theorem]{Corollary}
\newtheorem{remark}[theorem]{Remark}
\mathchardef\mh="2D
\usepackage[center]{caption}

\usepackage{graphicx}
\usepackage{xcolor}
\usepackage{tikz}
\usepackage[subpreambles=true]{standalone}
\usepackage{import}

\newtheorem{corollary}[theorem]{Corollary}

\newcounter{tbox}
\newcommand{\tbox}[1]{\vspace*{0.3cm}\refstepcounter{tbox}\noindent{ \parbox{\textwidth}{(\thetbox) \emph{#1}}}\vspace*{0.3cm}}

\DeclareMathOperator{\und}{\textsf{Under}}

\DeclareMathOperator{\lft}{\textsf{Left}}

\DeclareMathOperator{\mx}{\textsf{MAX}}

\newcommand{\sta}[1]{\refstepcounter{tbox}\noindent{ \parbox{\textwidth}{\vspace*{0.3cm}(\thetbox) \emph{#1}\vspace*{0.3cm}}}}

\newcommand{\vsp}{\vspace*{3mm}}

\makeatletter
\newcommand{\otherlabel}[2]{\protected@edef\@currentlabel{#2}\label{#1}}
\makeatother

    \title[List-3-Coloring Ordered graphs]{List-3-Coloring Ordered graphs with a forbidden induced subgraph}
\author{Sepehr Hajebi$^{\ast}$}
\email{shajebi@uwaterloo.ca}
\author{Yanjia Li$^{\ast}$}
\email{yanjia.li@uwaterloo.ca}
\author{Sophie Spirkl$^{\ast \dagger}$}
\email{sspirkl@uwaterloo.ca}
\address{$^{\ast}$Department of Combinatorics and Optimization, University of Waterloo, Waterloo, Ontario N2L3G1, Canada}
\address{$^{\dagger}$We acknowledge the support of the Natural Sciences and Engineering Research Council of Canada (NSERC), [funding reference number RGPIN-2020-03912]. Cette recherche a été financée par le Conseil de recherches en sciences naturelles et en génie du Canada (CRSNG), [numéro de référence RGPIN-2020-03912]. This project was funded in part by the Government of Ontario.}

\thanks{This is an accepted manuscript. The published version appeared in SIAM Journal on Discrete Mathematics Vol. 38, Iss. 1, see \url{https://doi.org/10.1137/22M1515768}.}

\date{\today}

\allowdisplaybreaks

\begin{document}
\maketitle
\raggedbottom

\begin{abstract} The \textsc{List-$3$-Coloring Problem} is to decide, given a graph $G$ and a list $L(v)\subseteq \{1,2,3\}$ of colors assigned to each vertex $v$ of $G$, whether $G$ admits a proper coloring $\phi$ with $\phi(v)\in L(v)$ for every vertex $v$ of $G$, and the \textsc{$3$-Coloring Problem} is the \textsc{List-$3$-Coloring Problem} on instances with $L(v)=\{1,2,3\}$ for every vertex $v$ of $G$. The \textsc{List-$3$-Coloring Problem} is a classical \textsf{NP}-complete problem, and it is well-known that while restricted to \textit{$H$-free} graphs (meaning graphs with no induced subgraph isomorphic to a fixed graph $H$), it remains \textsf{NP}-complete unless $H$ is isomorphic to an induced subgraph of a path. However, the current state of art is far from proving this to be sufficient for a polynomial time algorithm; in fact, the complexity of the \textsc{$3$-Coloring Problem} on $P_8$-free graphs (where $P_8$ denotes the eight-vertex path) is unknown. Here we consider a variant of the \textsc{List-$3$-Coloring Problem} called the \textsc{Ordered Graph List-$3$-Coloring Problem}, where the input is an \textit{ordered graph}, that is, a graph along with a linear order on its vertex set. For ordered graphs $G$ and $H$, we say $G$ is \textit{$H$-free} if $H$ is not isomorphic to an induced subgraph of $G$ with the isomorphism  preserving the linear order. We prove, assuming $H$ to be an ordered graph, a nearly complete dichotomy for the \textsc{Ordered Graph List-$3$-Coloring Problem} restricted to $H$-free ordered graphs. In particular, we show that the problem can be solved in polynomial time if $H$ has at most one edge, and remains \textsf{NP}-complete if $H$ has at least three edges. Moreover, in the case where $H$ has exactly two edges, we give a complete dichotomy when the two edges of $H$ share an end, and prove several \textsf{NP}-completeness results when the two edges of $H$ do not share an end, narrowing the open cases down to three very special types of two-edge ordered graphs.
\end{abstract}
	
\section{Introduction}
Graphs in this paper are finite and simple.  We denote the set of positive integers by $\mathbb{N}$, and for every integer $k\in \mathbb{N}$, we denote by $[k]$ the set of all positive integers which are smaller than or equal to $k$. Let $G=(V(G),E(G))$ be a graph. For every $X$ which is either a vertex or a subset of vertices of $G$, we write $G\setminus X$ for the graph obtained from $G$ by removing $X$. Also, we denote by $G[X]$ \textit{the subgraph of $G$ induced by $X$}, that is, $G\setminus (V(G)\setminus X)=(X,\{e\in E(G): e\subseteq X\})$.  For graphs $G$ and $H$, we say $H$ is an \emph{induced subgraph} of $G$ if $H$ is isomorphic to $G[X]$ for some $X\subseteq V(G)$, and otherwise we say $G$ is \emph{$H$-free}. For all $t\in \mathbb{N}$, we use $P_t$ to denote the path on $t$ vertices. 

Let $G$ be a graph and $k\in \mathbb{N}$. By a $k$\textit{-coloring} of $G$, we mean a function $\phi: V(G) \rightarrow [k]$. A coloring $\phi$ of $G$ is said to be \textit{proper} if $\phi(u) \neq \phi(v)$ for every edge $uv \in E(G)$. In other words, $\phi$ is proper if and only if for every $i\in [k]$, $\phi^{-1}(i)$ is a stable set in $G$. We say $G$ is \emph{$k$-colorable} if $G$ has a proper $k$-coloring. For fixed $k\in \mathbb{N}$, the \textsc{$k$-Coloring Problem} asks, given graph $G$, whether $G$ is $k$-colorable. 
 
A \emph{$k$-list-assignment} of $G$ is a map $L:V(G)\rightarrow 2^{[k]}$. For $v\in V(G)$, we refer to $L(v)$ as the \textit{list of} $v$. Also, for every $i\in [k]$, we define $L^{(i)}=\{v\in V(G):i\in L(v)\}$. An \emph{$L$-coloring} of $G$ is a proper $k$-coloring $\phi$ of $G$ with $\phi(v) \in L(v)$ for all $v \in V(G)$. We say $G$ is \textit{$L$-colorable} if it admits an $L$-coloring. For example, if $L(v)=\emptyset$ for some $v\in V(G)$, then $G$ admits no $L$-coloring. Also, if $V(G)=\emptyset$, then $G$ vacuously admits an $L$-coloring for every $k$-list-assignment $L$. For fixed $k\in \mathbb{N}$, the \textsc{List-$k$-Coloring Problem} is to decide, given an instance $(G,L)$ consisting of a graph $G$ and a $k$-list-assignment $L$ of $G$, whether $G$ is $L$-colorable. Note that the \textit{$k$-coloring problem} is in fact the \textsc{List-$k$-Coloring Problem} restricted to instances $(G,L)$ where $L(v)=[k]$ for every $v\in V(G)$. 

The \textsc{List-$k$-coloring Problem} is famously known to be \textsf{NP}-complete for all $k\geq 3$ \cite{karp}, and understanding the complexity of this problem while restricted to graphs with a fixed forbidden induced subgraph, that is, $H$-free graphs for some fixed graph $H$, is of enormous interest. The following two results, combined, show that the \textsf{NP}-hardness persists unless $H$ is an induced subgraph of a path.
\begin{theorem}[Kami\'{n}ski and Lozin \cite{CycleFree}, see also \cite{emenem}] \label{thm:cycle}
	For all $k\geq 3$, the \textsc{$k$-Coloring} problem restricted to $H$-free graphs is \textsf{NP}-complete if $H$ contains a cycle. 
\end{theorem}

\begin{theorem}[Holyer \cite{ClawFree}, Leven and Galil \cite{leven}] \label{thm:claw}
	For all $k\geq 3$, the \textsc{$k$-Coloring Problem} restricted to $H$-free graphs is \textsf{NP}-complete if $H$ contains a claw (a vertex with three pairwise nonadjacent neighbors). 
\end{theorem}
Of course, for each $k\geq 3$, it is most desirable to completely distinguish graphs $H$ for which the \textsc{List-$k$-coloring Problem} on $H$-free graphs is polynomial-time solvable from those for which the problem remains \textsf{NP}-complete. Despite several attempts \cite{kP5, 3P7,L3P6+rP3, LkP5+rP1&L5P4+P2-NP}, no such value of $k$ had been known until very recently, when the authors of the present paper settled the case $k=5$:
\begin{theorem}[Hajebi, Li, Spirkl \textcolor{red}{\cite{hajebi:2021}}] \label{thm:rP3us}
Assuming \textsf{P}$\neq$\textsf{NP}, the \textsc{List-$5$-Coloring Problem} restricted to $H$-free graphs can be solved in polynomial time if and only if $H$ is an induced subgraph of a graph $H'$ of one of the following two types.
\begin{itemize}
\item Each component of $H'$ is isomorphic to $P_3$.
    \item One component of $H'$ is isomorphic to $P_5$ and all other components of $H$ are isomorphic to $P_1$.
\end{itemize}
\end{theorem}
To the best of our knowledge (and curiously enough), $k=5$ is  currently the only value of $k$ for which a complete dichotomy has been found. Among open cases, $k=3$ has attracted a great deal of attention, thanks to the special place of the \textsc{$3$-Coloring Problem} in complexity theory as a fundamental \textsf{NP}-complete problem. The following theorem summarizes how successful attempts in this direction have been so far.

\begin{theorem}\label{thm:List3status}
The \textsc{List-$3$-Coloring Problem} restricted to $H$-free graphs can be solved in polynomial time if $H$ is an induced subgraph of a graph $H'$ of one of the following two types.
\begin{itemize}
\item $H'$ is isomorphic to $P_7$ \textup{(Bonomo, Chudnovsky, Maceli, Schaudt, Stein and Zhong \cite{3P7})}.
    \item One component of $H'$ is isomorphic to $P_6$ and all other components of $H'$ are isomorphic to $P_3$ \textup{(Chudnovsky, Huang, Spirkl and Zhong \cite{L3P6+rP3})}.
\end{itemize}
\end{theorem}
In particular, determining the complexity of the \textsc{$3$-Coloring Problem} on $P_t$-free graphs (for each fixed $t\geq 8$) is a central open problem of both a structural and an algorithmic flavour. On a broader scale, graphs excluding a fixed path as an induced subgraph seem to have played a significant role in the development of modern structural graph theory. Examples include the theory of $\chi$-boundedness (see \cite{chisurvey}) and the great body of work built around the notoriously difficult Erd\H{o}s-Hajnal conjecture (see, for instance, \cite{DBLP:journals/corr/abs-1207-0016},) the simplest open case of which concerns $P_5$-free graphs. A common approach to problems in these areas \cite{tomon,orderedtree,diGS}, as well as several other problems in structural and extremal and graph theory \cite{OrderIdea1,OrderIdea2,OrderIdea3,OrderIdea4}, has been to study their variants on graphs with additional specifications, such as an orientation of edges or an ordering of vertices. The goal of this paper is to look at list-$3$-coloring $H$-free graphs from the same perspective.

Let us provide formal definitions. An \emph{ordered graph} $G$ is a triple $(V,E,\varphi)$ such that $(V,E)$ is a graph with vertex set $V$ and edge set $E$, and $\varphi:V \rightarrow \mathbb{R}$ is an injective function. We say $\varphi$ is the \emph{ordering} of $G$. For an ordered graph $G=(V,E,\varphi)$, we define $V(G)=V$, $E(G)=E$, and $\varphi_G= \varphi$. 
	For convenience, we sometimes write the ordered graph $(V,E,\varphi)$ as $(G, \varphi )$ where $G=(V,E)$ is a graph. For $X\subseteq V$, we define $G[X]=(X,\{e\in E:e\subseteq X\},\varphi|_X)$, and $G\setminus X=G[V(G)\setminus X]$. We also define $-G=(V,E,v\mapsto -\varphi_G(v))$. 
	
	Given an ordered graph $G=(V,E,\varphi)$, an ordered graph $G'=(V',E',\varphi')$ is \emph{isomorphic} to $G$ if there exists a bijective function $f:V'\rightarrow V$ such that for any two vertices $v$ and $w$ in $V'$, $\varphi'(v)<\varphi'(w)$ if and only if $\varphi(f(v))<\varphi(f(w))$, and $vw\in E'$ if and only if $f(v)f(w)\in E$. We denote $G$ and $G'$ being isomorphic by $G'\cong G$. An ordered graph $H$ is an \emph{induced subgraph} of $G$ if there exists a set $X\subseteq V(G)$ such that $H\cong G[X]$; otherwise $G$ is \emph{$H$-free}.
	
	Let $G=(V,E,\varphi)$ be an ordered graph and let $k\in \mathbb{N}$. By a \emph{$k$-coloring} and a \emph{$k$-list-assignment} of $G$, we mean a $k$-coloring and a $k$-list-assignment of the graph $(V,E)$. Also, for a $k$-list-assignment $L$ of $G$, an \emph{$L$-coloring} of $G$ is an $L$-coloring of $(V,E)$. We say $G$ is \emph{$L$-colorable} if $(V,E)$ is $L$-colorable. For a fixed $k\in \mathbb{N}$, the \textsc{Ordered Graph List-$3$-Coloring Problem} is to decide, given an instance $(G,L)$ consisting of an ordered graph $G$ and a $k$-list-assignment $L$ of $G$, whether $G$ admits an $L$-coloring. 
 
 Let us emphasize that, as opposed to our notion of ``induced subgraph,'' our notion of ``coloring'' for ordered graphs ignores the ordering of vertices. In particular, the \textsc{Ordered Graph List-$3$-Coloring Problem} is \textsf{NP}-complete, as so is the usual \textsc{List-$3$-Coloring Problem}. We establish a nearly complete dichotomy for the complexity of the \textsc{Ordered Graph List-$3$-Coloring Problem} on $H$-free ordered graphs, where $H$ is a fixed ordered graph. Our first two theorems, below, give a complete dichotomy for the case $|E(H)|\neq 2$.
	\begin{theorem}\label{mainwhale}
		The \textsc{Ordered Graph List-3-Coloring Problem} restricted to $H$-free ordered graphs is polynomial-time solvable if $H$ has at most one edge.
	\end{theorem}
	\begin{theorem}\label{OrderedNPatleast3}
	The \textsc{Ordered Graph List-3-Coloring Problem} restricted to $H$-free ordered graphs is \textsf{NP}-complete if $H$ has at least three edges..
	\end{theorem}

\begin{figure}[t]
		\centering 
			\begin{tikzpicture}[scale=0.8]
		\node [label=below: $u_1$] (u1) at (2,3){};
		\node [label=below: $u_2$] (u2) at (6,3){};
		\draw (u1) edge [bend left = 30]  (u2);
		\node [color=gray]  at (0,3){};
		\node [color=gray]  at (1,3){};
		\node [color=gray]  at (3,3){};
		\node [color=gray]  at (4,3){};
		\node [color=gray]  at (5,3){};
		\node [color=gray]  at (7,3){};
		\node [color=gray]  at (8,3){};
		
		\node [color=gray]  at (0,1){};
		\node [color=gray]  at (1,1){};
		\node [color=gray]  at (8,1){};
		\node [label=below: $u_1$] (u1) at (2,1){};
		\node [label=below: $u_2$] (u2) at (5,1){};
		\node [label=below: $u_3$] (u3) at (7,1){};
		\draw (u1) edge [bend left = 30]  (u3);
		\draw (u1) edge [bend left = 25]  (u2);
	\end{tikzpicture}	
		\caption{The cases of $H$ which are polynomial-time solvable.}\label{fig-orderedpolytime}
	\end{figure}
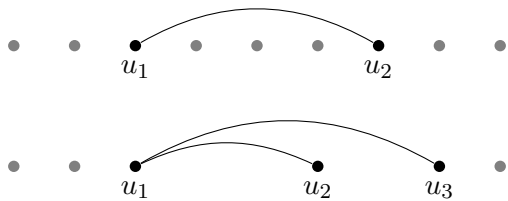
 
	

	In order to state our results concerning the case $|E(H)|=2$, we need to define some special ordered graphs, each with exactly two edges. The strange naming is due to the remaining ordered graphs that are to be defined in Section \ref{Sec-OrderedNP}.
	Let $U'=\{u_1, u_2, u_3, u_4, u_5\}$ and $U=U'\backslash\{u_5\}$, the ordering $\varphi':U'\rightarrow \mathbb{R}$ with $u_i\mapsto i$ for $i\in [5]$. Let $V=\{v_1,v_2,v_3,v_4,v_5,v_6\}$ and $\varphi:V\rightarrow \mathbb{R}$ with $v_i\mapsto i$ for $i\in [6]$.
	\begin{itemize}
		\item Let $J_9=(U,\{u_1u_2,u_3u_4\}, \varphi'|_U)$.
		\item Let $J_{16}=(U\backslash \{u_4\}, \{u_1u_2,u_1u_3\}, \varphi'|_{U\backslash \{u_4\}})$.
		\item Let $M_1=(V,\{v_1v_6,v_2v_5\},\varphi)$.
		\item Let $M_5=(V\backslash \{v_6\},\{v_1v_5,v_2v_3\},\varphi|_{V\backslash \{v_6\}})$.
		\item Let $M_6=(V\backslash \{v_5,v_6\},\{v_1v_3,v_2v_4\},\varphi|_{V\backslash \{v_5,v_6\}})$.
		\item Let $M_7=(V\backslash \{v_5,v_6\},\{v_1v_4,v_2v_3\},\varphi|_{V\backslash \{v_5,v_6\}})$.
		\item Let $M_8=(V\backslash \{v_6\},\{v_1v_5,v_2v_4\},\varphi|_{V\backslash \{v_6\}})$.		
	\end{itemize}
	
	Also, given an ordered graph $G=(V,E,\varphi)$ and $k,l\in \mathbb{N}\cup \{0\}$, we define $G(k,l)$ as the ordered graph obtained from $G$ by adding $k+l$ isolated vertices $\{a_i, b_j: i\in [k], j\in [l]\}$ to $G$ and extending $\varphi$ from $V$ to $V\cup \{a_i, b_j: i\in [k], j\in [l]\}$ by defining $\varphi(a_i)=\min_{v\in V} \varphi(v)-(k+1-i)$ for each $i\in[k]$ and $\varphi(b_i)=\max_{v\in V} \varphi(v)+i$ for each $i\in [l]$. The following theorem is a full dichotomy for the case where the two edges of $H$ share an end.
	
	\begin{theorem}\label{shareanend}
		Let $H$ be an ordered graph with $|E(H)|=2$ where the two edges of $H$ share an end. Then the \textsc{Ordered Graph List-$3$-Coloring Problem} restricted to $H$-free ordered graphs is polynomial-time solvable if $H$ is isomorphic to   $J_{16}(k,l)$ or $-J_{16}(k,l)$ for some $k,l\in \mathbb{N}\cup \{0\}$, and \textsf{NP}-complete otherwise.
	\end{theorem}
	Figure \ref{fig-orderedpolytime} depicts ordered graphs $H$ for which we prove, in Theorems \ref{mainwhale}  and \ref{shareanend}, that the \textsc{Ordered Graph List-$3$-Coloring Problem} restricted to $H$-free ordered graphs can be solved in polynomial time (gray vertices represent arbitrarily many isolated vertices.)
	
Let us now consider ordered graphs $H$ with $|E(H)|=2$ where the two edges of $H$ do not share an end. In this case, we prove the following hardness result.
	\begin{theorem}\label{OrderedNPdemo}
	Let $H$ be an ordered graph containing an induced subgraph isomorphic to $J_9$, $M_1$, $M_5$ or $-M_5$. Then the \textsc{Ordered Graph List-3-Coloring Problem} restricted to $H$-free ordered graphs is \textsf{NP}-complete.
	\end{theorem}
Note that Theorems \ref{mainwhale}, \ref{OrderedNPatleast3}, \ref{shareanend} and \ref{OrderedNPdemo}, together, determine the complexity of the \textsc{Ordered Graph List-$3$-Coloring Problem} restricted to $H$-free ordered graphs for all $H$, except for the following three cases, which remain open (see Figure \ref{fig-orderedopen} for a depiction, where gray vertices represent arbitrarily many isolated vertices.)
	\begin{itemize}
		\item $H$ has exactly two edges, and the ordered graph obtained from $H$ by removing isolated vertices is isomorphic to $M_6$.
		\item  $H$ is isomorphic to $M_7(k,l)$ for some $k,l\in \mathbb{N}\cup \{0\}$.
		\item  $H$ is isomorphic to $M_8(k,l)$ for some $k,l\in \mathbb{N}\cup \{0\}$.
	\end{itemize}
	
		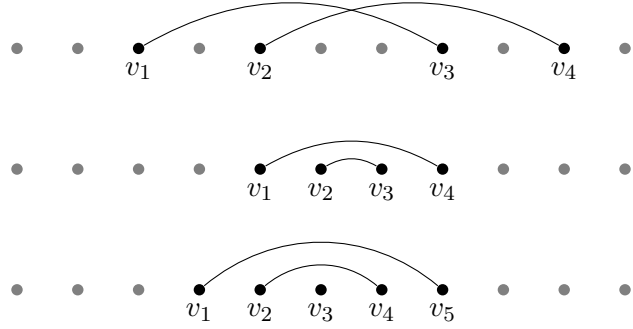
\begin{figure}[t]
		\centering 
			\begin{tikzpicture}[scale=0.8]
		\node [label=below: $v_1$] (u1) at (2,6){};
		\node [label=below: $v_2$] (u2) at (4,6){};
		\node [label=below: $v_3$] (u3) at (7,6){};
		\node [label=below: $v_4$] (u4) at (9,6){};
		\draw (u1) edge [bend left = 30]  (u3);
		\draw (u2) edge [bend left = 30]  (u4);
		\node [color=gray]  at (0,6){};
		\node [color=gray]  at (1,6){};
		\node [color=gray]  at (3,6){};
		\node [color=gray]  at (5,6){};
		\node [color=gray]  at (6,6){};
		\node [color=gray]  at (8,6){};
		\node [color=gray]  at (10,6){};
		
		\node [label=below: $v_1$] (u1) at (4,4){};
		\node [label=below: $v_2$] (u2) at (5,4){};
		\node [label=below: $v_3$] (u3) at (6,4){};
		\node [label=below: $v_4$] (u4) at (7,4){};
		\draw (u1) edge [bend left = 30]  (u4);
		\draw (u2) edge [bend left = 30]  (u3);
		\node [color=gray]  at (0,4){};
		\node [color=gray]  at (1,4){};
		\node [color=gray]  at (2,4){};
		\node [color=gray]  at (3,4){};
		\node [color=gray]  at (8,4){};
		\node [color=gray]  at (9,4){};
		\node [color=gray]  at (10,4){};
		
		\node [label=below: $v_1$] (u1) at (3,2){};
		\node [label=below: $v_2$] (u2) at (4,2){};
		\node [label=below: $v_3$] (u3) at (5,2){};
		\node [label=below: $v_4$] (u4) at (6,2){};
		\node [label=below: $v_5$] (u5) at (7,2){};
		\draw (u1) edge [bend left = 40]  (u5);
		\draw (u2) edge [bend left = 40]  (u4);
		\node [color=gray]  at (0,2){};
		\node [color=gray]  at (1,2){};
		\node [color=gray]  at (2,2){};
		\node [color=gray]  at (8,2){};
		\node [color=gray]  at (9,2){};
		\node [color=gray]  at (10,2){};
	\end{tikzpicture}	
		\caption{The cases of $H$ which are still open.}\label{fig-orderedopen}
	\end{figure}
	
	The organization of this paper is as follows. In Section \ref{sec:preliminary}, we set some notation and terminology to be used throughout. In section \ref{sec:oneedge}, we prove Theorem \ref{mainwhale}. Section \ref{Sec-H(k,l)proof} contains the polynomial-time algorithm promised in Theorem \ref{shareanend}. Finally, Section \ref{Sec-OrderedNP} is devoted to our hardness results that is, Theorems \ref{OrderedNPatleast3} and \ref{OrderedNPdemo}, and the \textsf{NP}-completeness assertion in Theorem \ref{shareanend}.

\section{Preliminaries}\label{sec:preliminary}
    	Here we introduce some notation and terminology. Let $G$ be a (ordered) graph. By a \textit{clique} in $G$ we mean a set of pairwise adjacent vertices, and a \textit{stable set} in $G$ is  a set of pairwise nonadjacent vertices. We say two disjoint sets $U,W\subseteq V(G)$ are \emph{anticomplete} if no vertex in $U$ has a neighbor in $W$. For $x,y\in \mathbb{R}\cup \{-\infty,\infty\}$ with $x<y$, we define $G(x:y]=\{v\in V:x< \varphi_G(v)\leq  y\}$, and $G[x:y), G[x:y]$ and $G(x:y)$ are defined similarly.  For every $\rho\in \mathbb{N}$ and every vertex $v\in V(G)$, we denote by $N^{\rho}_G(v)$ the set of all vertices in $G$ at distance $\rho$ from $v$, and by $N^{\rho}_G[v]$ the set of all vertices in $G$ at distance at most $\rho$ from $v$. In particular, we write $N_G(v)$ for $N_G^1(v)$, which is the set of neighbors of $v$ in $G$, and $N_G[v]$ for $N^1_G[v]=N_G(v)\cup \{v\}$. Also, the set of \emph{forward neighbors} of $v$ is defined as $N^+_G(v)=\{u\in N_G(v): \varphi(v)<\varphi(u)\}$, and the set of \emph{backward neighbors} of $v$ is $N^-_G(v)=\{u\in N_G(v): \varphi(v)>\varphi(u)\}$. Moreover, for every $X\subseteq V(G)$, we define $N^{\rho}_G[X]=\bigcup_{x\in X}N^{\rho}_G[x]$ and $N^{\rho}_G(X)=N^{\rho}_G[X]\setminus X$. Again, we write $N_G(X)$ for $N_G^1(X)$ and $N_G[X]$ for $N^1_G[X]$. Throughout, we sometimes omit the subscript $G$ from the latter notations if there is no ambiguity.
    	
    	Let $(G,L)$ be an instance of the \textsc{Ordered Graph List-3-Coloring Problem}. An instance $(G',L')$ is a $(G,L)$-\emph{refinement} if $G'$ is an induced subgraph of $G$ and for all $v\in V(G')$, $L'(v)\subseteq L(v)$. A $(G,L)$-refinement $(G',L')$ is \emph{spanning} if $G'=G$. A $(G,L)$-\emph{profile} $\mathcal{L}$ is a set of $(G,L)$-refinements. A $(G,L)$-profile is \emph{spanning} if all its elements are spanning. Two list assignments $L$ and $L'$ are \emph{equivalent} for $G$ if for every coloring $c$ of $G$, $c$ is an $L$-coloring if and only if it is an $L'$-coloring.
	
	\section{Excluding an ordered graph with at most one edge}\label{sec:oneedge}
  The goal of this section is to prove Theorem \ref{mainwhale}. For every $w\in \mathbb{N}$, let $J_w$ denote the ordered graph with $V(J_w)=\{v_1,\ldots, v_{3w+2}\}$, $E(J_w)=\{v_{w+1}v_{2w+2}\}$ and $\varphi_{J_w}(v_i)=i$ for all $i\in [3w+2]$. Note that every ordered graph with at most one edge is isomorphic to an induced subgraph of $J_w$ for some $w\in \mathbb{N}$. Therefore, in order to prove Theorem \ref{mainwhale}, it suffices to prove the following.
    
    \begin{theorem}\label{mainwhale2}
    For every fixed $w\in \mathbb{N}$, the \textsc{Ordered Graph List-3-Coloring Problem} restricted to $J_w$-free ordered graphs can be solved in polynomial time.
    \end{theorem}
    
   \sloppy The proof of Theorem  \ref{mainwhale2} is involved, but roughly speaking, the algorithm we give is based on dynamic programming on a decomposition of the input ordered graph using its ``maximal'' edges. A concrete presentation of the algorithm, though, consists of several steps. To begin with, let 
   $w\in \mathbb{N}$, $(G,L)$ be an instance of the \textsc{Ordered Graph List-3-Coloring Problem}, and $e=uv\in E(G)$. We set the notations $\und_G(e)=G[\min\{\varphi_G(u),\varphi_G(v)\}: \max\{\varphi_G(u),\varphi_G(v)\}]$ and $\lft_G(e)=G(-\infty:\min\{\varphi_G(u),\varphi_G(v)\})$. Also, we define $\Gamma_w(G,L,e)$ to be the set of all pairs $(S,\sigma)$ where $S$ is a subset of $\und_G(e)$ containing both $u$ and $v$, $\sigma:S\rightarrow [3]$ is an $L|_S$-coloring of $G[S]$, and for each $i\in [3]$, we have $|\sigma^{-1}(i)|\leq 27w^2+3$. From this definition, one may easily deduce the following.
   
   \begin{lemma}\label{computeguess}
   Given $w\in \mathbb{N}$, an instance $(G,L)$ of the \textsc{Ordered Graph List-3-Coloring Problem} and $e\in E(G)$, we have $|\Gamma_w(G,L,e)|\leq \mathcal{O}(|V(G)|^{81w^2+7})$ and one can compute $\Gamma_w(G,L,e)$ in time $\mathcal{O}(|V(G)|^{81w^2+7})$.
   \end{lemma}

  For an instance $(G,L)$ of the \textsc{Ordered Graph List-3-Coloring Problem} and $X,Y\subseteq V(G)$, an $L|_X$-coloring $\phi_1$ of $G[X]$ and an $L|_Y$-coloring $\phi_2$ of $G[Y]$ are said to be \textit{compatible} if $\phi_1|_{X\cap Y}=\phi_2|_{X\cap Y}$. Let $e=uv\in E(G)$ and $g=(S,\sigma)\in \Gamma_w(G,L,e)$. Let $X\subseteq V(G)$ and $\phi:X\rightarrow [3]$ be an $L|_X$-coloring of $G[X]$. We say $(\phi,g)$ \textit{has property} \textsf{X} if
     \vsp
     \begin{enumerate}[(X1)]
		\item\otherlabel{p1}{X1} $\phi$ and $\sigma$ are compatible, and;
        \item\otherlabel{p2}{X2} $\sigma\cup \phi$ is an $L|_{S\cup X}$-coloring of $G[S\cup X]$.
    \end{enumerate}
    \vsp
    
    Also, we say $(\phi,g)$ \textit{has property} \textsf{Y} if 
    \vsp
       \begin{enumerate}[(Y1)]
		\item\otherlabel{q1}{Y1} $\phi$ and $\sigma$ are compatible, and;
        	\item\otherlabel{q2}{Y2} for each $i\in [3]$, every $x\in \lft_G(e)$ with a neighbor in $\phi^{-1}(i)\cap \und_G(e)$ has a neighbor in $\sigma^{-1}(i)$.\vsp
    \end{enumerate}

The proof of the following lemma is straightforward and we omit it.

\begin{lemma}\label{PQprop}
Let $(G,L)$ be an instance of the \textsc{Ordered Graph List-3-Coloring Problem}, $e=uv\in E(G)$ and $g=(S,\sigma)\in \Gamma_w(G,L,e)$. Let $X\subseteq V(G)$ and $\phi:X\rightarrow [3]$ be an $L|_X$-coloring of $G[X]$. Then for every $X'\subseteq X$, if $(\phi,g)$ has property \textsf{X} (resp. \textsf{Y}), then $(\phi|_{X'},g)$ has property \textsf{X} (resp. \textsf{Y}), as well.
\end{lemma}

To prove our next lemma, we need the following classical result of Erd\H{o}s and Szekeres \cite{ErdSze}.

\begin{theorem}[Erd\H{o}s and Szekeres \cite{ErdSze}]\label{ErdosSzekeres}
    For all $n\geq 0$ and every sequence $a_1,\ldots, a_{n^2+1}$ of reals, there exists an increasing injection $\alpha:[n+1]\rightarrow [n^2+1]$ such that either $a_{\alpha(1)}<\cdots<a_{\alpha(n+1)}$ or $a_{\alpha(1)}>\cdots>a_{\alpha(n+1)}$.
\end{theorem}
    
    \begin{lemma}\label{ABexists}
    Let $w\in \mathbb{N}$ and let $(G,L)$ be an instance of the \textsc{Ordered Graph List-3-Coloring Problem} where $G$ is $J_w$-free. Let $\phi$ be an $L$-coloring of $G$. Then for every edge $e=uv\in E(G)$, there exists $g_e\in \Gamma_w(G,L,e)$ such that $(\phi,g_e)$ has both  properties \textsf{X} and \textsf{Y}.
    \end{lemma}
    \begin{proof}
    Let $i\in [3]$ be arbitrary. Let $U_i$ be the set of all vertices in $\lft_G(e)$ with at least one neighbor in $\phi^{-1}(i)\cap \und_G(e)$, and $S_i$ be a minimal subset of $\phi^{-1}(i)\cap \und_G(e)$ such that every vertex in $U_i$ has a neighbor in $S_i$. Let $|S_i|=k_i$ and $S_i=\{s_1^i,\ldots, s_{k_i}^i\}$ where $\varphi_G(s_1^i)<\cdots<\varphi_G(s_{k_i}^i)$. It follows from the minimality of $S_i$ that there are $k_i$ distinct vertices $t^i_1,\ldots, t^i_{k_i}\in U_i$ such that for each $j\in [k_i]$, we have $N_G(t^i_j)\cap S_i=\{s^i_j\}$. We further deduce:
    
    \sta{\label{Erdos-Szekeresing} For each $i\in [3]$, we have $k_i\leq 27w^2+2$.}
    
    Suppose for a contradiction that $k_i\geq 3(9w^2+1)$ for some $i\in [3]$. Since $\phi$ is a $3$-coloring of $G$, there exists an increasing injection from $\alpha:[9w^2+1]\rightarrow [k_i]$ such that $\{t^i_{\alpha(j)}:j\in [9w^2+1]\}$ is a stable set in $G$. In addition, $\{s^i_{\alpha(j)}:j\in [9w^2+1]\}\subseteq S_i\subseteq \phi^{-1}(i)$ is also a stable set of $G$. As a result, we have $E(G[\{s^i_{\alpha(j)},t^i_{\alpha(j)}:j\in [9w^2+1]\}])=\{t^i_{\alpha(j)}s^i_{\alpha(j)}:j\in [9w^2+1]\}$. By Theorem \ref{ErdosSzekeres} applied to the sequence $\varphi_G(t^i_{\alpha(1)}), \ldots, \varphi_G(t^i_{\alpha(9w^2+1)})$ of integers, there exists an increasing injection $\beta:[3w+1]\rightarrow \alpha([9w^2+1])$ such that either $\varphi_G(t^i_{\beta(1)})<\cdots<\varphi_G(t^i_{\beta(3w+1)})$ or $\varphi_G(t^i_{\beta(1)})>\cdots>\varphi_G(t^i_{\beta(3w+1)})$. Recall also that, since both $\alpha$ and $\beta$ are increasing, it follows that $\varphi_G(s^i_{\beta(1)})<\cdots<\varphi_G(s^i_{\beta(3w+1)})$. Now, in the former case, let
    
    $$W=\{s^i_{\beta(1)},\ldots, s^i_{\beta(2w+1)}\}\cup \{t^i_{\beta(w+1)}\}\cup \{t^i_{\beta(2w+2)},\ldots, t^i_{\beta(3w+1)}\}.$$
    
    Then one may readisly observe that $E(G[W])=\{s^i_{\beta(w+1)}t^i_{\beta(w+1)}\}$, and so $G[W]$ is isomorphic to $J_w$. Also, in the latter case, let $$W=\{s^i_{\beta(1)},\ldots, s^i_{\beta(w)}\}\cup \{s^i_{\beta(2w+1)},\ldots, s^i_{\beta(3w+1)}\}\cup \{t^i_{\beta(w+1)},\ldots, t^i_{\beta(2w+1)}\}.$$
    Then it straightforward to check that $E(G[W])=\{s^i_{\beta(2w+1)}t^i_{\beta(2w+1)}\}$, and so $G[W]$ is isomorphic to $J_w$. The last two conclusions violate the assumption that $G$ is $J_w$-free, hence proving \eqref{Erdos-Szekeresing}.\vsp
    
    Now, let $S=S_1\cup S_2\cup S_3\cup \{u,v\}$, $\sigma=\phi|_S$, and $g_e=(S,\sigma)$. Note that $S$ is a subset of $\und_G(e)$ with $u,v\in S$, $\sigma:S\rightarrow [3]$ is an $L|_S$-coloring of $G[S]$, and by \eqref{Erdos-Szekeresing}, for each $i\in [3]$, we have $|\sigma^{-1}(i)|\leq |S_i|+1=k_i+1\leq 27w^2+3$. Therefore, we have $g_e\in \Gamma_w(G,L,e)$. Also, it follows directly from the definition of $g_e$ that $(\phi,g_e)$ has property \textsf{X}. In particular, $\phi$ and $\sigma$ are compatible. Moreover, for each $i\in [3]$, by the definition of $U_i$ and $S_i$,  every vertex in $\lft_G(e)$ with a neighbor in $\phi^{-1}(i)\cap \und_G(e)$ has a neighbor in $S_i\subseteq \sigma^{-1}(i)$. Hence, $(\phi,\sigma)$ has property \textsf{Y}. This completes the proof of Lemma \ref{ABexists}.
    \end{proof}

We continue with a few more definitions.  For an instance $(G,L)$ of the \textsc{Ordered Graph List-3-Coloring Problem}, an edge $e=uv\in E(G)$ is said to be \textit{maximal} if there is no edge $u'v'\in E(G)$ with $\varphi_G(u')\leq \varphi_G(u)$ and $\varphi_G(v)\leq \varphi_G(v')$. We denote the set of all maximal edges of $G$ by $\mx(G)$. Let $u_1v_1,u_2v_2\in \mx(G)$ be two distinct maximal edges of $G$ with $\varphi_G(u_1)< \varphi_G(v_1)$ and $\varphi_G(u_2)< \varphi_G(v_2)$. Then we have $u_1\neq u_2$ and $v_1\neq v_2$, and if $\varphi_G(u_1)< \varphi_G(u_2)$, then $\varphi_G(v_1)< \varphi_G(v_2)$. As a result, $\varphi_G$ induces a natural ordering on the elements of $\mx(G)$. In accordance, for distinct $u_1v_1,u_2v_2\in \mx(G)$, we say $u_1v_1$ \textit{is before} (resp. \textit{after}) $u_2v_2$ if $\varphi_G(u_1)<\varphi_G(u_2)$ (resp. $\varphi_G(u_2)<\varphi_G(u_1)$), and we say $u_1v_1$ \textit{is immediately before} (resp. \textit{immediately after}) $u_2v_2$ if $u_1v_1$ is before (resp. after) $u_2v_2$ and there is no $u_3v_3\in \mx(G)$ with $\varphi_G(u_1)<\varphi_G(u_3)<\varphi_G(u_2)$ (resp. $\varphi_G(u_2)<\varphi_G(u_3)<\varphi_G(u_1)$). By the \textit{first} maximal edge of $G$, we mean the one that is before every other maximal edge of $G$. Note that for every vertex $x\in V(G)$, either $x$ is isolated or $x\in \und_G(e)$ for some maximal edge $e\in E(G)$.We leave the proof of the following lemma to the reader.

\begin{lemma}\label{computemax}
Given and instance $(G,L)$ of the \textsc{Ordered Graph List-3-Coloring Problem}, we have $|\mx(G)|\leq |V(G)|-1$ and one can compute $\mx(G)$ in time $\mathcal{O}(|V(G)|^{4})$.
\end{lemma}

    Next we introduce the key notion of ``successfulness'', which has a recursive definition. Let $(G,L)$ be an instance of the \textsc{Ordered Graph List-3-Coloring Problem}. For an edge $e\in \mx(G)$ and $g=(S,\sigma)\in \Gamma_w(G,L,e)$, we say $g$ is \textit{successful} if either $e$ is the first maximal edge of $G$, or, having the successful elements of $\Gamma_w(G,L,e')$ defined for the edge $e'\in \mx(G)$ immediately before $e$, there exists a successful pair $g'=(T,\tau)\in \Gamma_w(G,L,e')$, such that
    \vsp
    \begin{enumerate}[(S)]
         \item\otherlabel{t3}{S} $G[\und_G(e')]$ admits an $L|_{\und_G(e')}$-coloring $\psi$ for which both $(\psi,g)$ and $(\psi,g')$ have both properties \textsf{X} and \textsf{Y}.
   \end{enumerate}
   \vsp
The following two lemmas examine how successfulness interacts with properties \textsf{X} and \textsf{Y}.

   \begin{lemma}\label{hasA}
Let $w\in \mathbb{N}$ and let $(G,L)$ be an instance of the \textsc{Ordered Graph List-3-Coloring Problem}, where $G$ is $J_w$-free and $L(v)\neq \emptyset$ for all $v\in V(G)$. Let $e=uv\in \mx(G)$ and $g=(S,\sigma)\in \Gamma_w(G,L,e)$ be successful. Then $G[\lft_G(e)]$ admits an $L|_{\lft_G(e)}$-coloring $\phi$ such that $(\phi,g)$ has property \textsf{X}.
\end{lemma}
\begin{proof}
Suppose not. Going through the natural ordering of the elements of $\mx(G)$ induced by $\varphi_G$, let $e$ be the first edge violating Lemma \ref{hasA}. Note that if $e$ is the first maximal edge of $G$, then $\lft_G(e)=\emptyset$, and so $e$ vacuously satisfies Lemma \ref{hasA}, a contradiction. Therefore, we may choose $e'=u'v'$ distinct from $e$ to be the maximal edge of $G$ immediately before $e$. Since $g$ is successful, there exists a successful pair $g'=(T,\tau)\in \Gamma_w(G,L,e')$ such that $e,e',g$ and $g'$ satisfy \eqref{t3} with $\psi$ as in \eqref{t3}. Also, from $g' \in \Gamma_w(G,L,e')$ being successful, and the choice of $e$, it follows that there exists an $L|_{\lft_G(e')}$-coloring $\phi'$ of $G[\lft_G(e')]$ such that $(\phi',g')$ has property \textsf{X}. Note that every vertex in the set $\lft_G(e)$ is either isolated or belongs to the disjoint union of $\lft_G(e')$ and $\und_G(e')\setminus \und_G(e)$. So we can define the map $\phi:\lft_G(e)\rightarrow [3]$ as $\phi(x)=\phi'(x)$ for all $x\in \lft_{G}(e')$, $\phi(x)=\psi(x)$ for all $x\in \und_G(e')\setminus \und_G(e)$, and $\phi(x)\in L(x)$ is chosen arbitrarily for every isolated vertex $x\in \lft_G(e)$ (the latter is possible because $L(x)\neq \emptyset$.) We deduce the following.

\sta{\label{phiiscoloring} The map $\phi$ is an $L|_{\lft_{G}(e)}$-coloring of $G|\lft_{G}(e)$.}

Suppose not. Then since $\phi'$ is an $L|_{\lft_{G}(e')}$-coloring of $G[\lft_{G}(e')]$ and $\psi$ is an $L|_{\und_G(e')}$-coloring of $G[\und_G(e')]$, there exists an edge $xy\in E(G)$ with $x\in \lft_{G}(e')$ and $y\in \und(e')_G\setminus \und_G(e)$ such that $\phi'(x)=\phi(x)=\phi(y)=\psi(y)$. Now, by \eqref{t3}, $(\psi,g')$ has property \textsf{Y}, which together with \eqref{q2} for $\psi$ and $\tau$, implies that $x$ has a neighbor $z\in \tau^{-1}(\psi(y))$. In other words, $x$ has a neighbor $z\in T$ with $\tau(z)=\psi(y)=\phi'(x)$. But this shows that $\phi'$ and $\tau$ do not satisfy \eqref{p2}, which violates our assumption that $(\phi',g')$ has property \textsf{X}. This proves \eqref{phiiscoloring}.

\sta{\label{phihasA} The pair $(\phi,g)$ has property \textsf{X}.}

Note that from $\lft_{G}(e)\cap S=\emptyset$, it follows readily that $(\phi,\sigma)$ satisfies \eqref{p1}. So it remains to show that \eqref{p2} holds for $(\phi,\sigma)$, that is, $\sigma\cup \phi$ is an $L|_{S\cup \lft_{G}(e)}$-coloring of $G[(S\cup \lft_{G}(e)]$. Suppose not. By \eqref{phiiscoloring}, $\phi$ is an $L|_{\lft_{G}(e)}$-coloring of $G[\lft_{G}(e)]$ and $\sigma$ is an $L|_{S}$-coloring of $G[S]$. Thus, there exists an edge $xy\in E(G)$ with $x\in \lft_{G}(e)$ and $y\in S$ such that $\phi(x)=\sigma(y)$. First, assume that $x\in \lft_{G}(e')$, and so $\phi'(x)=\phi(x)=\sigma(y)$. Then since $e'$ is a maximal edge of $G$, we have $y\in \und_G(e')\cap S$. On the other hand, by \eqref{t3}, $(\psi,g)$ has property \textsf{X}, and so by \eqref{p1}, $\psi$ and $\sigma$ are compatible. Thus, we have $\sigma(y)=\psi(y)$, and so $\phi'(x)=\psi(y)$. Now, by \eqref{t3}, $(\psi,g')$ has property \textsf{Y}, which along with \eqref{q2} for $\psi$ and $\tau$ implies that $x$ has a neighbor $z\in \tau^{-1}(\psi(y))$. In other words, $x$ has a neighbor $z\in T$ with $\tau(z)=\psi(y)=\phi'(x)$. But this implies that $\phi'$ and $\tau$ do not satisfy \eqref{p2}, which violates our assumption that $(\phi',g')$ has property \textsf{X}. It follows that $x\notin \lft_{G}(e')$, that is, $x\in \und_G(e')\setminus \und_G(e)$, and so $\psi(x)=\phi(x)=\sigma(y)$. But this shows that $\sigma\cup \psi$ is not an $L|_{S\cup \und_G(e')}$-coloring of $G[(S\cup \und_G(e'))]$, and so by \eqref{p2}, $(\psi,g)$ does not have property \textsf{X}, which in turn contradicts \eqref{t3}. This proves \eqref{phihasA}.\vsp

Finally, from \eqref{phiiscoloring} and \eqref{phihasA}, we conclude that $G[\lft_{G}(e)]$ admits an $L|_{\lft_{G}(e)}$-coloring $\phi$ such that $(\phi,g)$ has property \textsf{X}, a contradiction. This completes the proof of Lemma \ref{hasA}.
\end{proof}

\begin{lemma}\label{hasAB}
Let $w\in \mathbb{N}$ and let $(G,L)$ be an instance of the \textsc{Ordered Graph List-3-Coloring Problem} where $G$ is $J_w$-free. Let $\phi$ be an $L$-coloring of $G$. Assume that for every $e\in \mx(G)$, there exists $g_e\in \Gamma_w(G,L,e)$ such that $(\phi,g_e)$ has both properties \textsf{X} and \textsf{Y}. Then for every $e\in \mx(G)$, $g_e$ is successful.
\end{lemma}
\begin{proof}
Suppose not. Going through the natural ordering of the elements of $\mx(G)$ induced by $\varphi_G$, let $e$ be the first edge where $g_e=(S,\sigma)$ is not successful. So $e$ is not the first maximal edge of $G$, and we may choose $e'=u'v'$ distinct from $e$ to be the maximal edge of $G$ immediately before $e$. It follows that $g_{e'}=(T,\tau)$ is successful. Since both $(\phi,g_e)$ and $(\phi,g_{e'})$ have property \textsf{X}, by \eqref{p1}, $\phi$ and $\sigma$ are compatible, and $\phi$ and $\tau$ are compatible, as well. It follows that $\sigma=\phi|_S$ and $\tau=\phi|_{T}$. Now, from $\sigma=\phi|_S$, Lemma \ref{PQprop} and the fact that $(\phi,g_{e'})$ has property \textsf{X}, it follows that:

\sta{\label{hass1} The pair $(\sigma,g_{e'})$ has property \textsf{X}, and so $\sigma\cup \tau$ is an $L|_{S\cup T}$-coloring of $G[(S\cup T)]$.}

Also, from \eqref{hass1}, $\sigma=\phi|_S$ and $\tau=\phi|_{T}$, it follows that $\sigma\cup \tau=\phi|_{(S\cup T)}$. This, along with Lemma \ref{PQprop} and the fact that $(\phi,g_{e'})$ has property \textsf{Y}, implies that

\sta{\label{hass2} The pair $(\sigma\cup \tau,g_{e'})$ has property \textsf{Y}.}

Finally, let $\psi=\phi|_{\und_G(e')}$. Then by Lemma \ref{PQprop} and the fact that both $(\phi,g_{e})$ and $(\phi,g_{e'})$ have both properties \textsf{X} and \textsf{Y}, we have

\sta{\label{hass3} $\psi$ is an $L|_{\und_G(e')}$-coloring of $G[\und_G(e')]$ such that both $(\psi,g_e)$ and $(\psi,g_{e'})$ have both properties \textsf{X} and \textsf{Y}.}

But then by \eqref{hass1}, \eqref{hass2} and \eqref{hass3}, $e,e',g_{e}$ and $g_{e'}$ satisfy \eqref{t3}, and so $g_e$ is successful, a contradiction. This completes the proof of Lemma \ref{hasAB}.
\end{proof}

Our last two lemmas use the following definition. For an instance $(G,L)$ of the \textsc{Ordered Graph List-3-Coloring Problem}, we define the instance $(G^*,L^*)$ as follows. Let $G^*$ be the ordered graph with $V(G^*)=V(G)\cup \{q_1,q_2\}$ for two new vertices $q_1$ and $q_2$ and $E(G^*)=E(G)\cup \{q_1q_2\}$, where  $\varphi_{G^*}|_{V(G)}=\varphi_G$ and $\varphi_{G^*}(q_i)=i+ \max_{v\in V(G)}\varphi_G(v)$ for each $i\in \{1,2\}$. Also, let $L^*|_{V(G)}=L$ and $L^*(q_i)=\{i\}$ for every $i\in \{1,2\}$. Note that if $G$ is $J_w$-free, then $G^*$ is $J_{w+1}$-free. Moreover, we have $\mx(G^*)=\mx(G)\cup \{q_1q_2\}$. We deduce the following from Lemmas \ref{ABexists}, \ref{hasA} and \ref{hasAB}.

\begin{lemma}\label{G+}
Let $w\in \mathbb{N}$ and let $(G,L)$ be an instance of the \textsc{Ordered Graph List-3-Coloring Problem} where $G$ is $J_w$-free and $L(v)\neq \emptyset$ for all $v\in V(G)$. Then $G$ admits an $L$-coloring if and only if the set $\Gamma_{w+1}(G^*,L^*,q_1q_2)$ contains a successful element.
\end{lemma}
\begin{proof}
For the ``only if'' implication, assume that $\phi$ is an $L$-coloring of $G$. Then we may extend $\phi$ to an $L^*$-coloring $\phi$ of $G^*$ by defining $\phi(q_i)=i$ for each $i\in \{1,2\}$. Now, since $G^*$ is $J_{w+1}$-free, by Lemma \ref{ABexists} applied to $(G^*,L^*)$ and $\phi$, for every $e\in \mx(G^*)$, there exists $g_e\in \Gamma_{w+1}(G^*,L^*,e)$ such that $(\phi,g_e)$ has both properties \textsf{X} and \textsf{Y}. It follows from Lemma \ref{hasAB} applied to $(G^*,L^*)$, $\phi$ and these $g_e$'s, that for every $e\in \mx(G^*)$, $g_e$ is successful, and in particular, $g_{q_1q_2}\in \Gamma_{w+1}(G^*,L^*,q_1q_2)$ is successful. For the ``if'' implication, let $g\in \Gamma_{w+1}(G^*,L^*,q_1q_2)$ be successful. By Lemma \ref{hasA} applied to $(G^*,L^*)$ and $g$, $G^*[\lft_{G^*}(q_1q_2)]$ admits an $L^*|_{\lft_{G^*}(q_1q_2)}$-coloring. Also, by the definition of $(G^*,L^*)$, we have $G^*[\lft_{G^*}(q_1q_2)]=G$ and $L^*|_{\lft_{G^*}(q_1q_2)}=L$. Hence, $G$ admits and $L$-coloring. This completes the proof of Lemma \ref{G+}.
\end{proof}

We need some preparation for proving our last lemma. The following theorem was proved in \cite{hajebi:2021} for unordered graphs (which, of course, does not interfere with the correctness of the following.)
\begin{theorem}[Hajebi, Li and Spirkl \cite{hajebi:2021}]\label{kill1}
				Let $(G,L)$ be an instance of the \textsc{Ordered Graph List-3-Coloring Problem}. Then there exists a $(G,L)$-refinement $(\hat{G},\hat{L})$ with the following specifications.
			\begin{itemize}
				\item $(\hat{G},\hat{L})$ can be computed from $(G,L)$ in time $\mathcal{O}(|V(G)|^2)$.
				\item $|\hat{L}(v)|\neq 1$ for all $v\in V(\hat{G})$.
				\item If $G$ admits an $L$-coloring if and only if $\hat{G}$ admits an $\hat{L}$-coloring.
			\end{itemize}
		\end{theorem}
   
The following can be proved via a reduction to the \textsc{2SAT Problem}, and has been discovered independently by many authors \cite{2sat1,2sat2,2sat3}.
\begin{theorem}[Edwards \cite{2sat1}]\label{2LC}
Let $k\in \mathbb{N}$ be fixed and  $(G,L)$ be an instance of the \textsc{List-$k$-Coloring Problem} with $|L(v)|\leq 2$ for every $v\in V(G)$. Then it can be decided in time $\mathcal{O}(|V(G)|^2)$ whether $G$ admits an $L$-coloring.
\end{theorem}

Our use of Theorem \ref{2LC} though is through the following straightforward corollary of it. We leave the proof to the reader.

\begin{cor}\label{2LCcor}
Let $(G,L)$ be an instance of the \textsc{Ordered Graph List-$3$-Coloring Problem} and $c\in \mathbb{N}\cup \{0\}$. Then the following hold.
\begin{itemize}
    \item If $|\{x\in V(G): |L(x)|=3\}|\leq c$, then one can decide in time $\mathcal{O}(|V(G)|^{2})$ whether $G$ admits an $L$-coloring.
    \item One can decide in time $\mathcal{O}(|V(G)|^{c+4})$ whether $G$ admits an $L$-coloring $\xi$ with $|\xi^{-1}(i)|<c$ for some $i\in [3]$.
\end{itemize}
\end{cor}
   The following is the penultimate step toward a proof of Theorem \ref{mainwhale2}.
   
 \begin{lemma}\label{checksuccess}
    Let $w\in \mathbb{N}$ and let $(G,L)$ be an instance of the \textsc{Ordered Graph List-3-Coloring Problem} where $G$ is $J_w$-free. Then one of the following holds.
    \begin{itemize}
    \item $G$ contains a clique on four vertices.
        \item $G$ admits an $L$-coloring $\xi$ with $|\xi^{-1}(i)|<2w$ for some $i\in [3]$.
        \item There exists a $(G,L)$-profile $\Sigma(G,L)$ with the following specifications.
        \begin{itemize}
            \item \sloppy $|\Sigma(G,L)|\leq (|V(G)|^{6w})$ and $\Sigma(G,L)$ can be computed from $(G,L)$ in time $\mathcal{O}(|V(G)|^{6w+2})$.
            \item $G$ admits an $L$-coloring if and only if for some $(G',L')\in \Sigma(G,L)$, $G'$ admits an $L'$-coloring.
            \item Given $(G',L')\in \Sigma(G,L)$ with $L(v)\neq \emptyset$ for all $v\in V(G')$, one can compute the set of all successful elements of $\Gamma_{w+1}(G'^*,L'^*,q_1q_2)$ in time $\mathcal{O}(|V(G)|^{162w^2+18})$.
        \end{itemize}
    \end{itemize}
    \end{lemma}
    \begin{proof}
    
Assume that the first two bullets of Lemma \ref{checksuccess} do not hold. We define $\mathcal{A}$ to be the set of all $6$-tuples $\alpha=(X_1,X_2,X_3,Y_1,Y_2,Y_3)$ where for each $i\in [3]$, $X_i$ and $Y_i$ are stable subsets of $L^{(i)}$ of cardinality $w$ such that for every choice of $x\in X_i$ and $y\in Y_i$, we have $\varphi_G(x)< \varphi_G(y)$, and $X_1,X_2,X_3,Y_1,Y_2$ and $Y_3$ are pairwise disjoint. It is readily observed that:

\sta{\label{computeA} We have $|\mathcal{A}|\leq \mathcal{O}(|V(G)|^{6w})$ and $\mathcal{A}$ can be computed from $(G,L)$ in time $\mathcal{O}(|V(G)|^{6w})$.}

Next, for every $\alpha=(X_1,X_2,X_3,Y_1,Y_2,Y_3)\in \mathcal{A}$, we define $L_{\alpha}:V(G)\rightarrow 2^{[3]}$ as follows. For every vertex $v\in V(G)$, if $v\in X_i\cup Y_i$ for some $i\in [3]$, then let $L_{\alpha}(v)=\{i\}$, and otherwise, let \[L_{\alpha}(v)=\{i\in L(v):\max_{x\in X_i}\varphi_G(x)<\varphi_G(v)<\min_{y\in Y_i}\varphi_G(y)\}\setminus \{i\in [3]:N_G(v)\cap (X_i\cup Y_i) \neq \emptyset\}.\]
Then for every $\alpha\in \mathcal{A}$,  $(G,L_{\alpha})$ is a spanning $(G,L)$-refinement. 
For every $\alpha\in \mathcal{A}$, let $(\hat{G},\hat{L}_{\alpha})$ be the $(G,L_{\alpha})$-refinement as in Theorem \ref{kill1}, and $\Sigma(G,L)=\{(\hat{G},\hat{L}_{\alpha}):\alpha\in \mathcal{A}\}$. Then  $\Sigma(G,L)$ is a $(G,L)$-profile. We deduce:

\sta{\label{computeSigma} $|\Sigma(G,L)|\leq \mathcal{O}(|V(G)|^{6w})$ and $\Sigma(G,L)$ can be computed from $(G,L)$ in time $\mathcal{O}(|V(G)|^{6w+2})$.}

The first assertion is immediate from \eqref{computeA} and the fact that $|\Sigma(G,L)|=|\mathcal{A}|$. The second assertion follows directly from the combination of \eqref{computeA}, the fact that for all $\alpha\in \mathcal{A}$ and $v\in V(G)$, $L_{\alpha}(v)$ can be computed from $(G,L)$ and $\alpha$ in constant time, and the first bullet of Theorem \ref{kill1}. This proves \eqref{computeSigma}.

\sta{\label{coloriffSigma} $G$ admits an $L$-coloring if and only if for some $\alpha\in \mathcal{A}$, $\hat{G}$ admits an $\hat{L}_{\alpha}$-coloring.}

For the ``if'' implication, let $\hat{G}$ admit an $\hat{L}_{\alpha}$-coloring for some $\alpha\in \mathcal{A}$. Then by the last bullet of Theorem \ref{kill1}, $(G,L_{\alpha})$ admits an $L_{\alpha}$-coloring, and so since $(G,L_{\alpha})$ is a spanning $(G,L)$-refinement, it follows that $G$ admits an $L$-coloring. To see the ``only if'' implication, assume that $G$ admits an $L$-coloring $\xi$. Since the second bullet of Lemma \ref{checksuccess} does not hold, we have $|\xi^{-1}(i)|\geq 2w$ for all $i\in [3]$. Thus, with respect to the ordering of $V(G)$ inherited from $\varphi_G$, for each $i\in [3]$, we may define $X_i$ to be the set of first $w$ vertices in $\xi^{-1}(i)$ and $Y_i$ to be the set of last $w$ vertices in $\xi^{-1}(i)$, where $X_i\cap Y_i\neq \emptyset$. We conclude immediately from this definition that $\alpha=(X_1,X_2,X_3,Y_1,Y_2,Y_3)\in \mathcal{A}$. Now, for each $i\in [3]$ and all $v\in X_i\cup Y_i$, we have $\xi(v)=i\in \{i\}=L_{\alpha}(v)$. Also, for every $v\in V(G)\setminus \cup_{i=1}^3(X_i\cup Y_i)$, since $\xi$ is an $L$-coloring of $G$ and by the choice of $X_i,Y_i$ for $i\in \{1,2,3\}$, we have $\max_{x\in X_{\xi(v)}}\varphi_G(x)<\varphi_G(v)<\min_{y\in Y_{\xi(v)}}\varphi_G(y)$ and $N_G(v)\cap (X_{\xi(v)}\cup Y_{\xi(v)})=\emptyset$, which in turn implies that $\xi(v)\in L_{\alpha}(v)$. Therefore, we have $\xi(v)\in L_{\alpha}(v)$ for all $v\in V(G)$, and so $\xi$ is an $L_{\alpha}$-coloring of $G$. But then by the third bullet of Theorem \ref{kill1}, $\hat{G}$ admits an $\hat{L}_{\alpha}$-coloring. This proves \eqref{coloriffSigma}.

\sta{\label{Algexistperguess} Assume that $(\hat{G},\hat{L}_{\alpha})\in \Sigma(G,L)$ with $\hat{L}_{\alpha}(v)\neq \emptyset$ for all $v\in V(\hat{G})$, and $e,e'\in \mx(\hat{G}^*)$ such that $e'$ is immediately before $e$. Let $g'=(T,\tau)\in \Gamma_{w+1}(\hat{G}^*,\hat{L}^*_{\alpha},e')$ be successful, and $g=(S,\sigma)\in \Gamma_{w+1}(\hat{G}^*,\hat{L}^*_{\alpha},e)$. Then one can decide in time $\mathcal{O}(|V(G)|^{3})$ whether $e,e',g$ and $g'$ satisfy \eqref{t3}.}

Let $\alpha=(X_1,X_2,X_3,Y_1,Y_2,Y_3)$. Let $e'=u'v'$, and by symmtery, assume that $\phi_G(u')\leq \phi_G(v')$. Then we have $u',v'\in T$ and $\{u_1,u_2\}\cap \{q_1,q_2\}=\emptyset$, with $q_1,q_2$ as in the definition of $(\hat{G}^*,\hat{L}_{\alpha}^*)$. Also, we have $\hat{L}^*_{\alpha}(u'),\hat{L}^*_{\alpha}(v')\neq \emptyset$, and so by the second bullet of Theorem \ref{kill1}, we have $|L^*_{\alpha}(u')|,|L^*_{\alpha}(v')|\geq 2$. As a result, we may pick $i_0\in \hat{L}^*_{\alpha}(u')\cap \hat{L}^*_{\alpha}(v')$. Note that by the definition of $\hat{L}_{\alpha}$, $\{u',v'\}$ is anticomplete to $X_{i_0}\cup Y_{i_0}$ and we have $\varphi_G(x)<\varphi_G(u')<\varphi_G(v')<\varphi_G(y)$ for every choice of $x\in X_{i_0}$ and $y\in Y_{i_0}$. For every $x\in \und_{\hat{G}^*}(e')$,  we define $C_x\subseteq [3]$ as follows. If $x\in \und_{\hat{G}^*}(e')\setminus \und_{\hat{G}^*}(e)$, then let $C_x$ be the set of all $i\in [3]$ for which $x$ has a neighbor in $\lft_{\hat{G}^*}(e')$ that is anticomplete to $\tau^{-1}(i)$, and if $x\in \und_{\hat{G}^*}(e')\cap \und_{\hat{G}^*}(e)$, the let $C_x$ be the set of all $i\in [3]$ for which either $x$ has a neighbor in $\lft_{\hat{G}^*}(e')$ that is anticomplete to $\tau^{-1}(i)$, or $x$ has a neighbor in $\lft_{\hat{G}^*}(e)$ that is anticomplete to $\sigma^{-1}(i)$. Also, for every $x\in \und_{\hat{G}^*}(e')$, let $D_{x}=C_x\cup \sigma(N_G(x)\cap S)\cup \tau(N_G(x)\cap T)\subseteq [3]$. It is straightforward to show that for every $x\in \und_{\hat{G}^*}(e')$, one can compute $D_{x}$ from $(\hat{G},\hat{L}_{\alpha})$ in time $\mathcal{O}(|V(G)|^2)$. Having the latter definition, we define $\tilde{L}_{\alpha}:\und_{\hat{G}^*}(e')\rightarrow 2^{[3]}$ as follows. For every $x\in \und_{\hat{G}^*}(e')$, 
\begin{itemize}
    \item if $x\in T\setminus S$, then let $\tilde{L}_{\alpha}(x)=\{\tau(x)\}\setminus D_{x}$,
    \item if $x\in (S\cap \und_{\hat{G}^*}(e'))\setminus T$, then let $\tilde{L}_{\alpha}(x)=\{\sigma(x)\}\setminus D_{x}$,
    \item if $x\in S\cap T$, then let $\tilde{L}_{\alpha}(x)=(\{\sigma(x)\}\cap \{\tau(x)\})\setminus D_{x}$, and; 
    \item if $x\in \und_{\hat{G}^*}(e')\setminus (S\cup T)$, then let $\tilde{L}_{\alpha}(x)=\hat{L}_{\alpha}(x)\setminus D_{x}$.
\end{itemize}
Now, one can easily show that
\begin{enumerate}[(i)]
		\item \label{i} $(\hat{G}^*[\und_{\hat{G}^*}(e')],\tilde{L}_{\alpha})$ is a $(\hat{G}^*,\hat{L}^*_{\alpha})$-refinement which can be computed from $(\hat{G},\hat{L}_{\alpha}),e,e',g$ and $g'$ in time $\mathcal{O}(|V(G)|^3)$, and;
		\item \label{ii} we have that $e,e',g$ and $g'$ satisfy \eqref{t3} if and only if $\hat{G}^*[\und_{\hat{G}^*}(e')]$ admits an $\tilde{L}_{\alpha}$-coloring.
\end{enumerate}
 On the other hand, it follows from the definition of $\tilde{L}_{\alpha}$ that $\{x\in \und_{\hat{G}^*}(e'):|\tilde{L}_{\alpha}(x)|=3\}$ is disjoint from and anticomplete to $X_{i_0}\cup Y_{i_0}\cup \{u',v'\}$. Thus, if $G[\{x\in \und_{\hat{G}^*}(e'):|\tilde{L}_{\alpha}(x)|=3\}]$ has a stable set $I$ of size $w$, then $G[I\cup X_{i_0}\cup Y_{i_0}\cup \{u',v'\}]$ is isomorphic to $J_w$, which is impossible. So $G[\{x\in \und_{\hat{G}^*}(e'):|\tilde{L}_{\alpha}(x)|=3\}]$ has no stable set of size $w$. This, together with the assumption that the first bullet of Lemma \ref{checksuccess} does not hold, implies that $|\{x\in \und_{\hat{G}^*}(e'):|\tilde{L}_{\alpha}(x)|=3\}|<R(w,4)$, where $R(w,4)$ is the Ramsey number. But then by first bullet of Theorem \ref{2LCcor}, one can decide in time $\mathcal{O}(|V(G)|^{2})$ whether $G[\und_{\hat{G}^*}(e')]$ admits an $\tilde{L}_{\alpha}$-coloring. Hence, by \eqref{i} and \eqref{ii}, we conclude that given $(\hat{G},\hat{L}_{\alpha}),e,e',g$ and $g'$, one can decide in time $\mathcal{O}(|V(G)|^{3})$ whether $e,e',g$ and $g'$ satisfy \eqref{t3}. This proves \eqref{Algexistperguess}.\vsp
 
 From \eqref{Algexistperguess} and Lemma \ref{computeguess}, we deduce:
 
 \sta{\label{Algexistperedge} Assume that $(\hat{G},\hat{L}_{\alpha})\in \Sigma(G,L)$ with $\hat{L}_{\alpha}(v)\neq \emptyset$ for all $v\in V(\hat{G})$, and $e,e'\in \mx(\hat{G}^*)$ such that $e'$ is immediately before $e$. Assume further that the set of all successful elements of $\Gamma_{w+1}(\hat{G}^*,\hat{L}^*_{\alpha},e')$ is given. Then one can compute the set of all successful elements of $\Gamma_{w+1}(\hat{G}^*,\hat{L}^*_{\alpha},e)$ in time $\mathcal{O}(|V(G)|^{162w^2+17})$.}
 
 As a result, 
 
\sta{\label{Algexist} Given $(\hat{G},\hat{L}_{\alpha})\in \Sigma(G,L)$ with $\hat{L}_{\alpha}(v)\neq \emptyset$ for all $v\in V(\hat{G})$,
one can compute the set of all successful elements of $\Gamma_{w+1}(\hat{G}^*,\hat{L}^*_{\alpha},q_1q_2)$ in time $\mathcal{O}(|V(G)|^{162w^2+18})$.}

First, we compute $\mx(\hat{G}^*)$, which by Lemma \ref{computemax} is doable in time $\mathcal{O}(|V(G)|^4)$. Then we sort the elements of $\mx(\hat{G}^*)$ according to the natural ordering induced by $\varphi_G$, which is immediately observed to be doable in time $\mathcal{O}(|V(G)|^2)$. Now, for the first maximal edge $e_0$ of $\hat{G}^*$, all elements of $\Gamma_w(\hat{G}^*,\hat{L}^*_{\alpha},e_0)$ are successful. Also, by Lemma \ref{computemax}, we have $|\mx(\hat{G}^*)|\leq |V(G)|-1$. Therefore, going through the ordering of the elements of $\mx(\hat{G}^*)$ starting at the maximal edge immediately after $e_0$ and applying \eqref{Algexistperedge}, we can compute the set of all successful elements of $\Gamma_w(\hat{G}^*,\hat{L}^*_{\alpha},q_1q_2)$ in time $\mathcal{O}(|V(G)|^{162w^2+18})$. This proves \eqref{Algexist}.\vsp

Finally, by \eqref{computeSigma}, \eqref{coloriffSigma} and \eqref{Algexist}, $\Sigma(G,L)$ is a $(G,L)$-refinement satisfying the third bullet of Lemma \ref{checksuccess}. This completes the proof.
\end{proof}

Now we are in a position to prove Theorem \ref{mainwhale2}.

\begin{proof}[Proof of Theorem \ref{mainwhale2}]
We propose the following algorithm. 

	\begin{enumerate}[\textit{Step }1:]
		\item \label{s1} Determine whether $G$ has a clique on four vertices. If so, return ``$G$ does not admit an $L$-coloring.'' Otherwise, go to the next step.
		\item \label{s2} Determine whether $G$ admits an $L$-coloring $\xi$ with $|\xi^{-1}(i)|<2w$ for some $i\in [3]$. If so, return ``$G$ admits an $L$-coloring.'' Otherwise, go to the next step.
		\item \label{s3}  Compute the $(G,L)$-profile $\Sigma(G,L)$ described in the third bullet of Lemma \ref{checksuccess}.
		\item \label{s4} Compute the set $\Sigma^{\flat}(G,L)$ of all $(G,L)$-refinements $(G',L')\in \Sigma(G,L)$ with $L'(v)\neq \emptyset$ for all $v\in V(G')$. If $\Sigma^{\flat}(G,L)=\emptyset$, then return ``$G$ does not admit an $L$-coloring." Otherwise, go to the next step.
		\item \label{s5} For each $(G',L')\in \Sigma^{\flat}(G,L)$, compute the set of all successful elements of $\Gamma_{w+1}(G'^*,L'^*,q_1q_2)$. If this set is non-empty for some $(G',L')\in \Sigma^{\flat}(G,L)$, the return ``$G$ admits an $L$-coloring,'' otherwise, return ``$G$ does not admit an $L$-coloring.''
		\end{enumerate}
		First we examine the correctness. If the algorithm stops at step \ref{s1} or \ref{s2}, then its correctness is trivial. Also, if it stops at step \ref{s4}, then the correctness follows from the second dash of the third bullet of Lemma \ref{checksuccess}. Moreover, if the algorithm stops at step \ref{s5}, then its correctness is implied by Lemma \ref{G+} combined with the second dash of the third bullet of Lemma \ref{checksuccess}.
		
		It remains to evaluate the running time of the algorithm. Step \ref{s1} runs in time $\mathcal{O}(|V(G)|^4)$. By the second bullet of Corollary \ref{kill1}, step \ref{s2} runs in time $\mathcal{O}(|V(G)|^{2w+4})$. By the first dash of the third bullet of Lemma \ref{checksuccess}, step \ref{s3} runs in time $\mathcal{O}(|V(G)|^{6w+2})$, and, having completed step \ref{s3}, step \ref{s4} runs in time $\mathcal{O}(|V(G)|^{6w+1})$. Finally, by the third dash of the third bullet of Lemma \ref{checksuccess}, step \ref{s5} runs in time $\mathcal{O}(|V(G)|^{162w^2+18})$, and so the entire algorithm runs in time $\mathcal{O}(|V(G)|^{162w^2+18})$. This completes the proof of Theorem \ref{mainwhale2}.
\end{proof}

\section{Excluding $J_{16}(k,l)$} \label{Sec-H(k,l)proof}
In this section, we present the polynomial-time algorithm claimed in Theorem \ref{shareanend}. 
\begin{theorem}\label{H(k,l)}
    For every fixed $k,l\in \mathbb{N}\cup \{0\}$, the \textsc{Ordered Graph List-3-Coloring Problem} restricted to $J_{16}(k,l)$-free ordered graphs can be solved in polynomial time. Consequently, for every $H\in \{J_{16}(k,l),-J_{16}(k,l)\}$, the \textsc{Ordered Graph List-3-Coloring Problem} restricted to $H$-free ordered graphs can be solved in polynomial time.
\end{theorem}
	
The main idea of the algorithm is the following. A graph $G$ is \emph{chordal} if in $G$, every cycle of length at least 4 has an edge connecting two vertices of the cycle but not in the cycle. Equivalently, every induced cycle in $G$ is a triangle.

An ordered graph is $J_{16}$-free if and only if it is chordal and its ordered is a \emph{perfect elimination ordering} of the chordal graph, that is, an ordering such that for each vertex, its forward neighbors are a clique. 

Our goal will be to ``guess'' the first $k$ and last $l$ vertices in $G$ of each color (Lemma \ref{fwdnbr}). After updating lists of their neighbors, we will be able to show that in the induced subgraph of vertices whose lists still contain at least two colors, each vertex has at most two forward neighbors (but they may not be adjacent). Now, within this subgraph, we guess the colors of the first $k$ vertices and their forward neighbors, as well as colors of the last $3l+6$ vertices (Lemma \ref{chordal}). This gives us enough ``padding'' on both ends of the ordering that in the ``middle'', where there are vertices whose colors are not yet determined, we do not have a copy of $J_{16}$; thus these vertices form a chordal graph (and with no clique of size four, for otherwise, a coloring does not exist). 
 
	There is an old known result derived from \cite{ChordalLCol} that the treewidth of a chordal graph can be computed in polynomial-time. Indeed, the treewidth of a chordal graph is bounded by its clique number minus 1, and what we compute is the clique number. 
	We also know that the \textsc{List-$k$-Coloring Problem} restricted to graphs with bounded treewidth is polynomial-time solvable with respect to the input size and the treewidth \cite{ChordalLCol1}. Since graphs with a clique of size $k+1$ do not have a list-$k$-coloring, it follows that:
	\begin{theorem} \label{thm:chordalL3Col}
		For every fixed $k$, the \textsc{List-$k$-Coloring Problem} restricted to chordal graphs is polynomial-time solvable.
	\end{theorem}

	Throughout this section, let $(G,L)$ be an instance of the \textsc{Ordered Graph List-3-Coloring Problem}.

	\begin{lemma}\label{updtL}
		There exists a spanning $(G,L)$-refinement $(G,L')$ such that for all $uv\in E(G)$ with $|L(v)|=1$, $L(u)\cap L(v)=\emptyset$, and $L$ and $L'$ are equivalent for $G$. Moreover, $L'$ can be computed from $L$ in time $\mathcal{O}(n^3)$.
	\end{lemma} 
	\begin{proof}
		We define a sequence of lists recursively. Let $L_0=L$. Suppose that we have defined $L_i$. If there is an edge $uv\in E(G)$ with $|L_i(v)|=1$ and $L_i(u)\cap L_i(v)\neq \emptyset$, let $L_{i+1}(u)=L_i(u)\backslash L_i(v)$, and $L_{i+1}(w)=L_i(w)$ for all $w\in V(G)\backslash \{u\}$. Otherwise stop and let $L'=L_i$.
		
		This terminates within at most $3n$ steps, as $\sum_{w\in V(G)} |L_{i+1}(w)|\leq \sum_{w\in V(G)} |L_{i}(w)|-1$, and $\sum_{w\in V(G)} |L_{0}(w)|\leq 3n$. In each step, finding an edge $uv\in E(G)$ with $|L(v)|=1$ and $L(u)\cap L(v)\neq \emptyset$ takes time at most $\mathcal{O}(n^2)$ and constructing a new list $L_{i+1}$ takes time $\mathcal{O}(n)$. Thus $L'$ can be computed from $L$ in time $\mathcal{O}(n^3)$.
		
		Since $L_0=L$, $G$ has an $L$-coloring if and only if $G$ has an $L_0$-coloring. For all $L_i$-colorings $c$ of $G$ and for all edges $uv\in E(G)$, $c(v)\in L_i(v)$ and $c(u)\neq c(v)$. Thus $c$ is an $L_{i+1}$-coloring of $G$. For all $L_{i+1}$-colorings $c'$ of $G$, since $L_{i+1}(w)\subseteq L_i(w)$ for all $w\in V(G)$, $c'$ is an $L_i$-coloring of $G$. Thus, $L$ and $L'$ are equivalent for $G$.
	\end{proof}
	
	\begin{lemma} \label{fwdnbr}
		Let $k,l\in \mathbb{N}$ be fixed, and $(G,L)$ be an instance of the \textsc{Ordered Graph List-3-Coloring Problem} restricted to $J_{16}(k,l)$-free ordered graphs. There is a spanning $(G,L)$-profile $\mathcal{L}'_1$ such that:
		\begin{itemize}
			\item $|\mathcal{L}'_1|\leq \mathcal{O}(n^{3(k+l)})$, and $\mathcal{L}'_1$ can be constructed from $L$ in time $\mathcal{O}(n^{3(k+l)+4})$.
			\item For all $(G,L') \in \mathcal{L}'_1$, let $X'=\{v\in V(G):|L'(v)|\geq 2\}$. Then in the graph $G[X']$, every vertex has at most 2 forward neighbors.
			\item If there is an $L$-coloring $c$ of $G$ with $|c^{-1}(i)|\geq k+l$ for all $i\in[3]$, then there exists $(G,L')\in \mathcal{L}'_1$ such that $c$ is an $L'$-coloring.
		\end{itemize}
	\end{lemma}
	\begin{proof}	
		Let $\mathcal{Q}$ be the set of all 6-tuples $Q=(A_1,A_2,A_3,B_1,B_2,B_3)$ of disjoint subsets of $V(G)$ such that for all $i\in[3]$, $|A_i|=k$ and $|B_i|=l$, $i\in L(v)$ for all $v\in A_i\cup B_i$, and $A_i\cup B_i$ are stable. For each $q\in \mathcal{Q}$, we construct a $(G,L)$-refinement $(G,L^Q)$ as follows. 
		
		The list $L^Q_0$ is defined as follows. For each vertex $v\in V(G)$, we let $L_0^{'Q}(v)=\{i\}$ if $v\in A_i \cup B_i$ for some $i\in [3]$, otherwise let $L^{'Q}_0(v)=L(v)$. For each $i\in [3]$, let $m_i=\max \{\varphi_G(a):a\in  A_i \}$, $n_i=\min \{\varphi_G(b):b\in  B_i \}$. Then, for all $i\in[3]$, remove $i$ from $L_0^{'Q}(v)$ for every $v\in V(G(-\infty:m_i])\backslash A_i$ and every $v\in V(G[n_i:\infty))\backslash B_i$. By Lemma \ref{updtL}, the list $L_0^Q$ such that for all $uv\in E(G)$ with $|L_0^Q(v)|=1$, $L_0^Q(u)\cap L_0^Q(v)=\emptyset$, can be constructed from $L^{'Q}_0$ in polynomial time.
		
		The list $L^Q$ is constructed recursively. Starting from the list $L_0^Q$, we construct a sequence of equivalent list assignments $L_1^Q,L_2^Q,\ldots$ until some $L_s^Q$ satisfies the second property of this lemma. For convenience, every time we define $L_t^Q$ for $0\leq t\leq s$, we also define the following sets. For $\{i,j\}\subseteq [3]$, let $X^{ij}_t=\{v\in V(G):L^Q_t(v)=\{i,j\}\}$, and let $X^{123}_t=\{v\in V(G):L^Q_t(v)=\{1,2,3\}\}$. Let $X_t= X^{12}_t\cup X^{13}_t\cup X^{23}_t\cup X^{123}_t$.
		
        If in the graph $G[X_t]$, every vertex has at most 2 forward neighbors, then let $L^Q=L_t^Q$. Otherwise, there is a vertex $v$ with at least 3 forward neighbors in the graph $G[X_t]$. Notice that if $G$ contains $K_4$ as a subgraph, then $G$ is not $L'$-colorable for any $L':V(G)\rightarrow 2^{[3]}$. We return $\mathcal{L}'_1=\emptyset$ in this case. Thus, we may assume that $G$ contains no $K_4$ from now on. 
		
		Also, we may assume that 
		
		\tbox{\label{noP3} The vertex $v$ does not have two distinct nonadjacent forward neighbors $u,w$ such that $L_t^Q(v)\cap L_t^Q(u)\cap L_t^Q(w)\neq \emptyset$. }
		
		Suppose not. Consider a color $i\in L_t^Q(v)\cap L_t^Q(u)\cap L_t^Q(w)$ for two distinct nonadjacent forward neighbors $u,w$ of $v$. From the construction of $L_t^Q$, we know that $v,u,w$ are not adjacent to any vertex from $A_i\cup B_i$. $A_i$ and $B_i$ are disjoint and $A_i\cup B_i$ is a stable set. The vertices $u,w$ are nonadjacent forward neighbors of $v$. For every $x\in A_i$, $y\in \{u,v,w\}$ and $z\in B_i$,  $\varphi(x) <\varphi(y)< \varphi(z)$. From the constructions of $L_0^Q$, we have $G[A_i\cup B_i\cup \{u,v,w\}]\cong J_{16}(k,l)$, which contradicts the fact that $G$ is $J_{16}(k,l)$-free. This proves \eqref{noP3}.
		
		\medskip 
		If $v\in X^{123}_t$, then for every two forward neighbors $u,w$ of $v$ in $X_t$, $L_t^Q(v)\cap L_t^Q(u)\cap L_t^Q(w)\neq \emptyset$. So by \eqref{noP3}, $u$ and $w$ are adjacent. But then, since $v$ has at least 3 forward neighbors, there exists a $K_4$ as a subgraph of $G$, which is a contradiction. Thus, this case is impossible.
		
		It remains to consider the case $v\in X^{ij}_t$. 
        Since $G$ has no $K_4$, there are two forward neighbors $u,w$ of $v$ in $X_t$ with $uw\notin E$. By \eqref{noP3}, it follows that $L_t^Q(u) \cap L_t^Q(v) \cap L_t^Q(w)= \emptyset$. By symmetry, we may assume that $L_t^Q(u)=\{i,m\}$, $L_t^Q(w)=\{j,m\}$ where $\{i,j,m\}=[3]$. Let $x$ be a forward neighbor of $v$ in $X_t$ different from $u$ and $w$. We have the following subcases.
		\begin{itemize}
			\item $L_t^Q(x)\supseteq \{i,j\}$.
			
			Then $ux, wx \in E(G)$ by \eqref{noP3}. Since $u$ and $w$ have two adjacent neighbors in common, it follows that $c(u)=c(w)$ for every 3-coloring of $G$. We let $L_{t+1}^{'Q}(u)= L_{t+1}^{'Q}(w)=\{m\}$, $L^{'Q}_{t+1}(y)=L^Q_{t}(y)\backslash \{m\}$ for $y\in N(u)\cup N(w)$, and $L^{'Q}_{t+1}(y)=L^Q_{t}(y)$ for all $y\in V(G)\backslash (N[u]\cup N[w])$.
			
			\item $L_t^Q(x)=\{i,m\}$. (The case $\{j,m\}$ follows from symmetry.)
			
			By \eqref{noP3}, we have $ux\in E(G)$. But now $v$ has two adjacent neighbors with list $\{i,m\}$. So in every $L_t^Q$-coloring $c$ of 
			$G$, we have $c(v)=j$. We let $L^{'Q}_{t+1}(v)=\{j\}$, $L^{'Q}_{t+1}(y)= L_{t}(y)\backslash \{j\}$ for $y\in N(v)$, and $L^{'Q}_{t+1}(y)=L_{t}(y)$ for all $y\in V(G)\backslash N[v]$.
		\end{itemize}
		
		At the end of each step, by applying Lemma \ref{updtL}, we replace the list $L^{'Q}_{t+1}$ by an equivalent list $L^Q_{t+1}$ such that for all $uv\in E(G)$ with $|L^Q_{t+1}(v)|=1$, we have $L^Q_{t+1}(u)\cap L^Q_{t+1}(v)=\emptyset$, in time $\mathcal{O}(n^3)$.
		
		For all $t$, $|X_{t+1}|\leq |X_t|-1$, and $|X_0|\leq n$. Thus the algorithm above terminates in at most $n$ steps. In each step $t$, finding the vertex $v$ with at least $3$ forward neighbors in $G[X_t]$ takes time $\mathcal{O}(n)$, constructing the list $L^{'Q}_{t+1}$ takes time $\mathcal{O}(n)$, and constructing the list $L^{Q}_{t+1}$ takes time $\mathcal{O}(n^3)$. So $L^{Q}$ can be constructed in time $\mathcal{O}(n^4)$. We let $\mathcal{L}'_1=\{L^Q: Q\in \mathcal{Q}\}$. There are at most ${n\choose k}\cdot {n-k\choose k}\cdot {n-2k\choose k}=\mathcal{O}(n^{3k})$ different choices of the triple ($A_1,A_2,A_3$), and at most $\mathcal{O}(n^{3l})$ different choices of the triple $(B_1,B_2,B_3)$. For each 6-tuple $Q$ we add at most one list to $\mathcal{L}'_1$. Thus, $|\mathcal{L}'_1|\leq \mathcal{O}(n^{3(k+l)})$. Therefore, $\mathcal{L}_1$ can be constructed from $L$ in time $\mathcal{O}(n^{3(k+l)+4})$.		
		
		Finally, let $c$ be an $L$-coloring of $G$ with $|c^{-1}(i)|\geq k+l$ for all $i\in[3]$. Define $A'_i\subseteq c^{-1}(i)$ to be the set of vertices such that $|A'_i|=k$, and $\varphi(v)>\varphi(u)$ for all $v\in c^{-1}(i)\backslash A'_i$ and $u\in A'_i$, that is, $A'_i$ is the set of first $k$ vertices colored $i$ in $c$. Similarly, for all $i\in [3]$, define $B'_i\subseteq c^{-1}(i)$ to be the set of vertices such that $|B'_i|=l$, and $\varphi(v)<\varphi(u)$ for any $v\in c^{-1}(i)\backslash B'_i$ and $u\in B'_i$. Let $Q'=(A'_1, A'_2, A'_3, B'_1, B'_2, B'_3)$. It follows that $Q'\in \mathcal{Q}$. Thus, the corresponding $(G,L)$-refinement $(G,L^{Q'})$ is in $\mathcal{L}'_1$.
		
		We want to show that $c$ is also an $L^{Q'}$-coloring. We will prove this by induction on $t$. For every vertex $v\in V(G)$, we have $c(v)\in L^{Q'}_0(v)$ from the choice of $Q'$. Thus, $c$ is an $L^{Q'}_0$-coloring. Suppose $c$ is an $L^{Q'}_t$-coloring. Then for $t+1$, from our construction, $c(v)\in L^{'Q'}_{t+1}(v)$ for all vertex $v$. So $c$ is an $L^{'Q'}_{t+1}$-coloring of $G$. By Lemma \ref{updtL}, $c$ is also an $L^{Q'}_{t+1}$-coloring. Thus, the $L$-coloring $c$ is an $L^{Q'}$-coloring of $G$.
	\end{proof}
	
	\begin{lemma} \label{fwdnbrspecial}
		Let $k,l\in \mathbb{N}$ be fixed, and $(G,L)$ be an instance of the \textsc{Ordered Graph List-3-Coloring Problem} restricted to $J_{16}(k,l)$-free ordered graphs. There is a spanning $(G,L)$-profile $\mathcal{L}'_2$ such that:
		\begin{itemize}
			\item $|\mathcal{L}'_2|\leq 3\cdot n^{k+l}$, and $\mathcal{L}'_2$ can be constructed from $L$ in time $\mathcal{O}(n^{k+l+1})$.
			\item For all $(G,L') \in \mathcal{L}'_2$, let $X=\{v\in V(G):|L'(v)|\geq 2\}$. Then $|L'(v)|=2$ and $L'(u)=L'(v)$ for all $u,v\in X$. 
			\item If $c$ is an $L$-coloring of $G$ with $|c^{-1}(i)|< k+l$ for some $i\in [3]$, then there exists $(G,L')\in \mathcal{L}'_2$ such that $c$ is an $L'$-coloring.
		\end{itemize}
	\end{lemma}
	\begin{proof}
		Let $\mathcal{P}$ be a set of all pairs $P=(i,A_i)$ such that $i\in [3]$ and $A_i\subseteq V(G)$ with $|A_i|< k+l$, $A_i$ stable and $i\in L(v)$ for all $v\in A_i$.  For each $P\in \mathcal{P}$, we construct a $(G,L)$-refinement $(G,L^P)$ as follows. 
		
		Let $L^{P}(v)=\{i\}$ for all $v\in A_i$, and $L^{P}(v)=L(v)\backslash \{i\}$ otherwise. It follows that $L^{P}(v)=[3]\backslash \{i\}$ for all $v\in V(G)$ with $|L^P(v)|\geq 2$.
		
		The set $\mathcal{P}$ is of size at most $3\cdot n^{k+l}$. For each pair $P\in \mathcal{P}$ we add at most one refinement to $\mathcal{L}'_2$. Thus, $|\mathcal{L}'_2|\leq 3\cdot n^{k+l}$. Constructing the list $L^{P}$ takes time $\mathcal{O}(n)$. Thus, $\mathcal{L}'_2$ can be constructed from $L$ in time $\mathcal{O}(n^{k+l+1})$.
		
		Let $c$ be an $L$-coloring of $G$ with $|c^{-1}(i)|< k+l$ for some $i\in [3]$. The pair $P'=(i,c^{-1}(i))$ satisfies the property that $|c^{-1}(i)|< k+l$, $c^{-1}(i)$ is stable and $i\in L(v)$ for all $v\in c^{-1}(i)$. Thus, the corresponding $(G,L)$-refinement $(G,L^{P'})$ is in $\mathcal{L}'_2$. By the construction of $L^{P'}$, $c$ is an $L^{P'}$-coloring. 
	\end{proof}

	\begin{lemma}\label{chordal}
		Let $k,l\in \mathbb{N}$ be fixed, and $(G,L)$ be an instance of the \textsc{Ordered Graph List-3-Coloring Problem} restricted to $J_{16}(k,l)$-free ordered graphs. Let $X=\{v\in V(G):|L(v)|\geq 2\}$ and let us assume that every vertex in $X$ has at most two forward neighbors in $G[X]$ and that $|X|\geq 3k+3l+6$. There is a spanning $(G,L)$-profile $\mathcal{L}_1$ such that:
		\begin{itemize}
			\item $|\mathcal{L}_1|= \mathcal{O}(1)$, and  $\mathcal{L}_1$ can be constructed in time $\mathcal{O}(n^3)$.
			\item For all $(G,L^*)\in \mathcal{L}_1$, let $X^*=\{v\in V(G):|L^*(v)|\geq 2\}$. Then the graph $G[X^*]$ is chordal.
			\item If $c$ is an $L$-coloring of $G$, then there exists $(G,L^*)\in \mathcal{L}_1$ such that $c$ is an $L^*$-coloring of $G$.
		\end{itemize}
	\end{lemma}
	\begin{proof}
		First, we define two sets $C'\subseteq C\subseteq X$ as follows. We start with $C'=C=\emptyset$. In each step, we take the vertex $v\in X\backslash C$ with the smallest $\varphi(v)$. Add $v$ and its forward neighbors in $G[X]$ to $C$, and add $v$ to $C'$. We repeat this $k$ times. Since every vertex in $X$ has at most two forward neighbors in $G[X]$, $|C|\leq 3k$. By construction, $C'$ is a stable set of size $k$. Moreover, no vertex in $C'$ is adjacent to a vertex in $X\backslash C$. Define $D\subseteq X$ to be the set of vertices such that $|D|=3l+6$, and $\varphi(v)<\varphi(u)$ for all $v\in X\backslash (C\cup D)$ and $u\in D$, that is, $D$ is the set of last $3l+6$ vertices in $X\backslash C$. Since $|X|\geq 3k+3l+6$, it follows that $C$, $C'$ and $D$ are well-defined.
		
		Let $\mathcal{F}$ be the set of all functions $f:C\cup D\rightarrow [3]$ such that $f$ is an $L$-coloring of $G[C\cup D]$. For every $f\in \mathcal{F}$, we construct a $(G,L)$-refinement $(G,L^{'f})$ such that $L^{'f}(v)=\{f(v)\}$ if $v\in C\cup D$, and $L^{'f}(v)=L(v)$ otherwise. By Lemma \ref{updtL}, there is an equivalent list $L^{f}$ of $L^{'f}$ such that for all $uv\in E(G)$ with $|L^{f}(v)|=1$, $L^{f}(u)\cap L^{f}(v)=\emptyset$. Let $\mathcal{L}_1=\{ (G,L^{f}): f\in \mathcal{F}\}$. There are at most $3^{3k+3l+6}=\mathcal{O}(1)$ possible choices of $f$. Thus, $|\mathcal{L}_1|=\mathcal{O}(1)$. Constructing the set $C$ and $D$ takes time $\mathcal{O}(1)$. Each $L^{'f}$ can be constructed in time $\mathcal{O}(n)$. Each $L^{f}$ can be constructed in time $\mathcal{O}(n^3)$. So $\mathcal{L}_1$ can be constructed in time $\mathcal{O}(n^3)$.
		
		Now let $(G,L^*)\in \mathcal{L}_1$. Every non-chordal ordered graph contains a vertex with two nonadjacent forward neighbors. To be more precise, the vertex with the smallest order in an induced cycle of size at least 4 is a desired vertex. Now we want to show that $G[X^*]$ is chordal using this property. Suppose for a contradiction that in $G[X^*]$, there is a vertex $v_1$ with two nonadjacent forward neighbors $v_2,v_3$. There is a stable set $D'\subseteq D$ of size at least $l$ such that $D'$ is anticomplete to $\{v_1,v_2,v_3\}$. That is because $X^*\subseteq X$ and every vertex in $X$ has at most 2 forward neighbors in $G[X]$, so $D \backslash N(\{v_1,v_2,v_3\})$ is of size at least $3l$. Since $D$ has a 3-coloring by construction, there is a stable set $D'\subseteq D \backslash N(\{v_1,v_2,v_3\})$ of size at least $l$ and which is anticomplete to $\{v_1,v_2,v_3\}$. From the construction above, the sets $\{v_1,v_2,v_3\}$, $C'$ and $D'$ are disjoint. Moreover, for every $x\in C'$, $y\in \{v_1,v_2,v_3\}$ and $z\in D'$,  $\varphi(x) <\varphi(y)< \varphi(z)$. So $G[C'\cup D'\cup \{v_1,v_2,v_3\}]\cong J_{16}(k,l)$, which is a contradiction. Therefore, $G[X^*]$ is chordal.
		
		Finally, let $c$ be an $L$-coloring of $G$. Take the coloring $c'=c|_{C\cup D}$ and consider the corresponding $(G,L)$-refinement $(G,L^{'c'})$ and $(G,L^{c'})$ defined above. Since we have covered all possible colorings $f$ of $G[C\cup D]$, $(G,L^{c'})$ is in $\mathcal{L}_1$. We can verify that $c(v)\in L^{'c'}(v)$ for all vertices $v\in V(G)$. Thus $c$ is also an $L^{c'}$-coloring.
	\end{proof}
	
	\begin{lemma} \label{chordalspecial}
		Let $k,l\in \mathbb{N}$ be fixed, and $(G,L)$ be an instance of the \textsc{Ordered Graph List-3-Coloring Problem} restricted to $J_{16}(k,l)$-free ordered graphs. Let $X=\{v\in V(G):|L(v)|\geq 2\}$ and let us assume that $|X| < 3k+3l+6$. There is a spanning $(G,L)$-profile $\mathcal{L}_2$ such that:
		\begin{itemize}
			\item $|\mathcal{L}_2|= \mathcal{O}(1)$, and $\mathcal{L}_2$ can be constructed in time $\mathcal{O}(n^3)$.
			\item For any $(G,L^*) \in \mathcal{L}_2$, $|L^*(v)|\leq 1$ for all $v\in V(G)$. 
			\item If $c$ is an $L$-coloring of $G$, then there exists $(G,L^*)\in \mathcal{L}_2$ such that $c$ is an $L^*$-coloring of $G$.
		\end{itemize}
	\end{lemma}
	\begin{proof}
		Let $\mathcal{F}$ be the set of all functions $f:X\rightarrow [3]$ such that $f$ is an $L$-coloring of $G[X]$. For every possible function $f\in \mathcal{F}$, we construct a list $L^{f}$ such that $L^{f}(v)=\{f(v)\}$ for all $v\in X$, and $L^{f}(v)=L(v)$ otherwise. Let $\mathcal{L}_2=\{(G,L^{f}): f\in \mathcal{F}\}$.
		
		For every $(G,L^f) \in \mathcal{L}_2$ and for every $v\in V(G)$, if $v\in X$ then $|L^f(v)|\leq 1$; otherwise by the definition of $X$, we have $|L^f(v)|\leq |L^{'f}(v)|\leq 1$. Thus $|L^f(v)|\leq 1$ for all $v\in V(G)$.
		
		Since there are at most $3^{3k+3l+6} = \mathcal{O}(1)$ possible choices of $f$, $|\mathcal{L}_2|= \mathcal{O}(1)$. Each $L^{'f}$ can be constructed in time $\mathcal{O}(n)$, and $L^{f}$ can be constructed from $L^{'f}$ in time $\mathcal{O}(n^3)$. So $\mathcal{L}_2$ can be constructed in time $\mathcal{O}(n^3)$. 
		
		Finally, let $c$ be an $L$-coloring of $G$. Let $c'=c|_{X}$ and consider the corresponding $(G,L)$-refinements $(G,L^{'c'})$ and $(G,L^{c'})$ defined above. Since we have covered all possible $L$-colorings $f:X\rightarrow [3]$, $(G,L^{c'})\in \mathcal{L}_2$. By the construction of $c'$ and $L^{'c'}$, $c$ is an $L^{'c'}$-coloring thus is an $L^{c'}$-coloring.
	\end{proof}

	\begin{theorem} \label{H(k,l)Alg}
		For fixed $k,l\in \mathbb{N}$, there is an algorithm with the following specifications:
		\begin{itemize}
			\item Input: $(G,L)$, which is an instance of the \textsc{Ordered Graph List-3-Coloring Problem} and $G$ is $J_{16}(k,l)$-free.
			\item Output: one of
			\begin{itemize}
				\item an $L$-coloring of $G$;
				\item a determination that $G$ is not $L$-colorable;
				\item a spanning $(G,L)$-profile $\mathcal{L}$ with $|\mathcal{L}|\leq \mathcal{O}(n^{3(k+l)})$ such that for every $(G,L^*)\in \mathcal{L}$, if $X^{L^*}=\{v\in V(G):|L^*(v)|\geq 2\}$, then $G[X^{L^*}]$ is chordal.
			\end{itemize}
			\item Running time: $\mathcal{O}(n^{3(k+l+1)})$.
		\end{itemize}
	\end{theorem}
	\begin{proof}
		Let $\mathcal{L}'_1$ be as in Lemma \ref{fwdnbr}. Let $\mathcal{L}'_2$ be as in Lemma \ref{fwdnbrspecial}. By Theorem \ref{2LC}, every $(G,L)$-refinement $(G,L')\in \mathcal{L}'_2$ can be solved in time $\mathcal{O}(n^2)$. If this finds an $L$-coloring of $G$, we just output the coloring instead of processing with the other things.
		
		For every $(G,L)$-refinement $(G,L')\in \mathcal{L}'_1$, let $X=\{v\in V(G):|L'(v)|\geq 2\}$. If $|X|\geq 3k+3l+6$, then there is a spanning $(G,L)$-profile $\mathcal{L}^{L'}$ which satisfies the properties in Lemma \ref{chordal}. If $|X|< 3k+3l+6$, then there is a spanning $(G,L)$-profile $\mathcal{L}^{L'}$ which satisfies the properties in Lemma \ref{chordalspecial}. Finally, let $\mathcal{L}=\cup_{(G,L')\in \mathcal{L}'_1} \mathcal{L}^{L'}$. From the constructions, for every $(G,L^*)\in \mathcal{L}$, $G[X^{L^*}]$ is chordal. Since $|\mathcal{L}^{L'}|=\mathcal{O}(1)$ and $|\mathcal{L}'_1|\leq \mathcal{O}(n^{3(k+l)})$,  $|\mathcal{L}|\leq \mathcal{O}(n^{3(k+l)})$. The collection $\mathcal{L}^{L'}$ can be constructed from $L'$ in time $\mathcal{O}(n^3)$, and $|\mathcal{L}'_1|\leq \mathcal{O}(n^{3(k+l)})$. Thus, $\mathcal{L}$ can be constructed from $L$ in time $\mathcal{O}(n^{3(k+l+1)})$. 
	\end{proof}

  	Now we are ready to prove Theorem \ref{H(k,l)}.
   	\begin{proof}[Proof of Theorem \ref{H(k,l)}]
   		Let $G$ be a $J_{16}(k,l)$-free ordered graph and $L$ be a 3-list-assignment for $G$. 
     We apply the algorithm from Theorem \ref{H(k,l)Alg} to $(G, L)$. If the output is an $L$-coloring of G or a determination that G is not $L$-colorable, then we are done; so we may assume that the output is a $(G, L)$-profile $\mathcal{L}$. For each $(G, L^*) \in \mathcal{L}$, we let $X^{L^*}=\{v\in V(G):|L^*(v)|\geq 2\}$ as in Theorem \ref{H(k,l)Alg}. By Theorem \ref{updtL}, we may assume that $L^*(u) \cap L^*(v) = \emptyset$ for all $uv \in E(G)$ such that $|L^*(u)| = 1$. If $L^*(u) = \emptyset$ for some $u \in V(G)$, then $G$ has no $L^*$-coloring and we continue. Otherwise, since Theorem \ref{H(k,l)Alg} guarantees that $G[X^{L^*}]$ is chordal, and by Theorem \ref{thm:chordalL3Col}, we can check in polynomial time if $G[X^{L^*}]$ is $L^*$-colorable. If this returns a coloring $f$, then by Theorem \ref{updtL}, we obtain an $L$-coloring of $G$ as follows:
   		\begin{itemize}
   			\item for $x \in X^{L^*}$, let $c(x) = f(x)$;
   			\item for all other $x \in V(G)$, let $c(x)$ be the unique color in $L^*(x)$.
   		\end{itemize}
   		If there is no $(G, L^*) \in \mathcal{L}$ such that this returns a coloring of $G$, then, from the definition of a $(G, L)$-profile, it follow that $G$ is not $L$-colorable. This concludes the proof.
   	\end{proof}

	\section{\textsf{NP}-completeness results} \label{Sec-OrderedNP}
	Let us begin with giving the full list of all ordered graphs we will use. Let $U'=\{u_1, u_2, u_3, u_4, u_5\}$ and $U=U'\backslash\{u_5\}$, the ordering $\varphi':U'\rightarrow \mathbb{R}$ with $u_i\mapsto i$ for $i\in [5]$.
	\begin{itemize}
		\item Let $J_1=(U,\{u_1u_2,u_2u_3,u_3u_4\}, \varphi'|_U)$.
		\item Let $J_2=(U,\{u_1u_2,u_2u_4,u_3u_4\}, \varphi'|_U)$.
		\item Let $J_3=(U,\{u_1u_3,u_2u_3,u_2u_4\}, \varphi'|_U)$.
		\item Let $J_4=(U,\{u_1u_3,u_2u_4,u_3u_4\}, \varphi'|_U)$.
		\item Let $J_5=(U,\{u_1u_4,u_2u_3,u_2u_4\}, \varphi'|_U)$.
		\item Let $J_6=(U,\{u_1u_4,u_2u_3,u_3u_4\}, \varphi'|_U)$.
		\item Let $J_7=(U,\{u_1u_2,u_1u_4,u_3u_4\}, \varphi'|_U)$.
		\item Let $J_8=(U,\{u_1u_3,u_1u_4,u_2u_4\}, \varphi'|_U)$.
		\item Let $J_9=(U,\{u_1u_2,u_3u_4\}, \varphi'|_U)$.
		\item Let $J_{10}=(U,\{u_1u_2,u_1u_4\}, \varphi'|_U)$.
		\item Let $J_{11}=(U,\{u_1u_3,u_1u_4\}, \varphi'|_U)$.
		\item Let $J_{12}=(U,\{u_1u_2,u_2u_4\}, \varphi'|_U)$.
		\item Let $J_{13}=(U',\{u_1u_5,u_2u_3, u_3u_4\}, \varphi')$.
		\item Let $J_{14}=(U',\{u_1u_5,u_2u_3, u_2u_4\}, \varphi')$.
		\item Let $J_{15}=(U\backslash \{u_4\}, \{u_1u_2,u_2u_3\}, \varphi'|_{U\backslash \{u_4\}})$.
		\item Let $J_{16}=(U\backslash \{u_4\}, \{u_1u_2,u_1u_3\}, \varphi'|_{U\backslash \{u_4\}})$.
	\end{itemize}
	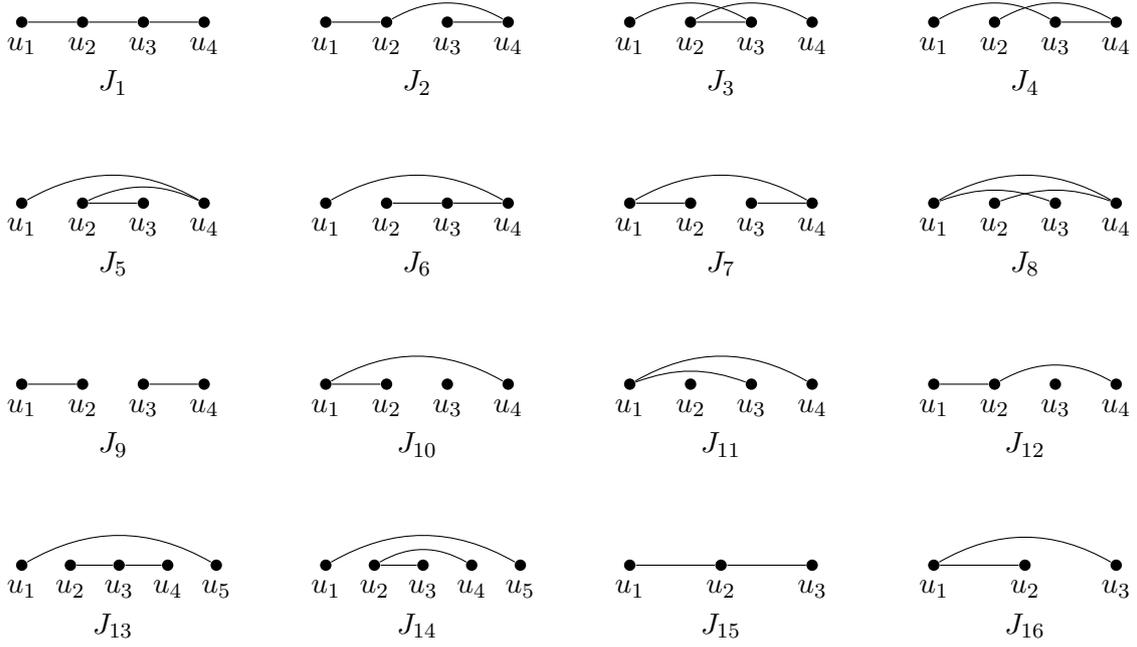
\begin{figure}[t]
		\centering 
					\begin{tikzpicture}[scale=0.8]
		\node [draw=none,fill=none] at (1.5,9){$J_1$};
		\node [label=below: $u_1$] (u1) at (0,10){};
		\node [label=below: $u_2$] (u2) at (1,10){};
		\node [label=below: $u_3$] (u3) at (2,10){};
		\node [label=below: $u_4$] (u4) at (3,10){};
		\draw (u1) edge [bend left = 0]  (u2);
		\draw (u2) edge [bend left = 0]  (u3);
		\draw (u3) edge [bend left = 0]  (u4);
		
		\node [draw=none,fill=none] at (6.5,9){$J_2$};
		\node [label=below: $u_1$] (u1) at (5,10){};
		\node [label=below: $u_2$] (u2) at (6,10){};
		\node [label=below: $u_3$] (u3) at (7,10){};
		\node [label=below: $u_4$] (u4) at (8,10){};
		\draw (u1) edge [bend left = 0]  (u2);
		\draw (u2) edge [bend left = 30]  (u4);
		\draw (u3) edge [bend left = 0]  (u4);
		
		\node [draw=none,fill=none] at (11.5,9){$J_3$};
		\node [label=below: $u_1$] (u1) at (10,10){};
		\node [label=below: $u_2$] (u2) at (11,10){};
		\node [label=below: $u_3$] (u3) at (12,10){};
		\node [label=below: $u_4$] (u4) at (13,10){};
		\draw (u1) edge [bend left = 30]  (u3);
		\draw (u2) edge [bend left = 0]  (u3);
		\draw (u2) edge [bend left = 30]  (u4);
		
		\node [draw=none,fill=none] at (16.5,9){$J_4$};
		\node [label=below: $u_1$] (u1) at (15,10){};
		\node [label=below: $u_2$] (u2) at (16,10){};
		\node [label=below: $u_3$] (u3) at (17,10){};
		\node [label=below: $u_4$] (u4) at (18,10){};
		\draw (u1) edge [bend left = 30]  (u3);
		\draw (u2) edge [bend left = 30]  (u4);
		\draw (u3) edge [bend left = 0]  (u4);
		
		\node [draw=none,fill=none] at (1.5,6){$J_5$};
		\node [label=below: $u_1$] (u1) at (0,7){};
		\node [label=below: $u_2$] (u2) at (1,7){};
		\node [label=below: $u_3$] (u3) at (2,7){};
		\node [label=below: $u_4$] (u4) at (3,7){};
		\draw (u1) edge [bend left = 30]  (u4);
		\draw (u2) edge [bend left = 0]  (u3);
		\draw (u2) edge [bend left = 25]  (u4);
		
		\node [draw=none,fill=none] at (6.5,6){$J_6$};
		\node [label=below: $u_1$] (u1) at (5,7){};
		\node [label=below: $u_2$] (u2) at (6,7){};
		\node [label=below: $u_3$] (u3) at (7,7){};
		\node [label=below: $u_4$] (u4) at (8,7){};
		\draw (u1) edge [bend left = 30]  (u4);
		\draw (u2) edge [bend left = 0]  (u3);
		\draw (u3) edge [bend left = 0]  (u4);
		
		\node [draw=none,fill=none] at (11.5,6){$J_7$};
		\node [label=below: $u_1$] (u1) at (10,7){};
		\node [label=below: $u_2$] (u2) at (11,7){};
		\node [label=below: $u_3$] (u3) at (12,7){};
		\node [label=below: $u_4$] (u4) at (13,7){};
		\draw (u1) edge [bend left = 0]  (u2);
		\draw (u1) edge [bend left = 30]  (u4);
		\draw (u3) edge [bend left = 0]  (u4);
		
		\node [draw=none,fill=none] at (16.5,6){$J_8$};
		\node [label=below: $u_1$] (u1) at (15,7){};
		\node [label=below: $u_2$] (u2) at (16,7){};
		\node [label=below: $u_3$] (u3) at (17,7){};
		\node [label=below: $u_4$] (u4) at (18,7){};
		\draw (u1) edge [bend left = 20]  (u3);
		\draw (u1) edge [bend left = 30]  (u4);
		\draw (u2) edge [bend left = 20]  (u4);
		
		\node [draw=none,fill=none] at (1.5,3){$J_9$};
		\node [label=below: $u_1$] (u1) at (0,4){};
		\node [label=below: $u_2$] (u2) at (1,4){};
		\node [label=below: $u_3$] (u3) at (2,4){};
		\node [label=below: $u_4$] (u4) at (3,4){};
		\draw (u1) edge [bend left = 0]  (u2);
		\draw (u3) edge [bend left = 0]  (u4);
		
		\node [draw=none,fill=none] at (6.5,3){$J_{10}$};
		\node [label=below: $u_1$] (u1) at (5,4){};
		\node [label=below: $u_2$] (u2) at (6,4){};
		\node [label=below: $u_3$] (u3) at (7,4){};
		\node [label=below: $u_4$] (u4) at (8,4){};
		\draw (u1) edge [bend left = 0]  (u2);
		\draw (u1) edge [bend left = 30]  (u4);
		
		\node [draw=none,fill=none] at (11.5,3){$J_{11}$};
		\node [label=below: $u_1$] (u1) at (10,4){};
		\node [label=below: $u_2$] (u2) at (11,4){};
		\node [label=below: $u_3$] (u3) at (12,4){};
		\node [label=below: $u_4$] (u4) at (13,4){};
		\draw (u1) edge [bend left = 20]  (u3);
		\draw (u1) edge [bend left = 30]  (u4);
		
		\node [draw=none,fill=none] at (16.5,3){$J_{12}$};
		\node [label=below: $u_1$] (u1) at (15,4){};
		\node [label=below: $u_2$] (u2) at (16,4){};
		\node [label=below: $u_3$] (u3) at (17,4){};
		\node [label=below: $u_4$] (u4) at (18,4){};
		\draw (u1) edge [bend left = 0]  (u2);
		\draw (u2) edge [bend left = 30]  (u4);
		
		\node [draw=none,fill=none] at (1.5,0){$J_{13}$};
		\node [label=below: $u_1$] (u1) at (0,1){};
		\node [label=below: $u_2$] (u2) at (0.8,1){};
		\node [label=below: $u_3$] (u3) at (1.6,1){};
		\node [label=below: $u_4$] (u4) at (2.4,1){};
		\node [label=below: $u_5$] (u5) at (3.2,1){};
		\draw (u1) edge [bend left = 30]  (u5);
		\draw (u2) edge [bend left = 0]  (u3);
		\draw (u3) edge [bend left = 0]  (u4);
		
		\node [draw=none,fill=none] at (6.5,0){$J_{14}$};
		\node [label=below: $u_1$] (u1) at (5,1){};
		\node [label=below: $u_2$] (u2) at (5.8,1){};
		\node [label=below: $u_3$] (u3) at (6.6,1){};
		\node [label=below: $u_4$] (u4) at (7.4,1){};
		\node [label=below: $u_5$] (u5) at (8.2,1){};
		\draw (u1) edge [bend left = 30]  (u5);
		\draw (u2) edge [bend left = 30]  (u4);
		\draw (u2) edge [bend left = 0]  (u3);
		
		\node [draw=none,fill=none] at (11.5,0){$J_{15}$};
		\node [label=below: $u_1$] (u1) at (10,1){};
		\node [label=below: $u_2$] (u2) at (11.5,1){};
		\node [label=below: $u_3$] (u3) at (13,1){};
		\draw (u1) edge [bend left = 0]  (u2);
		\draw (u2) edge [bend left = 0]  (u3);
		
		\node [draw=none,fill=none] at (16.5,0){$J_{16}$};
		\node [label=below: $u_1$] (u1) at (15,1){};
		\node [label=below: $u_2$] (u2) at (16.5,1){};
		\node [label=below: $u_3$] (u3) at (18,1){};
		\draw (u1) edge [bend left = 30]  (u3);
		\draw (u1) edge [bend left = 0]  (u2);
	\end{tikzpicture}
		
		\caption{The ordered graphs $J_i$ for $i\in [16]$.}\label{J}
	\end{figure}
	
	Let $V=\{v_1,v_2,v_3,v_4,v_5,v_6\}$ and $\varphi:V\rightarrow \mathbb{R}$ with $v_i\mapsto i$ for $i\in [6]$.
	\begin{itemize}
		\item Let $M_1=(V,\{v_1v_6,v_2v_5\},\varphi)$.
		\item Let $M_2=(V,\{v_1v_6,v_2v_5,v_3v_4\},\varphi)$.
		\item Let $M_3=(V,\{v_1v_4,v_2v_5,v_3v_6\},\varphi)$.
		\item Let $M_4=(V,\{v_1v_5,v_2v_4,v_3v_6\},\varphi)$.
		\item Let $M_5=(V\backslash \{v_6\},\{v_1v_5,v_2v_3\},\varphi|_{V\backslash \{v_6\}})$.
		\item Let $M_6=(V\backslash \{v_5,v_6\},\{v_1v_3,v_2v_4\},\varphi|_{V\backslash \{v_5,v_6\}})$.
		\item Let $M_7=(V\backslash \{v_5,v_6\},\{v_1v_4,v_2v_3\},\varphi|_{V\backslash \{v_5,v_6\}})$.
		\item Let $M_8=(V\backslash \{v_6\},\{v_1v_5,v_2v_4\},\varphi|_{V\backslash \{v_6\}})$.			
	\end{itemize}
	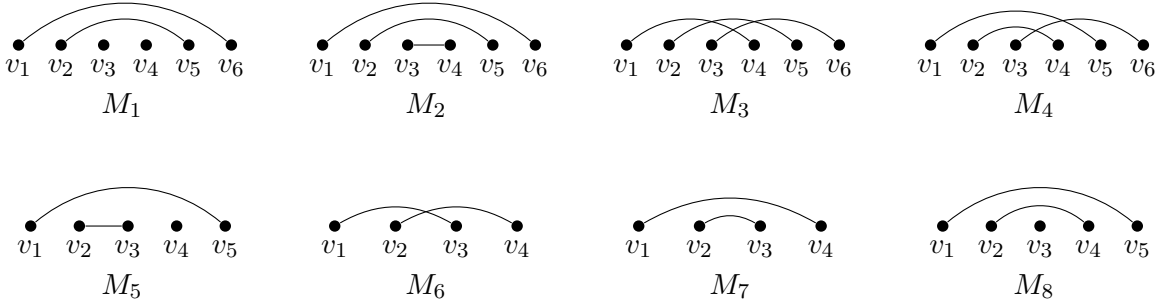
\begin{figure}[t]
		\centering 
				\begin{tikzpicture}[scale=0.8]
		
		\node [draw=none,fill=none] at (1.5,3){$M_1$};
		\node [label=below: $v_1$] (u1) at (-0.2,4){};
		\node [label=below: $v_2$] (u2) at (0.5,4){};
		\node [label=below: $v_3$] (u3) at (1.2,4){};
		\node [label=below: $v_4$] (u4) at (1.9,4){};
		\node [label=below: $v_5$] (u5) at (2.6,4){};
		\node [label=below: $v_6$] (u6) at (3.3,4){};
		\draw (u1) edge [bend left = 40]  (u6);
		\draw (u2) edge [bend left = 40]  (u5);
		
		\node [draw=none,fill=none] at (6.5,3){$M_2$};
		\node [label=below: $v_1$] (u1) at (4.8,4){};
		\node [label=below: $v_2$] (u2) at (5.5,4){};
		\node [label=below: $v_3$] (u3) at (6.2,4){};
		\node [label=below: $v_4$] (u4) at (6.9,4){};
		\node [label=below: $v_5$] (u5) at (7.6,4){};
		\node [label=below: $v_6$] (u6) at (8.3,4){};
		\draw (u1) edge [bend left = 40]  (u6);
		\draw (u2) edge [bend left = 40]  (u5);
		\draw (u3) edge [bend left = 0]  (u4);
		
		\node [draw=none,fill=none] at (11.5,3){$M_3$};
		\node [label=below: $v_1$] (u1) at (9.8,4){};
		\node [label=below: $v_2$] (u2) at (10.5,4){};
		\node [label=below: $v_3$] (u3) at (11.2,4){};
		\node [label=below: $v_4$] (u4) at (11.9,4){};
		\node [label=below: $v_5$] (u5) at (12.6,4){};
		\node [label=below: $v_6$] (u6) at (13.3,4){};
		\draw (u1) edge [bend left = 40]  (u4);
		\draw (u2) edge [bend left = 40]  (u5);
		\draw (u3) edge [bend left = 40]  (u6);
		
		\node [draw=none,fill=none] at (16.5,3){$M_4$};
		\node [label=below: $v_1$] (u1) at (14.8,4){};
		\node [label=below: $v_2$] (u2) at (15.5,4){};
		\node [label=below: $v_3$] (u3) at (16.2,4){};
		\node [label=below: $v_4$] (u4) at (16.9,4){};
		\node [label=below: $v_5$] (u5) at (17.6,4){};
		\node [label=below: $v_6$] (u6) at (18.3,4){};
		\draw (u1) edge [bend left = 40]  (u5);
		\draw (u2) edge [bend left = 40]  (u4);
		\draw (u3) edge [bend left = 40]  (u6);
		
		\node [draw=none,fill=none] at (1.5,0){$M_5$};
		\node [label=below: $v_1$] (u1) at (0,1){};
		\node [label=below: $v_2$] (u2) at (0.8,1){};
		\node [label=below: $v_3$] (u3) at (1.6,1){};
		\node [label=below: $v_4$] (u4) at (2.4,1){};
		\node [label=below: $v_5$] (u5) at (3.2,1){};
		\draw (u1) edge [bend left = 40]  (u5);
		\draw (u2) edge [bend left = 0]  (u3);
		
		\node [draw=none,fill=none] at (6.5,0){$M_6$};
		\node [label=below: $v_1$] (u1) at (5,1){};
		\node [label=below: $v_2$] (u2) at (6,1){};
		\node [label=below: $v_3$] (u3) at (7,1){};
		\node [label=below: $v_4$] (u4) at (8,1){};
		\draw (u1) edge [bend left = 30]  (u3);
		\draw (u2) edge [bend left = 30]  (u4);
		
		\node [draw=none,fill=none] at (11.5,0){$M_7$};
		\node [label=below: $v_1$] (u1) at (10,1){};
		\node [label=below: $v_2$] (u2) at (11,1){};
		\node [label=below: $v_3$] (u3) at (12,1){};
		\node [label=below: $v_4$] (u4) at (13,1){};
		\draw (u1) edge [bend left = 30]  (u4);
		\draw (u2) edge [bend left = 30]  (u3);
		
		\node [draw=none,fill=none] at (16.5,0){$M_8$};
		\node [label=below: $v_1$] (u1) at (15,1){};
		\node [label=below: $v_2$] (u2) at (15.8,1){};
		\node [label=below: $v_3$] (u3) at (16.6,1){};
		\node [label=below: $v_4$] (u4) at (17.4,1){};
		\node [label=below: $v_5$] (u5) at (18.2,1){};
		\draw (u1) edge [bend left = 40]  (u5);
		\draw (u2) edge [bend left = 40]  (u4);
	\end{tikzpicture}
		\caption{The ordered graphs $M_i$ for $i\in [8]$.}\label{M}
	\end{figure}

	The main theorem of this section is the following,  which implies Theorem \ref{OrderedNPatleast3}, the hardness assertion in Theorem \ref{shareanend}, and the entire Theorem \ref{OrderedNPdemo}.
	\begin{theorem}\label{OrderedNP}
		If $H$ is an ordered graph such that at least one of the following holds:
		\begin{itemize}
			\item $H$ has at least three edges;
			\item $H$ has a vertex of degree at least 2 and is not isomorphic to $J_{16}(k,l)$ or $-J_{16}(k,l)$ for any $k,l$;
			\item $H$ contains $J_9$, $M_1$ or $M_5$ as induced ordered subgraph;
		\end{itemize}
		then the \textsc{Ordered Graph List-3-Coloring Problem} restricted to $(H,\varphi)$-free ordered graphs is \textsf{NP}-complete.
	\end{theorem}

	In order to show Theorem \ref{OrderedNP}, we will show the following two theorems.
	\begin{theorem} \label{OrderedJ}
		The \textsc{Ordered Graph List-3-Coloring Problem} is \textsf{NP}-complete
		when restricted to the class of $J_j$-free ordered graphs, for $j\in [15]$.
	\end{theorem}
	
	\begin{theorem} \label{OrderedM}
		The \textsc{Ordered Graph List-3-Coloring Problem} is \textsf{NP}-complete
		when restricted to the class of $M_j$-free ordered graphs, for $j\in [5]$.
	\end{theorem}
	
	Moreover, as a consequence of Theorem \ref{OrderedJ}, we also have the following result, which will be proved later:
	\begin{theorem} \label{P4&P3+P2-NP}
		Let $H$ be a graph and $\varphi:V(H)\rightarrow \mathbb{Z}$. 
		The \textsc{Ordered Graph List-3-Coloring Problem} restricted to $(H,\varphi)$-free ordered graphs is \textsf{NP}-complete if $H$ contains a copy of $P_4$ or $P_3+P_2$ as an induced subgraph.
	\end{theorem}
	
	Before we start the proof, let us first point out the following observation:
	\begin{corollary} \label{NP-neg}
		If the \textsc{Ordered Graph List-3-Coloring Problem} restricted to $H$-free ordered graphs is \textsf{NP}-complete, then the \textsc{Ordered Graph List-3-Coloring Problem} restricted to $-H$-free ordered graphs is \textsf{NP}-complete.
	\end{corollary}
	
	\begin{proof}[Proof of Theorem \ref{OrderedNP} assuming Theorems \ref{OrderedJ}, \ref{OrderedM} and \ref{P4&P3+P2-NP}]
		If $H$ is a graph with at least three edges, then $H$ contains at least one of the following induced subgraph: a cycle, a claw, $P_4$, $P_3+P_2$ or $3P_2$. NP-hardness in the first two cases follow from Theorems \ref{thm:cycle} and \ref{thm:claw}, respectively. The third and fourth case follow from Theorem \ref{P4&P3+P2-NP}. Finally, if $H$ contains $3P_2$, the result follows from the cases of $J_9$ of Theorem \ref{OrderedJ} and $M_2$, $M_3$, $M_4$, $-M_4$ and $M_5$ of Theorem \ref{OrderedM}.
		
		If $H$ has a vertex of degree at least 2 and is not isomorphic to $J_{16}(k,l)$ or $-J_{16}(k,l)$ for all $k,l$, then either $H$ has at least three edges, or $H$ contains a $P_3$ and $H$ contains $J_{10}$, $-J_{10}$, $J_{11}$, $-J_{11}$, or $J_{15}$; now the result follows from Theorem \ref{OrderedJ}.
	\end{proof}
	
	We will use two constructions to show the \textsc{Ordered Graph List-3-Coloring Problem} is \textsf{NP}-complete when restricted to the class of $J_i$-free ordered graphs, for every $i\in [15]$. Then with these proofs, we will show Theorem \ref{P4&P3+P2-NP}.
	
	In the first construction, we will use the following result.
	
	\begin{theorem}[Chleb{\'i}k and Chleb{\'i}kov{\'a} \cite{L3ColBip-NP}] \label{BipL3COl}
		The \textsc{List-3-Coloring Problem} restricted to bipartite graphs is \textsf{NP}-complete.
	\end{theorem}
	
	\begin{theorem} \label{NP-Bip}
		The \textsc{Ordered Graph List-3-Coloring Problem} restricted to $J_1$, $J_2$, $J_4$, $J_6$, $J_7$, $J_9$, $J_{12}$, $J_{13}$ or $J_{15}$-free ordered graphs is \textsf{NP}-complete.
	\end{theorem}
	\begin{proof}
		Given a bipartite graph $G$ with bipartition $(X,Y)$ and its list assignment $L$, we construct an ordered graph $(G,\tau_5)$ as follows. We enumerate the set $X=\{x_1,\ldots,x_s\}$ and $Y=\{y_1,\ldots,y_t\}$. Let $\tau_5:V(G)\rightarrow \mathbb{Z}$ be a function with $\tau_5(x_i)=i$ for $i\in [s]$ and $\tau_5(y_j)=s+j$ for $j\in [t]$. 
		
		Clearly, the ordered graph $(G,\tau_5)$ can be computed in time $\mathcal{O}(n)$, and $(G,\tau_5)$ is list-3-colorable if and only if $G$ is list-3-colorable. The ordered graph $(G,\tau_5)$ is $J_7$-free, $J_9$-free and $J_{15}$-free, as for every edge $zw\in E(G)$, without loss of generality, we have $z\in X$ and $w\in Y$. The fact that $(G,\tau_5)$ is $J_{15}$-free implies that $(G, \tau_5)$ does not contain $J_1$, $J_2$, $J_4$, $J_6$, $J_{12}$ and $J_{13}$.
	\end{proof}
	
 The second construction is a reduction from NAE3SAT. Given an instance $I$ consisting of $n$ Boolean variables and $m$ clauses, each of which contains 3 literals, the \textsc{Not-All-Equal-3-Satisfiability Problem (NAE3SAT)} is to decide whether there exists a truth assignment for every variable such that every clause contains at least one true literal and one false literal. We say $I$ is \textit{satisfiable} if it admits such an assignment. \emph{Monotone} \textsc{NAE3SAT} is a \textsc{NAE3SAT} restricted to instances with no negated literals.
	\begin{theorem}[Garey and Johnson \cite{NAE-NP}]\label{NAENP}
		Monotone \textsc{NAE3SAT} is \textsf{NP}-complete.
	\end{theorem}
	
	Notice that we will prove a stronger result: In the following theorem, we actually prove the \textsf{NP}-completeness of the \textsc{Ordered Graph 3-Coloring Problem} instead of the \textsc{Ordered Graph List-3-Coloring Problem} within specific classes of graphs.
	
	\begin{theorem} \label{NP-XMT}
		Given a monotone NAE3SAT instance $I$, there is a graph $H_1$ such that:
		\begin{enumerate}
			\item The graph $H_1$ can be computed from $I$ in time $\mathcal{O}(m+n)$, where $m$ is the number of clauses of $I$ and $n$ is the number of variables of $I$;
			\item The graph $H_1$ is 3-colorable if and only if $I$ is satisfiable;
			\item There is an injective function $\tau_1:V(H_1)\rightarrow \mathbb{Z}$ such that $(H_1, \tau_1)$ is $J_3$ and $-J_{11}$-free, and $\tau_1$ can be computed from $H_1$ in time $\mathcal{O}(m+n)$;
			\item There is an injective function $\tau_2:V(H_1)\rightarrow \mathbb{Z}$ such that $(H_1, \tau_2)$ is $J_5$ and $J_8$-free, and $\tau_2$ can be computed from $H_1$ in time $\mathcal{O}(m+n)$;
			\item There is an injective function $\tau_3:V(H_1)\rightarrow \mathbb{Z}$ such that $(H_1, \tau_3)$ is $J_{10}$ and $J_{14}$-free, and $\tau_3$ can be computed from $H_1$ in time $\mathcal{O}(m+n)$.
		\end{enumerate}
		Therefore, the \textsc{Ordered Graph 3-Coloring Problem} is \textsf{NP}-complete when restricted to the class of $J_3$, $J_5$, $J_8$, $J_{10}$, $-J_{11}$ or $J_{14}$-free ordered graphs.
	\end{theorem}
	\begin{figure}[t]
		\centering 
			\begin{tikzpicture}[scale=0.8]
		\node [label=left: $x$] (x) at (1, 5){};
		\node [draw=none, fill=none] () at (4,1) {$M$};
		\node [label=right: $m_1$] (m1) at (4, 8){};
		\node  (m2) at (4, 7){};
		\node  (mi1) at (4, 6){};
		\node [draw=none, fill=none] () at (4.1,6.3) {$m_{i_1}$};
		\node  (m3) at (4, 5){};
		\node  (mi2) at (4, 4){};
		\node [draw=none, fill=none] () at (4.1,4.3) {$m_{i_2}$};
		\node  (mi3) at (4, 3){};
		\node [draw=none, fill=none] () at (4.1,3.3) {$m_{i_3}$};
		\node [label=right: $m_n$] (mn) at (4, 2){};
		\draw (x) -- (m1) ;
		\draw (x) -- (m2) ;
		\draw (x) -- (m3) ;
		\draw (x) -- (mi1) ;
		\draw (x) -- (mi2) ;
		\draw (x) -- (mi3) ;
		\draw (x) -- (mn) ;
		
		\node [draw=none, fill=none] () at (7.5,1) {$T$};
		\node [label=right: $t_{j,1}$] (t1) at (8, 6){};
		\node [label=above: $t_{j,2}$] (t2) at (7, 5){};
		\node [label=right: $t_{j,3}$] (t3) at (8, 4){};
		\draw (t1) -- (mi1) ;
		\draw (t2) -- (mi2) ;
		\draw (t3) -- (mi3) ;
		\draw (t1) -- (t2) ;
		\draw (t2) -- (t3) ;
		\draw (t3) -- (t1) ;
	\end{tikzpicture}
		\caption{Construction of $H_1$ from Theorem \ref{NP-XMT}, with $M$ corresponding to variables and $T$ corresponding to clauses. }\label{fig-XMT}
	\end{figure}
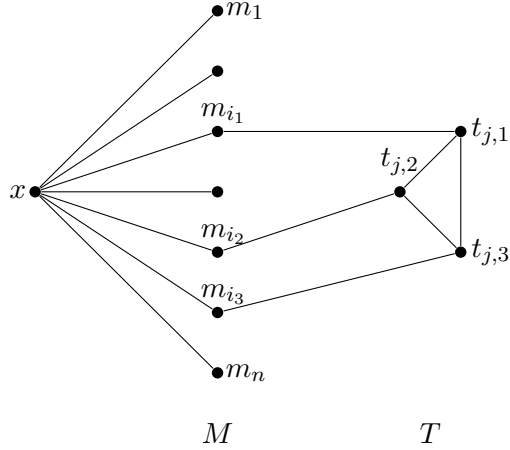
	\begin{proof}
		First we create a vertex $x$. For every variable $x_i$ of $I$, we create a vertex $m_i$, and denote the set of such vertices as $M$. For every clause $C_j$ of $I$, we create three vertices $t_{j,k}$ for $k\in [3]$, and denote the set of such vertices as $T$. Let $V(H_1)=\{x\}\cup M \cup T$.
		
		For every vertex $m_i\in M$, we add an edge $xm_i$. For every clause $C_j$ of $I$, if the variables in $C_j$ are $x_{i_1}, x_{i_2}, x_{i_3}$ with $1\leq i_1<i_2<i_3\leq n$, we add edges $m_{i_k}t_{j,k}$ for $k\in [3]$ and edges $t_{j,1}t_{j,2}$, $t_{j,2}t_{j,3}$, $t_{j,1}t_{j,3}$. Let $E(H_1)$ be the set of all defined edges. We deduced directly from the construction of $H_1$ that:
		
		\sta{ $H_1[M]$ is stable and $H_1[T]$ is disjoint union of triangles.} \label{XMT-helper}
		
		\sta{The graph $H_1$ can be computed from $I$ in time $\mathcal{O}(m+n)$.}
		
		Also, we have:
		
		\sta{The graph $H_1$ is 3-colorable if and only if $I$ is satisfiable.}
		
		Let $f:V(H_1)\rightarrow [3]$ be a 3-coloring of $H_1$. Without loss of generality, we assume $f(x)=1$. Since every vertex in $M$ is adjacent to $x$, we have $f(m_i)\in \{2,3\}$ for every $i\in [n]$. We claim that if the variables in $C_j$ are $x_{i_1}, x_{i_2}, x_{i_3}$ with $1\leq i_1<i_2<i_3\leq n$, then at least one of $f(m_{i_1})$, $f(m_{i_2})$ and $f(m_{i_3})$ has value 2 and at least one of them has value 3. Suppose for a contradiction, without loss of generality, that $f(m_{i_k})=2$ for every $k\in [3]$. Then $f(t_{j,k})\in \{1,3\}$ for every $k\in [3]$, at least two of the three vertices receive the same color. But by (\ref{XMT-helper}), the vertices $t_{j,1}$, $t_{j,2}$ and $t_{j,3}$ form a triangle. which leads to a contradiction as desired. Thus, by assigning true to the variable $x_i$ if $f(m_i)=2$ and false otherwise, we get a valid truth assignment to the monotone NAE3SAT instance $I$.
		
		If there is a valid truth assignment to $I$, we define a 3-coloring $g:V(H_1)\rightarrow [3]$ as follows. For every $i\in [n]$, let $g(m_i)=2$ if the variable $v_i$ is true in this truth assignment, otherwise let $g(m_i)=3$. Let $g(x)=1$. For each clause $C_j$, we denote the variables in $C_j$ as $x_{i_1}, x_{i_2}, x_{i_3}$ with $1\leq i_1<i_2<i_3\leq n$. Let $g(t_{j,1})\in \{2,3\}\backslash \{g(m_{i_1})\}$. Since $g|_M$ is constructed from a valid truth assignment of $I$, at least one of the $g(m_{i_1})$, $g(m_{i_2})$ and $g(m_{i_3})$ has value 2 and at least one of them has value 3. If $g(m_{i_2})\neq g(m_{i_1})$, then let $g(t_{j,2})\in \{2,3\}\backslash \{g(m_{i_2})\}$ and $g(t_{j,3})=1$, otherwise let $g(t_{j,3})\in \{2,3\}\backslash \{g(m_{i_3})\}$ and $g(t_{j,2})=1$. To verify this is a valid 3-coloring, we simply go through and check every edge in $E(H_1)$.				
		
		\tbox{There is an injective function $\tau_1:V(H_1)\rightarrow \mathbb{Z}$ such that $(H_1, \tau_1)$ is $J_3$ and $-J_{11}$-free, and $\tau_1$ can be computed from $H_1$ in time $\mathcal{O}(m+n)$.}
		
		The function $\tau_1:V(H_1)\rightarrow \mathbb{Z}$ is defined as follows. Let $\tau_1(x)=1$. Let $\tau_1(m_i)=i+1$ for every $i\in[n]$. For every $j\in[m]$ and $k\in [3]$, let $\tau_1(t_{j,k})=n+3j+k-2$. The function $\tau_1$ can be constructed in time $\mathcal{O}(m+n)$ as we go through every vertex once.
		
		From the construction we have $\tau_1|M$ and $\tau_1|T$ are injective, and  $\tau_1(x)<\tau_1(m_i)<\tau_1(t_{j,k})$ for every $i\in[n]$, $j\in[m]$ and $k\in [3]$. So the function $\tau_1$ is injective.
		
		Suppose for a contradiction that $(H_1,\tau_1)$ contains an induced path $w_1w_3w_2w_4$ with $\tau_1(w_1) < \tau_1(w_2) < \tau_1(w_3) < \tau_1(w_4)$. Since the vertex $x$ does not have any backward neighbor and every $m_i\in M$ has only one backward neighbor $x$, we have $w_3\in T$. At most one of $w_1$ and $w_2$ is in $T$ as $w_1,w_2\in N(w_3)$ and $w_1w_2\notin E(H_1)$. Thus we have $w_1\in M$, $w_2\in T\cup M$ and $w_4\in T$. If $w_2\in T$, then $w_3$ is also adjacent to $w_4$ since $w_3,w_4\in N(w_2)$, which is a contradiction. If $w_2\in M$, then $w_3$ has two neighbors in $M$, which is a contradiction. Thus, we have proved $(H_1,\tau_1)$ is $J_3$-free.	
		
		
		Suppose $(H_1,\tau_1)$ contains an induced subgraph $(\{w_1,w_2,w_3,w_4\}, \{w_1w_4,w_2w_4\})$ with $\tau_1(w_1) <\tau_1(w_2) <\tau_1(w_3) <\tau_1(w_4)$. Since $w_4$ has two backward neighbors, we have $w_4\in T$. For the two backward neighbors $w_1,w_2$ of $w_4$, since $w_1w_2\notin E(H_1)$ and $\tau_1(w_1)<\tau_1(w_2)$, we have $w_1\in M$ and $w_2\in T$. From the construction of $\tau_1$, since $w_2w_4\in E(H_1)$ and $\tau_1(w_2)<\tau_1(w_3)<\tau_1(w_4)$, we have $w_3\in T$ and $w_2w_3, w_3w_4\in E(H_1)$, which is a contradiction. Thus, we have proved $(H_1,\tau_1)$ is $-J_{11}$-free.			
		
		\tbox{There is an injective function $\tau_2:V(H_1)\rightarrow \mathbb{Z}$ such that $(H_1, \tau_2)$ is $J_5$ and $J_8$-free, and $\tau_2$ can be computed from $I$ in time $\mathcal{O}(m+n)$.}
		
		The function $\tau_2:V(H_1)\rightarrow \mathbb{Z}$ is defined as follows. Let $\tau_2(m_i)=i$ for every $i\in[n]$. Let $\tau_2(x)=n+1$. For every $j\in[m]$ and $k\in [3]$, let $\tau_2(t_{j,k})=n+3j+k-2$. The function $\tau_2$ can be constructed in time $\mathcal{O}(m+n)$ as we go through every vertex once.
		
		From the construction we have $\tau_2|M$ and $\tau_2|T$ are injective, and  $\tau_2(m_i)<\tau_2(x)<\tau_2(t_{j,k})$ for every $i\in[n]$, $j\in[m]$ and $k\in [3]$. So the function $\tau_2$ is injective.

		Suppose for a contradiction that $(H_1,\tau_2)$ contains an induced path $w_1w_4w_2w_3$ with $\tau_2(w_1) < \tau_2(w_2) < \tau_2(w_3) < \tau_2(w_4)$. Since the vertex $x$ does not have any forward neighbor, we have $w_1,w_2\neq x$. Since every vertex in $M$ has no backward neighbor, we have $w_3, w_4\notin M$. If $w_2\in T$, then we have $w_3,w_4\in T$ as $\tau_2(w_2)< \tau_2(w_3)< \tau_2(w_4)$. But from the construction of $H_1$ and $\tau_2$, we also have $w_3w_4\in E(H_1)$ as $w_2w_4\in E(H_1)$, which is a contradiction. Thus we have $w_2\in M$, which implies $w_1\in M$. Also since $w_2$ has two forward neighbors and $\tau_2(x)<\tau_2(t_{j,k})$ for every $j\in[m]$ and $k\in [3]$, we have $w_4\in T$. But then $w_4$ has two neighbors $w_1,w_2\in M$, which is a contradiction. Thus, we have proved $(H_1,\tau_2)$ is $J_5$-free.	
		
		Suppose for a contradiction that $(H_1,\tau_2)$ contains an induced path $w_3w_1w_4w_2$ with $\tau_2(w_1) < \tau_2(w_2) < \tau_2(w_3) < \tau_2(w_4)$. Since the vertex $x$ does not have any forward neighbor, we have $w_1,w_2\neq x$. Since every vertex in $M$ has no backward neighbor, we have $w_3, w_4\notin M$. Since $w_1$ has two nonadjacent forward neighbors $w_3,w_4$ with $\tau_2(w_3)<\tau_2(w_4)$, we have $w_1\in M$ and $w_4\in T$. Thus we have $w_2\in T$, as $w_4$ has exactly one neighbor in $M$. But from the construction of $H_1$ and $\tau_2$, we also have $w_3\in T$ and so $w_2w_3, w_3w_4\in E(H_1)$ as $w_2w_4\in E(H_1)$ and $\tau_2(w_2) < \tau_2(w_3) < \tau_2(w_4)$, which is a contradiction. Thus, we have proved $(H_1,\tau_2)$ is $J_8$-free.
		
		
		\tbox{There is an injective function $\tau_3:V(H_1)\rightarrow \mathbb{Z}$ such that $(H_1, \tau_3)$ is $J_{10}$ and $J_{14}$-free, and $\tau_3$ can be computed from $H_1$ in time $\mathcal{O}(m+n)$.}
		
		The function $\tau_3:V(H_1)\rightarrow \mathbb{Z}$ is defined as follows. Let $\tau_3(x)=1$. Let $\tau_1(m_i)=i+1$ for every $i\in[n]$. For every $j\in[m]$ and $k\in [3]$, let $m_i\in M$ be the vertex such that $t_{j,k}\in N(m_i)$, then we set $\tau_3(t_{j,k})=n+2+ \sum_{i'=1}^{i-1} (deg(m_{i'})-1) +|\{t_{j',k'}\in N(m_i): j'<j\}|$. The function $\tau_1$ can be constructed in time $\mathcal{O}(m+n)$ as we go through every vertex once.
		
		From the construction we have $\tau_3|M$ and $\tau_3|T$ are injective, and  $\tau_3(x)<\tau_3(m_i)<\tau_3(t_{j,k})$ for every $i\in[n]$, $j\in[m]$ and $k\in [3]$. So the function $\tau_3$ is injective.
		
		Suppose that $(H_1,\tau_3)$ contains an induced subgraph $(\{w_1,w_2,w_3,w_4\}, \{w_1w_2,w_1w_4\})$ with $\tau_3(w_1) < \tau_3(w_2) < \tau_3(w_3) < \tau_3(w_4)$. Now we consider the vertex $w_1$. Since $w_1=x$ implies $w_1w_3\in E(H_1)$ as $w_2,w_4\in N(w_1)$, and $w_1\in T$ implies $w_2w_4\in E(H_1)$, we have $w_1\in M$. But from the construction of $\tau_3$, we also have $w_1w_3\in E(H_1)$, which is a contradiction. Thus, we have proved $(H_1,\tau_2)$ is $J_{10}$-free.
		
		Suppose that $(H_1,\tau_3)$ contains an induced subgraph $(\{w_1,w_2,w_3,w_4,w_5\}, \{w_1w_5, w_2w_3, w_2w_4\})$ with $\tau_3(w_1) < \tau_3(w_2) < \tau_3(w_3) < \tau_3(w_4) < \tau_3(w_5)$. Now we consider the vertex $w_1$. If $w_1=x$, then $w_1w_2\in E(H_1)$ as $w_5\in N(w_1)$, a contradiction. If $w_1\in T$, then $w_2, w_3, w_4\in T$, which implies $w_3w_4\in E(H_1)$, again a contradiction. So we have $w_1\in M$. But from the construction of $\tau_3$, 
		if $w_2\in M$, we have $\tau_3(w_3)>\tau_3(w_5)$, which is a contradiction. If $w_2\in T$, then we have $w_3w_4\in E(H_1)$ as a contradiction. Thus, we have proved $(H_1,\tau_2)$ is $J_{14}$-free.
	\end{proof}
	\begin{remark}
		We can verify that the ordered graph $(H_1,\tau_1)$ is $J_6$-free and the ordered graph $(H_1,\tau_2)$ is $J_1$, $J_2$, $-J_4$ and $J_{12}$-free. So the \textsc{Ordered Graph 3-Coloring Problem} is also \textsf{NP}-complete when restricted to the class of $J_1$, $J_2$, $-J_4$, $J_6$ or $J_{12}$-free ordered graphs.
	\end{remark}	
	
	\begin{remark}
		The following theorem gives another proof that the \textsc{Ordered Graph 3-Coloring Problem} is also \textsf{NP}-complete when restricted to the class of $J_7$, $J_{13}$ or $J_{14}$-free ordered graphs. Therefore, the only two cases that we do not know whether  the \textsc{Ordered Graph 3-Coloring Problem} restricted to the class of $J_j$-free ordered graphs is \textsf{NP}-complete are $J_9$ and $J_{15}$.
	\end{remark}
	
	\begin{theorem} \label{NP-XMST}
		Given a monotone NAE3SAT instance $I$, there is an ordered graph $(H_2,\tau_4)$ such that
		\begin{enumerate}
			\item The ordered graph $(H_2,\tau_4)$ can be computed from $I$ in time $\mathcal{O}(m+n)$;
			\item The graph $H_2$ is 3-colorable if and only if $I$ is satisfiable;
			\item The ordered graph $(H_2,\tau_4)$ is $J_7$, $J_{13}$ and $J_{14}$-free.
		\end{enumerate}
		Therefore, the \textsc{Ordered Graph 3-Coloring Problem} is \textsf{NP}-complete when restricted to the class of $J_7$, $J_{13}$ or $J_{14}$-free ordered graphs.
	\end{theorem}
	\begin{figure}[t]
		\centering 
			\begin{tikzpicture}[scale=0.8]
		\node [label=left: $x$] (x) at (1, 5){};
		\node [draw=none, fill=none] () at (4,1) {$M$};
		\node [label=right: $m_1$] (m1) at (4, 8){};
		\node  (m2) at (4, 7){};
		\node  (mi1) at (4, 6){};
		\node [draw=none, fill=none] () at (4.1,6.3) {$m_{i_1}$};
		\node  (m3) at (4, 5){};
		\node  (mi2) at (4, 4){};
		\node [draw=none, fill=none] () at (4.1,4.3) {$m_{i_2}$};
		\node  (mi3) at (4, 3){};
		\node [draw=none, fill=none] () at (4.1,3.3) {$m_{i_3}$};
		\node [label=right: $m_n$] (mn) at (4, 2){};
		\draw (x) -- (m1) ;
		\draw (x) -- (m2) ;
		\draw (x) -- (m3) ;
		\draw (x) -- (mi1) ;
		\draw (x) -- (mi2) ;
		\draw (x) -- (mi3) ;
		\draw (x) -- (mn) ;
		
		\node [draw=none, fill=none] () at (5.5,1) {$S$};
		\node  (s11) at (5.5, 6.5){};
		\node  (s12) at (5.5, 6){};
		\node  (s13) at (5.5, 5.5){};
		\draw (s11) -- (mi1) ;
		\draw (s12) -- (mi1) ;
		\draw (s13) -- (mi1) ;
		\node [draw=none, fill=none] () at (5.5,6.8) {$s_{i_1,j}$};
		
		\node  (s21) at (5.5, 4.5){};
		\node  (s22) at (5.5, 4){};
		\draw (s21) -- (mi2) ;
		\draw (s22) -- (mi2) ;
		\node [draw=none, fill=none] () at (5.5,4.8) {$s_{i_2,j}$};
		
		\node  (s31) at (5.5, 3.5){};
		\node  (s32) at (5.5, 3){};
		\node  (s33) at (5.5, 2.5){};
		\draw (s31) -- (mi3) ;
		\draw (s32) -- (mi3) ;
		\draw (s33) -- (mi3) ;
		\node [draw=none, fill=none] () at (6,2.8) {$s_{i_3,j}$};
		
		\node [draw=none, fill=none] () at (7.5,1) {$T$};		
		\node [label=right: $t_{j,i_1}$] (t1) at (8, 6){};
		\node [label=above: $t_{j,i_2}$] (t2) at (7, 5){};
		\node [label=right: $t_{j,i_3}$] (t3) at (8, 4){};
		\draw (t1) -- (s11) ;
		\draw (t2) -- (s21) ;
		\draw (t3) -- (s32) ;
		\draw (t1) -- (t2) ;
		\draw (t2) -- (t3) ;
		\draw (t3) -- (t1) ;
	\end{tikzpicture}
		\caption{Construction of $H_2$ from Theorem \ref{NP-XMST} omitting edges between $x$ and all vertices in $S$. }\label{fig-XMST}
	\end{figure}
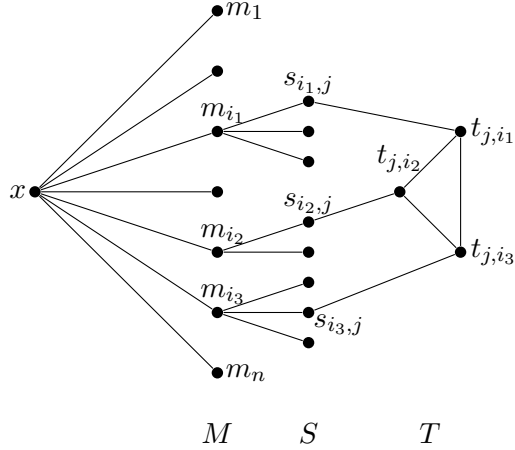
	\begin{proof}
		First we create a vertex $x$. For every variable $x_i$ of $I$, we create a vertex $m_i$ and add edge $xm_i$. We denote the set of such vertices $m_i$ as $M$. For every clause $C_j$ of $I$, if the variables in $C_j$ are $x_{i_1}, x_{i_2}, x_{i_3}$ with $1\leq i_1<i_2<i_3\leq n$, we create six vertices $t_{j,i_k}$ and $s_{i_k,j}$ for $k\in [3]$. We add edges $t_{j,i_k}t_{j,i_{k'}}$ for $\{k,k'\}\subseteq [3]$, and $xs_{i_k,j}$, $m_{i_k}s_{i_k,j}$ and $s_{i_k,j}t_{j,i_k}$ for $k\in [3]$. We denote the set of vertices $t_{j,i_k}$ as $T$, and the set of vertices $s_{i_k,j}$ as $S$. Finally, let $V(H_2)=\{x\}\cup M\cup S \cup T$ and $E(H_2)$ be the set of all edges defined above.
		
		The function $\tau_4:V(H_2)\rightarrow \mathbb{Z}$ is defined as follows. Let $\tau_4(x)=1$ and $\tau_1(m_i)=i+1$ for every $i\in[n]$. For every $s_{i,j}\in S$, let $\tau_4(s_{i,j})=n+2+ \sum_{i'=1}^{i-1} (deg(m_{i'})-1) +|\{s_{i,j'}\in N(m_i): j'<j\}|$. For every $t_{j,i}\in T$, let $\tau_4(t_{j,i})= \tau_4(s_{i,j})+ 3m$.

		\sta{The ordered graph $(H_2,\tau_4)$ can be computed from $I$ in time $\mathcal{O}(m+n)$.}
		
		It takes time $\mathcal{O}(m+n)$ to compute both the set $V(H_2)$ and $E(H_2)$. The function $\tau_4$ can be constructed in time $\mathcal{O}(m+n)$, since $|V(H_2)|=n+6m+1$ and we go through every vertex once.
		
		\tbox{The graph $H_2$ is 3-colorable if and only if $I$ is satisfiable.}
		
		Let $f:V(H_2)\rightarrow [3]$ be a 3-coloring of $H_2$. Without loss of generality, we assume $f(x)=1$. Since every vertex in $M\cup S$ is adjacent to $x$, we have $f(y)\in \{2,3\}$ for every $y\in M\cup S$. 
		
		We claim that if the variables in $C_j$ are $x_{i_1}, x_{i_2}, x_{i_3}$ with $1\leq i_1<i_2<i_3\leq n$, then at least one of $f(m_{i_1})$, $f(m_{i_2})$ and $f(m_{i_3})$ has value 2 and at least one of them has value 3. Suppose for a contradiction, without loss of generality, that $f(m_{i_k})=2$ for every $k\in [3]$. We have $f(s_{i_k,j})=3$ for every $k\in [3]$. Then $f(t_{j,i_k})\in \{1,2\}$ for every $k\in [3]$, at least two of the three vertices receive the same color. But from the construction, the vertices $t_{j,i_1}$, $t_{j,i_2}$ and $t_{j,i_3}$ form a triangle, which leads to a contradiction as desired. Thus, by assigning true to the variable $x_i$ if $f(m_i)=2$ and false otherwise, we get a valid truth assignment to the monotone NAE3SAT instance $I$.
		
		If there is a valid truth assignment to $I$, we define a 3-coloring $g:V(H_2)\rightarrow [3]$ as follows. For every $i\in [n]$, let $g(m_i)=2$ if the variable $v_i$ is true in this truth assignment, otherwise let $g(m_i)=3$. Let $g(x)=1$. For every vertex $s_{i,j}\in S$, we define $g(s_{i,j})\in \{2,3\}\backslash \{ g(m_{i})\}$.
		
		For each clause $C_j$, we denote the variables in $C_j$ as $x_{i_1}, x_{i_2}, x_{i_3}$ with $1\leq i_1<i_2<i_3\leq n$. Let $g(t_{j,i_1})\in \{2,3\}\backslash \{g(s_{i_1,j})\}$. Since $g|_M$ is constructed from a valid truth assignment of $I$, at least one of the $g(m_{i_1})$, $g(m_{i_2})$ and $g(m_{i_3})$ has value 2 and at least one of them has value 3. If $g(m_{i_2})\neq g(m_{i_1})$, then let $g(t_{j,2})\in \{2,3\}\backslash \{g(s_{i_2,j})\}$ and $g(t_{j,3})=1$, otherwise let $g(t_{j,3})\in \{2,3\}\backslash \{g(s_{i_3,j})\}$ and $g(t_{j,2})=1$. 
		
		To verify this is a valid 3-coloring, we simply go through every edge $yz$ in $E(H_1)$. If without loss of generality $y=x$ and $z\in M\cup S$, then $g(y)=1$ and $g(z)\in \{2,3\}$. So $g(y)\neq g(z)$. If $y\in M$ and $z\in S$, then from the construction $g(z)\in \{2,3\}\backslash \{ g(y)\}$, so $g(z)\neq g(y)$. If $y\in S$ and $z\in T$, then from the construction either $g(z)=1$ or $g(z)\in \{2,3\}\backslash \{ g(y)\}$, so $g(y)\neq g(z)$. If $y,z\in T$, we have $g(y)\neq g(z)$ as we use all three colors to color the vertices whose corresponding variables are in the same clause.
		
		\tbox{The ordered graph $(H_2,\tau_4)$ is $J_7$, $J_{13}$ and $J_{14}$-free.}
		
		Suppose for a contradiction that $(H_2,\tau_4)$ contains an induced path $w_2w_1w_4w_3$ with $\tau_4(w_1) < \tau_4(w_2) < \tau_4(w_3) < \tau_4(w_4)$. Now we consider the vertex $w_1$. Since $w_1$ has two nonadjacent forward neighbors, we know that $w_1\notin S\cup T$. If $w_1=x$, then $w_2,w_4\in M\cup S$. But from the construction of $\tau_4$, we also have $w_3\in M\cup S$, which implies $w_3\in S$ and so $w_1w_3\in E(H_2)$ as a contradiction. If $w_1\in M$, then $w_2,w_4\in S$, which implies $w_1w_3\in E(H_2)$ as a contradiction. Thus, we have proved $(H_2,\tau_4)$ is $J_{7}$-free.
		
		Suppose $(H_2,\tau_4)$ contains an induced subgraph $(\{w_1,w_2,w_3,w_4, w_5\}, \{w_1w_5,w_2w_3, w_3w_4\})$ with $\tau_4(w_1) < \tau_4(w_2) < \tau_4(w_3) < \tau_4(w_4) < \tau_4(w_5)$. Now we consider the vertex $w_1$. Since $w_1w_5\in E(H_2)$ and $w_1w_2\notin E(H_2)$, we have $w_1\notin \{x\}$. If $w_1\in M$, we have $w_5\in S$, which causes $w_2\in M$ and $w_3\in S$. But then $w_4$ has no place to go, which is a contradiction. If $w_1\in T$, we know that $w_2,w_3,w_4\in T$. But then from the construction of $H_2$ and $\tau_4$, we have $w_2w_4\in E(H_2)$, which is a contradiction. If $w_1\in S$, then $w_5\in T$. Since $w_2w_4\notin E(H_2)$, at least one of $w_2,w_3,w_4$ is in $S$. From the construction of $\tau_4$, we know that $w_2\in S$. So the forward neighbor $w_3$ of $w_2$ is in $T$. But from the construction of $\tau_4$, the inequality $\tau_4(w_1)<\tau_4(w_2)$ implies $\tau_4(w_5)<\tau_4(w_3)$, which is a contradiction. Thus, we have proved $(H_2,\tau_4)$ is $J_{13}$-free. 
		
		A similar argument holds for the case that $(H_2,\tau_4)$ is $J_{14}$-free.
		
	\end{proof}
	
	Theorem \ref{NP-Bip}, \ref{NP-XMT} and Corollary \ref{NP-neg} together give us Theorem \ref{OrderedJ}. Now we are ready to prove Theorem \ref{P4&P3+P2-NP}. 
	\begin{proof}[Proof of Theorem \ref{P4&P3+P2-NP}]
		Let $(H,\varphi)$ be an ordered graph. If $H$ contains a copy of $P_4$, say $Q_1$, then since $(Q_1, \varphi|_{Q_1})\in \{J_i:i\in[8]\}\cup \{-J_i:i\in[8]\}$, by Theorems \ref{NP-Bip}, \ref{NP-XMT} and Corollary \ref{NP-neg}, we have the \textsc{Ordered Graph List-3-Coloring Problem} restricted $(Q_1, \varphi|_{Q_1})$-free ordered graphs is \textsf{NP}-complete.
		
		If $H$ contains a copy of $P_3+P_2$, we denote it $Q_2=(\{v_1, v_2, v_3,v_4, v_5\}, \{v_1v_2, v_2v_3, v_4v_5\})$. By symmetry, we assume without loss of generality that $\varphi(v_4)<\varphi(v_5)$. Then we consider $\min \{\varphi(v_1), \varphi(v_2), \varphi(v_3)\}$ and $\max \{\varphi(v_1), \varphi(v_2), \varphi(v_3)\}$. If $\max \{\varphi(v_1), \varphi(v_2), \varphi(v_3)\} <\varphi(v_4)$ or $\min \{\varphi(v_1), \varphi(v_2), \varphi(v_3)\}> \varphi(v_5)$, then Theorem \ref{NP-Bip} indicates the \textsc{Ordered Graph List-3-Coloring Problem} restricted $(Q_2, \varphi|_{Q_2})$-free ordered graphs is \textsf{NP}-complete.
		
		If $\varphi(v_4)< \min \{\varphi(v_1), \varphi(v_2), \varphi(v_3)\}$ and $\max \{\varphi(v_1), \varphi(v_2), \varphi(v_3)\}< \varphi(v_5)$, then $Q_2$ either contains a copy of $J_{13}$ or $-J_{13}$ as induced subgraph, or a copy of $J_{14}$ or $-J_{14}$. By Theorems \ref{NP-Bip} and \ref{NP-XMT}, respectively, the \textsc{Ordered Graph List-3-Coloring Problem} restricted $(Q_2, \varphi|_{Q_2})$-free ordered graphs is \textsf{NP}-complete.
		
		If $\min \{\varphi(v_1), \varphi(v_2), \varphi(v_3)\}< \varphi(v_i)< \max \{\varphi(v_1), \varphi(v_2), \varphi(v_3)\}$ for some $i\in\{4,5\}$, then $Q_2$ either contains a copy of $J_{10}$ as induced subgraph, or a copy of $J_{11}$, or a copy of $J_{12}$, or $-J_{10}$, $-J_{11}$, $-J_{12}$. By Theorem \ref{NP-Bip}, \ref{NP-XMT} and Corollary \ref{NP-neg}, the \textsc{Ordered Graph List-3-Coloring Problem} restricted $(Q_2, \varphi|_{Q_2})$-free ordered graphs is \textsf{NP}-complete.
	\end{proof}
	
	Before we start proving Theorem \ref{OrderedM}, let us prove the following lemma:
	
	\begin{lemma} \label{Realization}
		Given a graph $G$, let $H$ be a graph satisfying the following conditions:
		\begin{enumerate}
			\item $V(G)\subseteq V(H)$.
			\item For every edge $uv\in E(G)$, we have $uv\notin E(H)$ and there are three vertex disjoint $uv$-paths $P_1^{uv}$, $P_2^{uv}$ and $P_3^{uv}$ in $H$ of length at least 3. Moreover, for all edges $uv, st\in E(G)$ and $i,j\in [3]$, we have $V(P_i^{uv})\cap V(P_j^{st})=\{u,v\}\cap \{s,t\}$.
			\item The correspondence between every edge $uv\in E(G)$ and its paths $P_1^{uv}$, $P_2^{uv}$ and $P_3^{uv}$ in $H$ is given.
			\item Every edge in $H$ is contained in some $P_i^{e}$, for $i\in [3]$ and $e\in E(G)$.
		\end{enumerate}
		Then there is a list assignment $L:V(H)\rightarrow 2^{[3]}$ such that:
		\begin{enumerate}
			\item The list assignment $L$ can be computed from $H$ in time $\mathcal{O}(|E(H)|)$. 
			\item For every vertex $u\in V(G)$, we have $L(u)=[3]$.
			\item For every $L$-coloring $f$ of $H$, the function $f|_{V(G)}$ is a 3-coloring of $G$.
			\item For every 3-coloring $g$ of $G$, there is a corresponding $L$-coloring $g'$ of $H$ with $g'|_{V(G)}=g$.
		\end{enumerate}
	\end{lemma}	
	Note: \begin{itemize}
		\item We say a pair $(H,L)$ as in Lemma \ref{Realization} a \textit{realization} of $G$.
		\item As the \textsc{3-Coloring Problem} is \textsf{NP}-complete, to decide whether $H$ is $L$-colorable is also \textsf{NP}-complete. 
	\end{itemize}
	\begin{figure}[t]
		\centering 
			\begin{tikzpicture}[scale=1]
		\node [label=left: $u$] (u) at (0, 1.5){};
		\node [draw=none, fill=none] () at (-0.3,2.1) {$\{1,2,3\}$};
		\node [label=right: $v$] (v) at (9, 1.5){};
		\node [draw=none, fill=none] () at (9.3,2.1) {$\{1,2,3\}$};
		
		\node [label=above: {$\{1,2\}$}] (w1) at (1.5,3){};
		\node [label=above: {$\{2,3\}$}] (w2) at (3,3){};
		\node [label=above: {$\{3,1\}$}] (w3) at (4.5,3){};
		\node [label=above: {$\{1,2\}$}] (w4) at (6,3){};
		\node [label=above: {$\{1,2\}$}] (w5) at (7.5,3){};
		\draw (u) -- (w1) ;
		\draw (w1) -- (w2) ;
		\draw (w2) -- (w3) ;
		\draw (w3) -- (w4) ;
		\draw (w4) -- (w5) ;
		\draw (w5) -- (v) ;
		
		\node [label=above: {$\{2,3\}$}] (w1) at (1.5,1.5){};
		\node [label=above: {$\{3,1\}$}] (w2) at (3,1.5){};
		\node [label=above: {$\{1,2\}$}] (w3) at (4.5,1.5){};
		\node [label=above: {$\{2,3\}$}] (w4) at (6,1.5){};
		\node [label=above: {$\{2,3\}$}] (w5) at (7.5,1.5){};
		\draw (u) -- (w1) ;
		\draw (w1) -- (w2) ;
		\draw (w2) -- (w3) ;
		\draw (w3) -- (w4) ;
		\draw (w4) -- (w5) ;
		\draw (w5) -- (v) ;
		
		\node [label=below: {$\{3,1\}$}] (w1) at (1.5,0){};
		\node [label=below: {$\{3,1\}$}] (w2) at (3.5,0){};
		\node [label=below: {$\{3,1\}$}] (w3) at (5.5,0){};
		\node [label=below: {$\{3,1\}$}] (w4) at (7.5,0){};
		\draw (u) -- (w1) ;
		\draw (w1) -- (w2) ;
		\draw (w2) -- (w3) ;
		\draw (w3) -- (w4) ;
		\draw (w4) -- (v) ;
		
		\node [label=left: $u$] (u) at (-4, 1.5){};
		\node [label=right: $v$] (v) at (-3, 1.5){};
		\draw (u) -- (v) ;
		\node [draw=none, fill=none] () at (-3.5,0.5) {$G$};
		\node [draw=none, fill=none] () at (4.5,-1.5) {$H$};
		\node [draw=none, fill=none] () at (0.8,3) {$P_1^{uv}$};
		\node [draw=none, fill=none] () at (1.2,1.1) {$P_2^{uv}$};
		\node [draw=none, fill=none] () at (0.8,0) {$P_3^{uv}$};
	\end{tikzpicture}
		\caption{An example of a realization $(H,L)$ of $G$. Each vertex is labeled with its list. }\label{fig-realization}
	\end{figure}
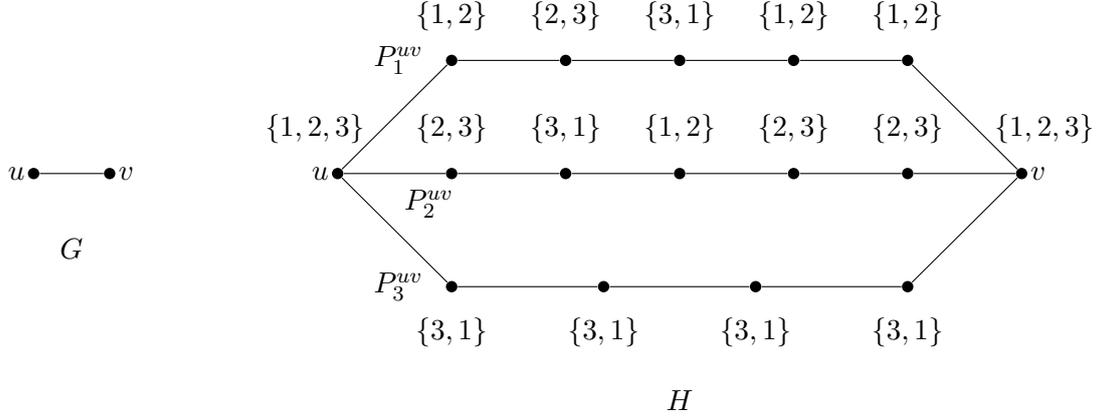
	\begin{proof}
		For every edge $uv\in  E(G)$, let the three vertex disjoint $uv$-paths be $P_1^{uv}$, $P_2^{uv}$ and $P_3^{uv}$ of length at least 4. For convenience, in this proof, we read every color modulo 3 (so if this would assign color 4, we assign color 1 instead). We define the list assignment $L:V(H)\rightarrow 2^{[3]}$ as follows. Let $L(u)=[3]$ for every vertex $u\in V(G)$. Then we take a path $P_i^{uv}$, $i\in [3]$ and denote $P_i^{uv}=uw_1w_2\ldots w_tv$. If $t$ is even, we set $L(w_j)=\{i,i+1\}$ for all $i\in [t]$. If $t$ is odd, we set $L(w_j)=\{i+j-1,i+j\}$ for $j\in [3]$, and $L(w_j)=\{i,i+1\}$ for $i\in \{4,\ldots,t\}$.

		\sta{The list assignment $L$ can be computed from $H$ in time $\mathcal{O}(|E(H)|)$, if given the correspondence between every edge $uv\in E(G)$ and its paths $P_1^{uv}$, $P_2^{uv}$ and $P_3^{uv}$ in $H$.}
		
		For a given path $P_i^e$, defining $L|_{V(P_i^e)}$ takes time $|E(P_i^e)|$. So the running time is $\sum_{e\in E(G)} \sum_{i=1}^3 |E(P_i^e) |=|E(H)|$. 
		
		\sta{For every vertex $u\in V(G)$, we have $L(u)=[3]$.}
		
		This holds immediately from the construction.
		
		\sta{For every $L$-coloring $f$ of $H$, the function $f|_{V(G)}$ is a 3-coloring of $G$.}
		
		Suppose not, then there is an $L$-coloring $f$ of $H$ such that there is an edge $uv\in E(G)$ with $f(u)=f(v)=i$ for some $i\in [3]$. We consider the path $P_i^{uv}=uw_1w_2\ldots w_tv$ in $H$. 
		If $t$ is even, then from the construction we have that $f(w_j)=i+1$ if $j$ is odd, and $f(w_j)=i$ if $j$ is even. But then we have $f(w_t)=f(v)$, which leads to a contradiction.
		If $t$ is odd, then from the construction we have $f(w_1)=i+1$, $f(w_2)=i+2$, $f(w_3)=i$, and for $j\in \{4,\ldots,t\}$, $f(w_j)=i+1$ if $j$ is even and $f(w_j)=i$ if $j$ is odd. But then $f(w_t)=f(v)=i$, which is a contradiction. Thus $f|_{V(G)}$ is a 3-coloring of $G$.
		
		\sta{For every 3-coloring $g$ of $G$, there is a corresponding $L$-coloring $g'$ of $H$ with $g'|_{V(G)}=g$.}
		
		For every $u\in V(G)$, let $g'(u)=g(u)$. 
		For every edge $uv\in V(G)$, we denote $g(u)=a$ and $g(v)=b$, $\{a,b\}\subseteq [3]$. Then we consider the path $P_i^{uv}=uw_1w_2\ldots w_tv$. We may assume that $a\in L(w_1)$ and $b\in L(w_t)$, for otherwise $P_i^{uv}$ is $L$-colorable.
		
		If $t$ is even, let $g'(w_1)\in L(w_1)\backslash \{a\}\neq \emptyset$, $g'(w_2)=a$, and $g'(w_j) =g'(w_1)$ if $j\in\{3,\ldots,t\}$ is odd, and $g'(w_j) =g'(w_2)$ if $j\in\{3,\ldots,t\}$ is even. Notice that we have $g'(w_t)=g'(w_2)=a$. Since $a\neq b$, we have $g'(v)\neq g'(w_t)$.
		
		So let us assume $t$ is odd. If $a=i$, or $a=i+2$ and $b=i+1$, then we set $g'(w_1)=i+1$, $g'(w_2)=i+2$, $g'(w_3)=i$, and $g'(w_j)=i$ if $j\in \{4,\ldots,t\}$ odd, $g'(w_j)=i+1$ if $j\in \{4,\ldots,t\}$ even. Thus $g'(w_t)=i\neq b=g'(v)$. If $a=i+1$, or $a=i+2$ and $b=i$, then we set $g'(w_1)=i$, $g'(w_2)=i+1$, $g'(w_3)=i+2$, and $g'(w_j)=i+1$ if $j\in \{4,\ldots,t\}$ is odd, $g'(w_j)=i$ if $j\in \{4,\ldots,t\}$ is even. Thus $g'(w_t)=i+1\neq b=g'(v)$.  
		
		Therefore, we have defined an $L$-coloring $g'$ of $H$ with $g'|_{V(G)}=g$.			
	\end{proof}
	
	The proof of Theorem \ref{OrderedM} is divided into three constructions, all of which use Lemma \ref{Realization} as a helper method.
	Intuitively, we want to construct a realization $H$ of $G$ replacing edges of $G$ by long paths in such a way that the interference between the paths is well-controlled. To do so, we create ``levels,'' ordering vertices in $H$ by their distance to ``original'' vertices of $G$ in $H$. The levels are ordered from left to right. All edges are either within a level of between two consecutive levels; and between two consecutive levels, the edges usually form a matching (except edges from the first to the second level; and edges that ``close'' a path).  Then in each \textquotedblleft level\textquotedblright, the ordering is defined according the function $f$ of its closest vertex in $V(G)$, either the same or the reverse of $f$. This means that the matchings between consecutive levels will either consist of edges that all pairwise cross, or all pairwise do not cross (depending on whether the ordering stays the same or is reversed). 
 Moreover, within each level, we define a way to \textquotedblleft close\textquotedblright \ some path. The key points in our constructions are: (1) whether the ordering is reversed between consecutive levels; (2) how to \textquotedblleft close the path\textquotedblright; (3) where to put the vertices used to \textquotedblleft close\textquotedblright \ paths.
	
	\begin{theorem} \label{H3}
		Given a graph $G$, there is a graph $H_3$ and two injective functions $\tau_5,\tau_6:V(G)\rightarrow \mathbb{R}$ such that:
		\begin{enumerate}
			\item There is a list assignment $L_1:V(H_3)\rightarrow 2^{[3]}$ such that the pair $(H_3,L_1)$ is a realization of $G$.
			\item The ordered graphs $(H_3,\tau_5)$ and $(H_3,\tau_6)$ can be constructed from $G$ in time $\mathcal{O}(m^2)$, where $m=|E(G)|$.
			\item The ordered graph $(H_3,\tau_5)$ is $M_1$ and $M_2$-free.
			\item The ordered graph $(H_3,\tau_6)$ is $M_3$-free.
		\end{enumerate} 
		Therefore, the \textsc{Ordered Graph List-3-Coloring Problem} is \textsf{NP}-complete when restricted to the class of $M_1$, $M_2$ or $M_3$-free ordered graphs.
	\end{theorem}
	\begin{proof}
		We denote $|V(G)|=n$ and $|E(G)|=m$. Let $f:E(G)\rightarrow [m]$ be an ordering of $E(G)$, and $g:V(G)\rightarrow [n]$ be an ordering of $V(G)$. Intuitively, in each level, we use a path with very short edges to connect two given vertices, and put the vertices used to \textquotedblleft close paths\textquotedblright \ in each level. For $\tau_5$, the ordering within each level  is the same as $f$, which creates more crossing edges. For $\tau_6$, the ordering within each level  is either the same or the reverse of $f$ depending on the parity, which creates more parallel (nested) edges. Formally, we construct the ordered graphs $(H_3,\tau_5)$ and $(H_3,\tau_6)$ as follows.
		
		For every edge $uv\in E(G)$, we create 6 vertices $w_1(u,v,j)$ and $w_1(v,u,j)$ for $j\in \{3f(uv)-2, 3f(uv)-1, 3f(uv)\}$, and add edges $uw_1(u,v,j)$ and $vw_1(v,u,j)$ for $j\in \{3f(uv)-2, 3f(uv)-1, 3f(uv)\}$. Let $W_1$ be the set of such vertices $w_1(u,v,j)$. Suppose now $W_{i-1}$ has been defined. We create a new vertex $w_i(u,v,j)$ if $w_{i-1}(u,v,j)\in W_{i-1}$ and $j\geq i$, and add an edge $w_{i-1}(u,v,j)w_{i}(u,v,j)$. Let $W_i$ be the set of such vertices $w_i(u,v,j)$. For convenience, we also denote $W_0=V(G)$.
		
		We define $\tau_5(v)=\tau_6(v)=g(v)$ for every $v\in V(G)$. 
		For every $i\in \{1,\ldots,3m\}$, we define $\tau_5 (w_i(u,v,j))=n + (\sum_{i'=1}^{i-1} |W_{i'}|) + |\{w_i(x,y,k)\in W_i: g(x)<g(u)\}|+ |\{w_i(u,y,k)\in W_i: g(y)<g(v)\}|+j+3-3f(uv)$ for every $w_i(u,v,j)\in W_i$.
		
		Let us consider the vertices $w_i(u',v',i)$ and $w_i(v',u',i)$ in $W_i$. Without loss of generality we may assume that $\tau_5(w_{i}(u',v',i))<  \tau_5(w_{i}(v',u',i))$. For every vertex $w_i(u,v,j)\in W_i$ with $\tau_5(w_{i}(u',v',i)) < \tau_5(w_{i}(u,v,j)) < \tau_5(w_{i}(v',u',i))$, we add one new vertex $z_i(u,v,j)$. Let $Z_i$ be the set of such vertices $z_i(u,v,j)$. Let $\tau_5 (z_i(u,v,j))=\tau_5 (w_i(u,v,j))+\frac{1}{2}$. For every $z_{i}(u,v,j) ,z_{i}(u^*,v^*,j^*) \in Z_i$ with $\tau_5(z_{i}(u,v,j)) < \tau_5(z_{i}(u^*,v^*,j^*))$, we add edges $z_{i}(u,v,j)z_{i}(u^*,v^*,j^*)$ if $\tau_5(z_{i}(u,v,j))=\tau_5(z_{i}(u^*,v^*,j^*))-1$, and edges $w_i(u',v',i) z_{i}(u,v,j)$ if $\tau_5(z_{i}(u,v,j))=\tau_5(w_i(u',v',i))+\frac{3}{2}$, and $w_i(v',u',i)z_{i}(u^*,v^*,j^*)$ if $\tau_5(z_{i}(u^*,v^*,j^*))=\tau_5(w_i(v',u',i))-\frac{1}{2}$.
		
		We define $\tau_6(v)$ as follows. Let $\tau_6(v)=\tau_5(v)$ for every $v\in V(G)$. For every $i\in \{1,\ldots,3m\}$ and $w_i(u,v,j)\in W_i$, we define $\tau_6 (w_i(u,v,j))= \tau_5 (w_i(u,v,j))$ if $i$ is even, and $\tau_6 (w_i(u,v,j))=n + (\sum_{i'=1}^{i-1} |W_{i'}|) + |W_i|+1 - ( |\{w_i(x,y,k)\in W_i: g(x)<g(u)\}|+ |\{w_i(u,y,k)\in W_i: g(y)<g(v)\}|+ j+3-3f(uv))$ if $i$ is odd.
		For every $i\in \{1,\ldots,3m\}$ and $z_{i}(u,v,j)\in Z_i$, let $\tau_6 (z_i(u,v,j))=\tau_6 (w_i(u,v,j))+\frac{1}{2}$ if $i$ odd, and $\tau_6 (z_i(u,v,j))=\tau_6 (w_i(u,v,j))-\frac{1}{2}$ if $i$ is even.
		
		Let $V(H_3)=V(G)\cup \bigcup_{i=1}^{3m} (W_i\cup Z_i)$ and $E(H_3)$ be the set of all edges defined above. From the construction, the functions $\tau_5$ and $\tau_6$ are orderings of $H_3$.

		We then let $P_i^{uv}$ consist of vertices 
		\[u, v, w_1(u,v,3f(uv)+i-3), w_2(u,v,3f(uv)+i-3),\ldots, w_{3f(uv)+i-3}(u,v,3f(uv)+i-3),\] \[w_1(v,u,3f(uv)+i-3), w_2(v,u,3f(uv)+i-3), \ldots, w_{3f(uv)+i-3}(v,u,3f(uv)+i-3),\] as well as all vertices in $Z_{3f(uv)+i-3}$, for $uv\in E(G)$ and $i\in [3]$. The graph $H_3$ and the paths $P_i^{uv}$ satisfy the condition of Lemma \ref{Realization}. Thus, letting $L_1$ be as in Lemma \ref{Realization}, we have:
		
		\sta{The pair $(H_3,L_1)$ is a realization of $G$.}
		
		Also, we deduce
		
		\sta{The ordered graphs $(H_3,\tau_5)$ and $(H_3,\tau_6)$ can be computed from $G$ in time $\mathcal{O}(m^2)$.}
		It takes time $\mathcal{O}(m)$ and $\mathcal{O}(n)$ to get the functions $f$ and $g$, respectively. 
		The set $W_1$ can be computed from $G$ and $f$ in time $\mathcal{O}(m)$. For $i\in \{2,\ldots,3m\}$, the set $W_i$ can be computed from $W_{i-1}$ in time $\mathcal{O}(m)$. And the function $\tau_5|_{W_i}$ can be computed in time $\mathcal{O}(m)$ for $i\in [3m]$. The set $Z_i$ and the function $\tau_5|_{Z_i}$ can be computed from $W_{i}$ and $\tau_5|_{W_i}$ in time $\mathcal{O}(m)$. Finally, the function  $\tau_6$ can be computed in time $\mathcal{O}(m^2)$. Thus, the ordered graphs $(H_3,\tau_5)$ and $(H_3,\tau_6)$ can be computed from $G$ in time $\mathcal{O}(m^2)$.
		
		\medskip
		From the construction defined above, we notice that:
		
		\sta{Given $k\in \{5,6\}$, for each edge $uv\in E(H_3)$ with $\tau_k(u)<\tau_k(v)$ and $u \in Z_i$ or $v\in Z_i$ for some $i\in \{1,\ldots,3m\}$, there is at most one vertex $w$ with $\tau_k(u)<\tau_k(w)<\tau_k(v)$.}\label{shortpath}
		
		With this observation, we are ready to prove the remaining two properties.
		
		\sta{The ordered graph $(H_3,\tau_5)$ is $M_1$ and $M_2$-free.}	
		
		Suppose $(H_3,\tau_5)$ contains a copy of $M_1$ or $M_2$ as ordered induced subgraph. We enumerate the vertices $v_i$ of $M_1$ or $M_2$ in increasing order with respect to the ordering $\tau_5$. Let us consider the vertex $v_1$. Because of the vertices $v_3,v_4$, by (\ref{shortpath}) we have $v_1,v_2,v_5,v_6\notin Z_i$ for every $i\in [3m]$. If $v_1\in V(G)$, then we have $v_6\in W_1$. So $v_2\in V(G)$ and $v_5\in W_1$. But then either $\tau_5(v_2)<\tau_5(v_1)$ or $\tau_5(v_5)>\tau_5(v_6)$, both of which cause a contradiction. If $v_1\in W_i$ for some $i\in [3m]$, then $v_6\in W_{i+1}$, and $v_2\in W_i$, $v_5\in W_{i+1}$, which is a contradiction. Thus, we have proved $(H_3,\tau_5)$ is $M_1$-free and $M_2$-free.
		
		\sta{The ordered graph $(H_3,\tau_6)$ is $M_3$-free.}
		
		Suppose $(H_3,\tau_6)$ contains a copy of $M_3$ as ordered induced subgraph. We enumerate the vertices $v_i$ of $M_3$ in increasing order with respect to the ordering $\tau_6$. Let us consider the vertex $v_1$. 
		Because of the vertices $v_2,v_3$, by (\ref{shortpath}) we have $v_1,v_4\notin Z_i$ for every $i\in [3m]$. Similarly, we have $v_t\notin Z_i$ for every $t\in [6]$ and $i\in [3m]$. Let $v_1\in W_i$ for some $i\in \{0,1,\ldots,3m\}$, then $v_4\in W_{i+1}$. Then let us consider the vertex $v_2$. For every vertex $x\in W_i$ with $\tau_6(x)>\tau_6(v_1)$ and for every vertex $y\in N(x)$ with $\tau_6(y)>\tau_6(x)$ and $y\notin Z_i$, we have $\tau_6(y)<\tau_6(v_4)$. Thus, the vertex $v_2$ is in $W_{i+1}$. But then for every vertex $x\in W_{i+1}$ with $\tau_6(v_2)<\tau_6(x)<\tau_6(v_4)$ and for every vertex $y\in N(x)$ with $\tau_6(y)>\tau_6(x)$ and $y\notin Z_i$, we have $\tau_6(y)<\tau_6(v_5)$, which means the vertex $v_3$ has nowhere to go. Thus, we have proved $(H_3,\tau_6)$ is $M_3$-free.
	\end{proof}
	
	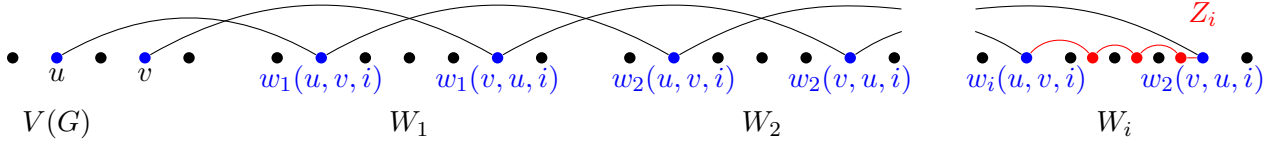
\begin{figure}[t]
		\centering 
				\begin{tikzpicture}[scale=0.58]
		
		\node  () at (3,0){};
		\node [label=below: $u$, blue] (u) at (4,0){};
		\node  () at (5,0){};
		\node [label=below: $v$, blue] (v) at (6,0){};
		\node  () at (7,0){};
		\node [draw=none, fill=none] () at (4,-1.5) {$V(G)$};
		
		\node  () at (9,0){};
		\node [blue] (u1) at (10,0){};
		\node [draw=none, fill=none, blue] () at (10,-.5) {$w_1(u,v,i)$};
		\node  () at (11,0){};
		\node  () at (12,0){};
		\node  () at (13,0){};
		\node [blue] (v1) at (14,0){};
		\node [draw=none, fill=none, blue] () at (14,-.5) {$w_1(v,u,i)$};
		\node  () at (15,0){};
		\draw (u) edge [bend left = 30]  (u1);
		\draw (v) edge [bend left = 30]  (v1);
		\node [draw=none, fill=none] () at (12,-1.5) {$W_1$};
		
		\node  () at (17,0){};
		\node [blue] (u2) at (18,0){};
		\node [draw=none, fill=none, blue] () at (18, -.5) {$w_2(u,v,i)$};
		\node  () at (19,0){};
		\node  () at (20,0){};
		\node  () at (21,0){};
		\node [blue] (v2) at (22,0){};
		\node [draw=none, fill=none, blue] () at (22,-.5) {$w_2(v,u,i)$};
		\node  () at (23,0){};
		\draw (u1) edge [bend left = 30]  (u2);
		\draw (v1) edge [bend left = 30]  (v2);
		\node [draw=none, fill=none] () at (20,-1.5) {$W_2$};
		
		\node  () at (25,0){};
		\node [blue] (ui) at (26,0){};
		\node [draw=none, fill=none, blue] () at (26,-.5) {$w_i(u,v,i)$};
		\node  () at (27,0){};
		\node [red] (z1) at (27.5,0){};
		\node  () at (28,0){};
		\node [red] (z2) at (28.5,0){};
		\node  () at (29,0){};
		\node [red] (z3) at (29.5,0){};
		\node [blue] (vi) at (30,0){};
		\node [draw=none, fill=none, blue] () at (30,-.5) {$w_2(v,u,i)$};
		\node  () at (31,0){};
		\draw [red] (ui) edge [bend left = 50]  (z1);
		\draw [red] (z1) edge [bend left = 50]  (z2);
		\draw [red] (z2) edge [bend left = 50]  (z3);
		\draw [red] (z3) edge [bend left = 0]  (vi);
		\node [draw=none, fill=none] () at (28,-1.5) {$W_i$};
		\node [red,draw=none, fill=none] () at (30,1) {$Z_i$};
		
		\node [draw=none, fill=none] (u3) at (23.3,1.1){};
		\node [draw=none, fill=none] (v3) at (23.3,0.65){};
		\node [draw=none, fill=none] (u4) at (24.7, 0.65){};
		\node [draw=none, fill=none] (v4) at (24.7, 1.1){};
		\draw (u2) edge [bend left = 18]  (u3);
		\draw (v2) edge [bend left = 7]  (v3);
		\draw (u4) edge [bend left = 7]  (ui);
		\draw (v4) edge [bend left = 18]  (vi);
		
	\end{tikzpicture}
	
		\caption{Construction of $(H_3,\tau_5)$ from Theorem \ref{H3}.}
	\end{figure}
	
	\begin{figure}[t]
		\centering 
				\begin{tikzpicture}[scale=0.58]
		
		\node  () at (3,0){};
		\node [label=below: $u$, blue] (u) at (4,0){};
		\node  () at (5,0){};
		\node [label=below: $v$, blue] (v) at (6,0){};
		\node  () at (7,0){};
		\node [draw=none, fill=none] () at (4,-1.5) {$V(G)$};
		
		\node  () at (15,0){};
		\node [blue] (u1) at (14,0){};
		\node [draw=none, fill=none, blue] () at (14,-.5) {$w_1(u,v,i)$};
		\node  () at (13,0){};
		\node  () at (12,0){};
		\node  () at (11,0){};
		\node [blue] (v1) at (10,0){};
		\node [draw=none, fill=none, blue] () at (10,-.5) {$w_1(v,u,i)$};
		\node  () at (9,0){};
		\draw (u) edge [bend left = 30]  (u1);
		\draw (v) edge [bend left = 30]  (v1);
		\node [draw=none, fill=none] () at (12,-1.5) {$W_1$};
		
		\node  () at (17,0){};
		\node [blue] (u2) at (18,0){};
		\node [draw=none, fill=none,blue] () at (18,-.5) {$w_2(u,v,i)$};
		\node  () at (19,0){};
		\node  () at (20,0){};
		\node  () at (21,0){};
		\node [blue] (v2) at (22,0){};
		\node [draw=none, fill=none, blue] () at (22,-.5) {$w_2(v,u,i)$};
		\node  () at (23,0){};
		\draw (u1) edge [bend left = 30]  (u2);
		\draw (v1) edge [bend left = 30]  (v2);
		\node [draw=none, fill=none] () at (20,-1.5) {$W_2$};
		
		\node  () at (31,0){};
		\node [blue] (ui) at (30,0){};
		\node [draw=none, fill=none, blue] () at (30,-.5) {$w_i(u,v,i)$};
		\node  () at (29,0){};
		\node [red] (z1) at (28.5,0){};
		\node  () at (28,0){};
		\node [red] (z2) at (27.5,0){};
		\node  () at (27,0){};
		\node [red] (z3) at (26.5,0){};
		\node [blue] (vi) at (26,0){};
		\node [draw=none, fill=none, blue] () at (26,-0.5) {$w_i(v,u,i)$};
		\node  () at (25,0){};
		\draw [red] (ui) edge [bend left = -35]  (z1);
		\draw [red] (z1) edge [bend left = -50]  (z2);
		\draw [red] (z2) edge [bend left = -50]  (z3);
		\draw [red] (z3) edge [bend left = 0]  (vi);
		\node [draw=none, fill=none] () at (28,-1.5) {$W_i$, $i$ odd};
		\node [red,draw=none, fill=none] () at (30.5,1) {$Z_i$};
		
		\node [draw=none, fill=none] (u3) at (23.3,2){};
		\node [draw=none, fill=none] (v3) at (23.3,0.6){};
		\node [draw=none, fill=none] (u4) at (24.7, 2){};
		\node [draw=none, fill=none] (v4) at (24.7, 0.6){};
		\draw (u2) edge [bend left = 12]  (u3);
		\draw (v2) edge [bend left = 12]  (v3);
		\draw (u4) edge [bend left = 12]  (ui);
		\draw (v4) edge [bend left = 12]  (vi);
	\end{tikzpicture}
	
		\caption{Construction of $(H_3,\tau_6)$ from Theorem \ref{H3}.}
	\end{figure}
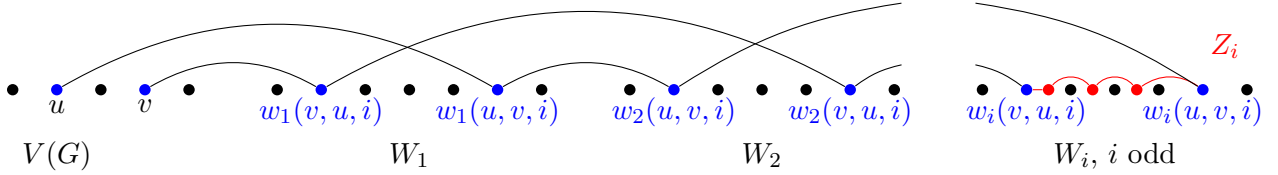
	
	\begin{theorem} \label{H4}
		Given a graph $G$, there is an ordered graph $(H_4,\tau_7)$ and a list assignment $L_2:V(H_4)\rightarrow 2^{[3]}$ such that:
		\begin{enumerate}
			\item The pair $(H_4,L_2)$ is a realization of $G$.
			\item The ordered graph $(H_4,\tau_7)$ can be constructed from $G$ in time $\mathcal{O}(m^2)$.
			\item The ordered graph $(H_4,\tau_7)$ is $M_4$-free.
		\end{enumerate} 
		Therefore, the \textsc{Ordered Graph List-3-Coloring Problem} is \textsf{NP}-complete when restricted to the class of $M_4$-free ordered graphs.
	\end{theorem}
	\begin{proof}
		We denote $|V(G)|=n$ and $|E(G)|=m$. Let $f:E(G)\rightarrow [m]$ be an ordering of $E(G)$, and $g:V(G)\rightarrow [n]$ be an ordering of $V(G)$. In this construction, the ordering within each level is the same as $f$. We \textquotedblleft close the path\textquotedblright \ directly using one extra vertex. The key idea is, the vertices used to \textquotedblleft close paths\textquotedblright \ are put at last. Formally, we construct the ordered graph $(H_4,\tau_7)$ as follows. 
		
		For every edge $uv\in E(G)$, we create 6 vertices $w_1(u,v,j)$ and $w_1(v,u,j)$ for $j\in \{3f(uv)-2, 3f(uv)-1, 3f(uv)\}$, and add edges $uw_1(u,v,j)$ and $vw_1(v,u,j)$ for $j\in \{3f(uv)-2, 3f(uv)-1, 3f(uv)\}$. Let $W_1$ be the set of such vertices $w_1(u,v,j)$. Suppose now $W_{i-1}$ has been defined. We create a new vertex $w_i(u,v,j)$ if $w_{i-1}(u,v,j)\in W_{i-1}$ and $j\geq i$, and add an edge $w_{i-1}(u,v,j)w_{i}(u,v,j)$. Let $W_i$ be the set of such vertices $w_i(u,v,j)$. For convenience, we also denote $W_0=V(G)$.
		
		We define $\tau_7(v)=g(v)$ for every $v\in V(G)$. 
		For every $i\in \{1,\ldots,3m\}$, we define $\tau_7 (w_i(u,v,j))=n + \sum_{i'=1}^{i-1} |W_{i'}| + |\{w_i(x,y,k)\in W_i: g(x)<g(u)\}|+ |\{w_i(u,y,k)\in W_i: g(y)<g(v)\}|+j+3-3f(uv)$ for every $w_i(u,v,j)\in W_i$.
		
		Let us consider the vertices $w_i(u',v',i)$ and $w_i(v',u',i)$ in $W_i$. For every $i\in \{1,\ldots,3m\}$, we add one new vertex $z_i$ and edges $w_i(u',v',i) z_i$ and $w_i(v',u',i)z_i$, and let $\tau_7 (z_i)=n+3m(3m+1)+(3m-i+1)$.
		
		Let $V(H_4)=V(G)\cup (\bigcup_{i=1}^{3m} W_i\cup \{z_i\})$ and $E(H_4)$ be the set of all edges defined above. From the construction, the function $\tau_7:V(H_4)\rightarrow \mathbb{R}$ is an ordering of $H_4$.

		We then let $P_i^{uv}$ consist of vertices 
		\[u, v, w_1(u,v,3f(uv)+i-3), w_2(u,v,3f(uv)+i-3), \ldots, w_{3f(uv)+i-3}(u,v,3f(uv)+i-3),\]
		\[w_1(v,u,3f(uv)+i-3), w_2(v,u,3f(uv)+i-3), \ldots, w_{3f(uv)+i-3}(v,u,3f(uv)+i-3),\] and $z_{3f(uv)+i-3}$, for every $uv\in E(G)$ and $i\in [3]$. The graph $H_4$ and the paths $P_i^{uv}$ satisfy the condition of Lemma \ref{Realization}. Thus, letting $L_2$ be as in Lemma \ref{Realization}, we have:
		
		\sta{The pair $(H_4,L_2)$ is a realization of $G$.}
		
		Also, we deduce: 
		
		\tbox{The ordered graph $(H_4,\tau_7)$ can be computed from $G$ in time $\mathcal{O}(m^2)$.}
		
		It takes time $\mathcal{O}(m)$ and $\mathcal{O}(n)$ to get the functions $f$ and $g$, respectively. 
		The set $W_1$ can be computed from $G$ and $f$ in time $\mathcal{O}(m)$. For $i\in \{2,\ldots,3m\}$, the set $W_i$ can be computed from $W_{i-1}$ in time $\mathcal{O}(m)$. And the function $\tau_7|_{W_i}$ can be computed in time $\mathcal{O}(m)$ for $i\in [3m]$. Thus, the ordered graph $(H_4,\tau_7)$ can be computed from $G$ in time $\mathcal{O}(m^2)$.
		
		\tbox{The ordered graph $(H_2,\tau_7)$ is $M_4$-free.}	
		
		Suppose that $(H_4,\tau_7)$ contains a copy of $M_4$ as ordered induced subgraph. We enumerate the vertices $v_i$ of $M_4$ in increasing order with respect to the ordering $\tau_7$. Because of the edges of $M_4$, we have $v_1,v_2,v_3 \notin\{z_1,\ldots,z_{3m}\}$, and $v_1,v_2\notin W_{3m}$. Let $v_1\in W_i$ for some $i\in \{0,\ldots,3m-1\}$. If $v_5\in W_{i+1}$, then we have $v_4 \notin\{z_1,\ldots,z_{3m}\}$. For every vertex $x$ with $\tau_7(v_1)< \tau_7(x)< \tau_7(v_5)$ and for every vertex $y\in N(x)\backslash \{z_1,\ldots,z_{3m}\}$ with $\tau_7(y)>\tau_7(x)$, we have $\tau_7(y)> \tau_7(v_5)$. Thus, we conclude that $v_5\in \{z_1,\ldots,z_{3m}\}$. But then for every vertex $x\in\{z_1,\ldots,z_{3m}\}$ with $\tau_7(x)> \tau_7(v_5)$ and for every vertex $y\in N(x)$, we have $\tau_7(y)<\tau_7(v_1)$, which is a contradiction. Thus, we have proved $(H_4,\tau_7)$ is $M_4$-free.
	\end{proof}
	\begin{figure}[t]
		\centering 
				\begin{tikzpicture}[scale=0.58]
		
		\node  () at (3,0){};
		\node [label=below: $u$, blue] (u) at (4,0){};
		\node  () at (5,0){};
		\node [label=below: $v$, blue] (v) at (6,0){};
		\node  () at (7,0){};
		\node [draw=none, fill=none] () at (4,-1.5) {$V(G)$};
		
		\node  () at (9,0){};
		\node [blue] (u1) at (10,0){};
		\node [blue] [draw=none, fill=none] () at (10,-0.5) {$w_1(u,v,i)$};
		\node  () at (11,0){};
		\node  () at (12,0){};
		\node  () at (13,0){};
		\node [blue] (v1) at (14,0){};
		\node [blue] [draw=none, fill=none] () at (14,-0.5) {$w_1(v,u,i)$};
		\node  () at (15,0){};
		\draw (u) edge [bend left = 30]  (u1);
		\draw (v) edge [bend left = 30]  (v1);
		\node [draw=none, fill=none] () at (12,-1.5) {$W_1$};
		
		\node  () at (17,0){};
		\node [blue] (ui) at (18,0){};
		\node [blue] [draw=none, fill=none] () at (18,-0.5) {$w_i(u,v,i)$};
		\node  () at (19,0){};
		\node  () at (20,0){};
		\node  () at (21,0){};
		\node [blue] (vi) at (22,0){};
		\node [blue] [draw=none, fill=none] () at (22,-0.5) {$w_i(v,u,i)$};
		\node  () at (23,0){};
		\node [draw=none, fill=none] () at (20,-1.5) {$W_i$};
		
		\node  () at (25,0){};
		\node  () at (26,0){};
		\node  () at (27,0){};
		\node [red] (z) at (28,0){};
		\node [red] [draw=none, fill=none] () at (28,-0.5) {$z_i$};
		\node  () at (29,0){};
		\draw [red] (ui) edge [bend left = 30]  (z);
		\draw [red] (vi) edge [bend left = 30]  (z);
		\node [draw=none, fill=none, red] () at (27,-1.5) {$Z$};
		
		\node [draw=none, fill=none] (u3) at (15.3,1.1){};
		\node [draw=none, fill=none] (v3) at (15.3,0.65){};
		\node [draw=none, fill=none] (u4) at (16.7, 0.65){};
		\node [draw=none, fill=none] (v4) at (16.7, 1.1){};
		\draw (u1) edge [bend left = 18]  (u3);
		\draw (v1) edge [bend left = 7]  (v3);
		\draw (u4) edge [bend left = 7]  (ui);
		\draw (v4) edge [bend left = 18]  (vi);
		
	\end{tikzpicture}
				
		\caption{Construction of $(H_4,\tau_7)$ from Theorem \ref{H4}.}
	\end{figure}
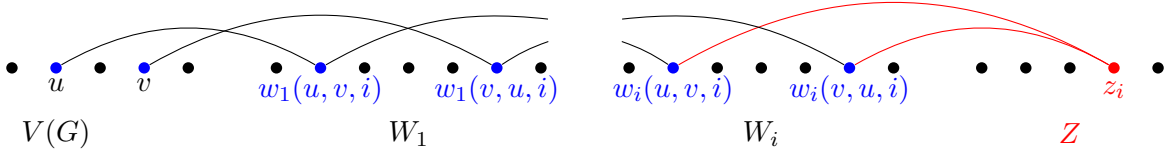
	
	\begin{theorem} \label{H5}
		Given a graph $G$, there is an ordered graph $(H_5,\tau_8)$ and a list assignment $L_3:V(H_5)\rightarrow 2^{[3]}$ such that:
		\begin{enumerate}
			\item The pair $(H_5,L_3)$ is a realization of $G$.
			\item The ordered graph $(H_5,\tau_8)$ can be constructed from $G$ in time $\mathcal{O}(m^3)$.
			\item The ordered graph $(H_5,\tau_8)$ is $M_5$-free.
		\end{enumerate} 
		Therefore, the \textsc{Ordered Graph List-3-Coloring Problem} is \textsf{NP}-complete when restricted to the class of $M_5$-free ordered graphs.
	\end{theorem}
	\begin{proof}
		We denote $|V(G)|=m$ and $|E(G)|=n$. Let $f:E(G)\rightarrow [m]$ be an ordering of $E(G)$, and $g:V(G)\rightarrow [n]$ be an ordering of $V(G)$. 
		The key point in this construction is the way we \textquotedblleft close path.\textquotedblright \ Instead of adding new vertices, one vertex  \textquotedblleft switches\textquotedblright \ (in the ordering) with its previous vertex until the two given vertices are consecutive. Formally, we construct the ordered graph $(H_5,\tau_8)$ as follows.
		
		For every edge $uv\in E(G)$, we create 6 vertices $w_1(u,v,j)$ and $w_1(v,u,j)$ for $j\in \{3f(uv)-2, 3f(uv)-1, 3f(uv)\}$, and add edges $uw_1(u,v,j)$ and $vw_1(v,u,j)$ for $j\in \{3f(uv)-2, 3f(uv)-1, 3f(uv)\}$. Let $W_1$ be the set of such vertices $w_1(u,v,j)$. Suppose now $W_{i-1}$ has been defined. We create a new vertex $w_i(u,v,j)$ if $w_{i-1}(u,v,j)\in W_{i-1}$ and $j\geq i$. Let $W_i$ be the set of such vertices $w_i(u,v,j)$. For convenience, we also denote $W_0=V(G)$.
		
		We define $\tau_8(v)=g(v)$ for every $v\in V(G)$. 
		For every $i\in \{1,\ldots,3m\}$, we define $\tau_8 (w_i(u,v,j))= 36m^2(i-1)+ n + \sum_{i'=1}^{i-1} |W_{i'}| + |\{w_i(x,y,k)\in W_i: g(x)<g(u)\}|+ |\{w_i(u,y,k)\in W_i: g(y)<g(v)\}|+ j+3-3f(uv)$ for every $w_i(u,v,j)\in W_i$.
		
		Let us consider the vertices $w_i(u',v',i)$ and $w_i(v',u',i)$ in $W_i$. Without loss of generality we may assume that $\tau_8(w_{i}(u',v',i))=\tau_8(w_{i}(v',u',i))-a_i$ for some $a_i\in \mathbb{N}$. We create a sequence of sets $X_k^i=\{x_k^i(u,v,j): w_i(u,v,j)\in W_i \}$, for $k\in [a_i-1]$. For convenience, we also denote $X_{a_i}^i=W_{i+1}=X_0^{i+1}$. For every $i\in [3m]$, we add edges $x^i_k(u,v,j)x^i_{k+1}(u,v,j)$ for every $k\in \{0,\ldots,a_i-1\}$ and $x^i_{k+1}(u,v,j)\in X_{k+1}^i$, and edge $x^i_{a_i-1}(u',v',i)x^i_{a_i-1}(v',u',i)$. 
		
		For every $k\in [a_i-1]$, let $\tau_8(x^i_k(v',u',i))=\tau_8(w_i(v',u',i))+6mk-k-\frac{1}{2}$, and $\tau_8(x^i_k(u,v,j))=\tau_8(w_i(u,v,j))+6mk$  otherwise. Notice that this implies $\tau_8(x_{k+1}^i(v',u',i)) - \tau_8(x_k^i(v',u',i)) = \tau_8(x_{k+1}^i(u,v,j)) - \tau_8(x_k^i(u,v,j)) -1$ for every $k\in \{0,\ldots,a_i-1\}$ and $\{u,v\}\neq \{u',v'\}$ and $x_{k+1}^i(u,v,j)\in X_{k+1}^i$. 
		
		Let $V(H_5)=V(G)\cup (\bigcup_{i=1}^{3m} W_i\cup (\bigcup_{k=1}^{a_i-1} X_k^i))$ and $E(H_5)$ be the set of all edges defined above. From the construction, the function $\tau_8:V(H_5)\rightarrow \mathbb{R}$ is an ordering of $H_5$.
		
		For $uv\in E(G)$ and $i\in [3]$, we then let $P_i^{uv}$ consist of vertices 
		\[u,v, w_1(u,v,3f(uv)+i-3), w_2(u,v,3f(uv)+i-3), \ldots, w_{3f(uv)+i-3}(u,v,3f(uv)+i-3),\]
		\[w_1(v,u,3f(uv)+i-3), w_2(v,u,3f(uv)+i-3), \ldots, w_{3f(uv)+i-3}(v,u,3f(uv)+i-3),\] 
		as well as all vertices  $x_{k}^j(u,v,3f(uv)+i-3)$ and $x_{k}^j(v,u,3f(uv)+i-3)$ for $j\in [3f(uv)+i-3]$ and $k\in [a_{j}-1]$. The graph $H_5$ and the paths $P_i^{uv}$ satisfy the conditions of Lemma \ref{Realization}. Thus, letting $L_3$ be as in Lemma \ref{Realization}, we have:

		\tbox{The pair $(H_5,L_3)$ is a realization of $G$.}
		
		Also, we deduce:
		
		\tbox{The ordered graph $(H_5,\tau_8)$ can be computed from $G$ in time $\mathcal{O}(m^3)$.}
		
		The set $W_1$ can be computed from $G$ and $f$ in time $\mathcal{O}(m)$. For $i\in \{2,\ldots,3m\}$, the set $W_i$ can be computed from $W_{i-1}$ in time $\mathcal{O}(m)$. And the function $\tau_8|_{W_i}$ can be computed in time $\mathcal{O}(m)$ for $i\in [3m]$. The set $X_k^i$ and the function $\tau_8|_{X_k^i}$ can be computed from $W_{i}$ and $\tau_8|_{W_i}$ in time $\mathcal{O}(m)$ for $k\in [a_i-1]$. And $a_i\leq |W_i|=\mathcal{O}(m)$, $|X_k^i|=|W_i|$ for $k\in [a_i-1]$. Thus, the ordered graph $(H_5,\tau_8)$ can be computed from $G$ in time $\mathcal{O}(m^3)$.
		
		\tbox{The ordered graph $(H_5,\tau_8)$ is $M_5$-free.}	
		
		Suppose $(H_5,\tau_8)$ contains a copy of $M_5$ as ordered induced subgraph. We enumerate the vertices $v_i$ of $M_5$ in increasing order with respect to the ordering $\tau_8$. Let us consider the edges $v_1v_5$ and $v_2v_3$. We claim that $v_2=x_k^i(v',u',i)$ and $v_3=x_{k+1}^i(v',u',i)$ for some $i\in [3m]$ and $k\in [a_i-1]$, and $u'$ and $v'$ being the vertices in $V(H_5)$ such that $g(u')<g(v')$ and $f(u'v')=\lceil \frac{i}{3}\rceil$. This is because otherwise from the construction of $\tau_8$, the condition $\tau_8(v_1)< \tau_8(v_2)$ implies $\tau_8(y)<\tau_8(z)$ for every $y\in N^+(v_1)$ and $z\in N^+(v_2)$.
		
		Since $v_1v_5$ is an edge, we have $v_1\in X_k^i$ and $v_5\in X_{k+1}^i$. Moreover, from the construction of $\tau_8$, for every two distinct $x, x'\in X_{k+1}^i\backslash \{x_k^i(v',u',i)\}$, we have $|\tau_8(x)-\tau_8(x')|\geq 1$. Since $\tau_8(x_{k+1}^i(v',u',i)) - \tau_8(x_k^i(v',u',i)) = \tau_8(x_{k+1}^i(u,v,i)) - \tau_8(x_k^i(u,v,i)) -1$, there is no vertex in $X_{k+1}^i$ which could be $v_4$, a contradiction. Thus, we have proved $(H_5,\tau_8)$ is $M_5$-free.
		
	\end{proof}	
	\begin{figure}[t]
		\centering 
						\begin{tikzpicture}[scale=0.58]
		\node [blue] (u1) at (10,0){};
		\node [draw=none, fill=none, blue] () at (10,-.6) {$x_{k}^i(u,v,i)$};
		\node  (a2) at (11,0){};
		\node  (a3) at (12,0){};
		\node  (a4) at (13,0){};
		\node [red] (v1) at (14,0){};
		\node [draw=none, fill=none, red] () at (14,-.6) {$x_{k}^i(v,u,i)$};
		\node  (a5) at (15,0){};
		\node [draw=none, fill=none] () at (12,-1.5) {$X_k^i$};
		
		\node [blue] (u2) at (18,0){};
		\node [draw=none, fill=none, blue] () at (18, -.6) {$x_{k+1}^i(u,v,i)$};
		\node  (b2) at (19,0){};
		\node  (b3) at (20,0){};
		\node  (b4) at (21,0){};
		\node [red] (v2) at (20.5,0){};
		\node [draw=none, fill=none, red] () at (22,-.6) {$x_{k+1}^i(v,u,i)$};
		\node  (b5) at (23,0){};
		\draw [gray](u1) edge [bend left = 40]  (u2);
		\draw [red] (v1) edge [bend left = 30]  (v2);
		\node [draw=none, fill=none] () at (20,-1.5) {$X_{k+1}^i$};
		
		\draw [gray](a2) edge [bend left = 40]  (b2);
		\draw [gray](a3) edge [bend left = 40]  (b3);
		\draw [green](a4) edge [bend left = 40]  (b4);
		\draw [gray](a5) edge [bend left = 40]  (b5);
		
		\node [blue] (u3) at (26,0){};
		\node [draw=none, fill=none, blue] () at (26,-1.2) {$x_{k+2}^i(v,u,i)$};
		\node  (c2) at (27,0){};
		\node [red] (v3) at (27.5,0){};
		\node [draw=none, fill=none, red] () at (29.5,-.6) {$x_{k+2}^i(v,u,i)$};
		\node  (c3) at (28,0){};
		\node  (c4) at (29,0){};
		\node  (c5) at (31,0){};
		\draw [gray](u2) edge [bend left = 40]  (u3);
		\draw [red] (v2) edge [bend left = 30]  (v3);
		\node [draw=none, fill=none] () at (27,-2) {$X_{k+2}^i$};
		
		\draw [gray] (b2) edge [bend left = 40]  (c2);
		\draw [green](b3) edge [bend left = 40]  (c3);
		\draw [gray](b4) edge [bend left = 40]  (c4);
		\draw [gray](b5) edge [bend left = 40]  (c5);
		
	\end{tikzpicture}
	
		\caption{Construction of $X_k^i$, $X_{k+1}^i$ and $X_{k+2}^i$ in $(H_5,\tau_8)$ from Theorem \ref{H5}.}
	\end{figure}
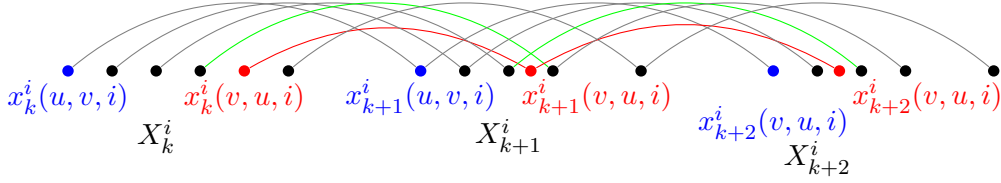

	Thus, combining Theorem \ref{H3}, \ref{H4} and \ref{H5}, we have proved Theorem \ref{OrderedM}.

	\bibliographystyle{abbrv}
	\bibliography{reference}

\begin{thebibliography}{10}

\bibitem{diGS}
P.~Aboulker, P.~Charbit, and R.~Naserasr.
\newblock Extension of {G}y{\'{a}}rf{\'{a}}s-{S}umner conjecture to digraphs.
\newblock arXiv.2009.13319, 2020.

\bibitem{OrderIdea4}
M.~Axenovich, J.~Rollin, and T.~Ueckerdt.
\newblock Chromatic number of ordered graphs with forbidden ordered subgraphs.
\newblock {\em Combinatorica}, 38:1021--1043, 2018.

\bibitem{3P7}
F.~Bonomo, M.~Chudnovsky, P.~Maceli, O.~Schaudt, M.~Stein, and M.~Zhong.
\newblock Three-coloring and list three-coloring of graphs without induced
  paths on seven vertices.
\newblock {\em Combinatorica}, 38(4):779--801, 2018.

\bibitem{L3ColBip-NP}
M.~Chleb{\'i}k and J.~Chleb{\'i}kov{\'a}.
\newblock Hard coloring problems in low degree planar bipartite graphs.
\newblock {\em Discret. Appl. Math.}, 154:1960--1965, 2006.

\bibitem{L3P6+rP3}
M.~Chudnovsky, S.~Huang, S.~Spirkl, and M.~Zhong.
\newblock List 3-coloring graphs with no induced {$P_6+rP_3$}.
\newblock {\em Algorithmica}, 83:216--251, 01 2021.

\bibitem{DBLP:journals/corr/abs-1207-0016}
M.~Chudnovsky, A.~Scott, P.~Seymour, and S.~Spirkl.
\newblock {E}rd{\H{o}}s-{H}ajnal for graphs with no 5-hole.
\newblock arXiv.2102.04994, 2021.

\bibitem{LkP5+rP1&L5P4+P2-NP}
J.-F. Couturier, P.~Golovach, D.~Kratsch, and D.~Paulusma.
\newblock List coloring in the absence of a linear forest.
\newblock {\em Algorithmica}, 71(1):21--35, January 2015.

\bibitem{ChordalLCol}
G.~A. Dirac.
\newblock On rigid circuit graphs.
\newblock {\em Abhandlungen aus dem Mathematischen Seminar der Universit{\"a}t
  Hamburg}, 25:71--76, 1961.

\bibitem{2sat1}
K.~Edwards.
\newblock The complexity of colouring problems on dense graphs.
\newblock {\em Theoret. Comput. Sci.}, 43(2-3):337--343, 1986.

\bibitem{emenem}
T.~Emden-Weinert, S.~Hougardy, and B.~Kreuter.
\newblock Uniquely colourable graphs and the hardness of colouring graphs of
  large girth.
\newblock {\em Combin. Probab. Comput.}, 7(4):375--386, 1998.

\bibitem{2sat2}
P.~Erd\H{o}s, A.~L. Rubin, and H.~Taylor.
\newblock Choosability in graphs.
\newblock In {\em Proceedings of the {W}est {C}oast {C}onference on
  {C}ombinatorics, {G}raph {T}heory and {C}omputing ({H}umboldt {S}tate
  {U}niv., {A}rcata, {C}alif., 1979)}, Congress. Numer., XXVI, pages 125--157.
  Utilitas Math., Winnipeg, Man., 1980.

\bibitem{ErdSze}
P.~Erd\"{o}s and G.~Szekeres.
\newblock A combinatorial problem in geometry.
\newblock {\em Compositio Math.}, 2:463--470, 1935.

\bibitem{ChordalLCol1}
M.~R. Fellows, F.~V. Fomin, D.~Lokshtanov, F.~Rosamond, S.~Saurabh, S.~Szeider,
  and C.~Thomassen.
\newblock On the complexity of some colorful problems parameterized by
  treewidth.
\newblock {\em Information and Computation}, 209(2):143--153, 2011.

\bibitem{NAE-NP}
M.~R. Garey and D.~S. Johnson.
\newblock {\em Computers and Intractability: A Guide to the Theory of
  NP-Completeness}.
\newblock W. H. Freeman and Company, USA, 1979.

\bibitem{hajebi:2021}
S.~Hajebi, Y.~Li, and S.~Spirkl.
\newblock Complexity dichotomy for {L}ist-5-{C}oloring with a forbidden induced
  subgraph.
\newblock arXiv.2105.01787, 2021.

\bibitem{kP5}
C.~T. Ho{\`a}ng, M.~Kami{\'n}ski, V.~Lozin, J.~Sawada, and X.~Shu.
\newblock Deciding $k$-colorability of {$P_5$}-free graphs in polynomial time.
\newblock {\em Algorithmica}, 57(1):74--81, 2010.

\bibitem{ClawFree}
I.~Holyer.
\newblock The {{NP}}-completeness of edge-coloring.
\newblock {\em SIAM J. Comput.}, 10:718--720, 1981.

\bibitem{CycleFree}
M.~Kami\'{n}ski and V.~Lozin.
\newblock Coloring edges and vertices of graphs without short or long cycles.
\newblock {\em Contributions Discret. Math.}, 2, 2007.

\bibitem{karp}
R.~M. Karp.
\newblock Reducibility among combinatorial problems.
\newblock In {\em Complexity of computer computations}, pages 85--103.
  Springer, 1972.

\bibitem{leven}
D.~Leven and Z.~Galil.
\newblock N{P} completeness of finding the chromatic index of regular graphs.
\newblock {\em J. Algorithms}, 4(1):35--44, 1983.

\bibitem{OrderIdea2}
J.~Ne\v{s}et\v{r}il.
\newblock On ordered graphs and graph orderings.
\newblock {\em Discrete Appl. Math.}, 51(1–2):113–116, jun 1994.

\bibitem{OrderIdea3}
J.~Pach and G.~Tardos.
\newblock Forbidden paths and cycles in ordered graphs and matrices.
\newblock {\em Israel Journal of Mathematics}, 155:359--380, 2006.

\bibitem{tomon}
J.~Pach and I.~Tomon.
\newblock Erd{\H{o}}s-hajnal-type results for ordered paths.
\newblock arXiv:2004.04594, 2020.

\bibitem{OrderIdea1}
V.~Rodl and P.~Winkler.
\newblock A ramsey-type theorem for orderings of a graph.
\newblock {\em SIAM Journal on Discrete Mathematics}, 2(3):402--5, 08 1989.

\bibitem{chisurvey}
A.~Scott and P.~Seymour.
\newblock A survey of {$\chi$}-boundedness.
\newblock {\em J. Graph Theory}, 95(3):473--504, 2020.

\bibitem{orderedtree}
A.~Scott, P.~Seymour, and S.~Spirkl.
\newblock Pure pairs {VI}: {E}xcluding an ordered tree.
\newblock {\em SIAM J. Discrete Math.}, 36(1):170--187, 2022.

\bibitem{2sat3}
V.~G. Vizing.
\newblock Coloring the vertices of a graph in prescribed colors.
\newblock {\em Diskret. Analiz}, (29, Metody Diskret. Anal. v Teorii Kodov i
  Shem):3--10, 101, 1976.

\end{thebibliography}
\end{document}